\documentclass[english,british,english, toc=bibliography, toc=index, hidelinks, DIV=12,abstract=true]{scrartcl}
\usepackage[T1]{fontenc}
\usepackage[latin9]{inputenc}

\usepackage{babel}
\usepackage{varioref}
\usepackage{prettyref}
\usepackage{float}
\usepackage{mathtools}
\usepackage{enumitem}
\usepackage{multirow}
\usepackage{amsmath}
\usepackage{amsthm}
\usepackage{amssymb}
\usepackage{graphicx}
\usepackage[numbers,longnamesfirst]{natbib}

\makeatletter

\providecommand{\tabularnewline}{\\}
\floatstyle{ruled}
\newfloat{algorithm}{tbp}{loa}
\providecommand{\algorithmname}{Algorithm}
\floatname{algorithm}{\protect\algorithmname}

\theoremstyle{plain}
\newtheorem{thm}{\protect\theoremname}
\theoremstyle{plain}
\newtheorem{prop}[thm]{\protect\propositionname}
\theoremstyle{remark}
\newtheorem{rem}[thm]{\protect\remarkname}
\theoremstyle{plain}
\newtheorem{cor}[thm]{\protect\corollaryname}
\theoremstyle{plain}
\newtheorem{lem}[thm]{\protect\lemmaname}

\def\documenttitle{Lost customer approximations of SOQN with application to RMFS}

\usepackage{mathalfa}
\usepackage{lscape}
\usepackage{subfigure}
\usepackage{relsize}
\usepackage{listings} 
\usepackage{booktabs} 
\usepackage{array}
\usepackage{siunitx}
\usepackage[titletoc,title]{appendix}
\usepackage{algorithm}
\usepackage{algpseudocode}
\usepackage{authblk} 

\usepackage[draft=false, headsepline=0.4pt, automark,
autooneside=false,
]{scrlayer-scrpage} 

\ohead{\leftmark}
\ihead{\rightmark}

\usepackage[unicode=true,
bookmarks=true,bookmarksnumbered=false,bookmarksopen=false,
breaklinks=true,pdfborder={0 0 0},backref=false,colorlinks=false]   
{hyperref}

\hypersetup{
	linkcolor=blue,     
	citecolor=blue, urlcolor=blue, filecolor=blue,
	pdftitle={\documenttitle},
	pdfauthor={Sonja Otten},
	pdfkeywords={}
}

\newrefformat{def}{Definition \ref{#1}}
\newrefformat{rem}{Remark \ref{#1}}
\newrefformat{sect}{Section \ref{#1}}
\newrefformat{sub}{Section \ref{#1}}
\newrefformat{prop}{Proposition \ref{#1}}
\newrefformat{thm}{Theorem \ref{#1}}
\newrefformat{chap}{Chapter \ref{#1}}
\newrefformat{cor}{Corollary \ref{#1}}
\newrefformat{par}{Paragraph \ref{#1}}
\newrefformat{ex}{Example \ref{#1}}
\newrefformat{part}{Part \ref{#1}}
\newrefformat{fig}{Figure \ref{#1}}
\newrefformat{appx}{Appendix \ref{#1}}
\newrefformat{alg}{Algorithm \ref{#1}}

\makeatother
\newcommand{\LostCustomers}{lost customers}
\newcommand{\similar}{modified}
\newcommand{\withLostCustomers}{with lost customers}
\newcommand{\stabilityNetwork}{stability network}
\newcommand{\StabilityNetwork}{Stability network}

\addto\captionsbritish{\renewcommand{\corollaryname}{Corollary}}
\addto\captionsbritish{\renewcommand{\lemmaname}{Lemma}}
\addto\captionsbritish{\renewcommand{\propositionname}{Proposition}}
\addto\captionsbritish{\renewcommand{\remarkname}{Remark}}
\addto\captionsbritish{\renewcommand{\theoremname}{Theorem}}
\addto\captionsenglish{\renewcommand{\algorithmname}{Algorithm}}
\addto\captionsenglish{\renewcommand{\corollaryname}{Corollary}}
\addto\captionsenglish{\renewcommand{\lemmaname}{Lemma}}
\addto\captionsenglish{\renewcommand{\propositionname}{Proposition}}
\addto\captionsenglish{\renewcommand{\remarkname}{Remark}}
\addto\captionsenglish{\renewcommand{\theoremname}{Theorem}}
\providecommand{\corollaryname}{Corollary}
\providecommand{\lemmaname}{Lemma}
\providecommand{\propositionname}{Proposition}
\providecommand{\remarkname}{Remark}
\providecommand{\theoremname}{Theorem}

\begin{document}
	
\title{\vspace{-1.5cm}Lost-customers approximation of semi-open queueing networks with
	backordering\\
	- An application to minimise the number of robots in robotic mobile
	fulfilment systems}

\date{}
\renewcommand\Affilfont{\small \itshape}
\author[a,c]{Sonja Otten} 
\author[a]{Ruslan Krenzler}
\author[a]{Lin Xie} 
\affil[a]{Leuphana University of L\"uneburg, Universit\"atsallee 1, 21335 L\"uneburg } 
\author[b]{Hans Daduna}
\affil[b]{University of Hamburg, Bundesstra\ss e 55, 20146 Hamburg}
\author[c]{Karsten Kruse}
\affil[c]{Hamburg University of Technology, Am Schwarzenberg-Campus 3, 21073 Hamburg}
\renewcommand\Authands{, }

\maketitle
\vspace{-1.0cm}
\begin{abstract}
	\noindent We consider a semi-open queueing network (SOQN), where a
	customer requires exactly one resource from the resource pool for
	service. If there is a resource available, the customer is immediately
	served and the resource enters an inner network. If there is no resource
	available, the new customer has to wait in an external queue until
	one becomes available (``backordering''). When a resource exits
	the inner network, it is returned to the resource pool and waits for
	another customer.
	
	In this paper, we present a new solution approach. To approximate
	the inner network with the resource pool  of the SOQN, we consider
	a modification, where newly arriving customers will decide not to
	join the external queue and are lost if the resource pool is empty
	(``{\LostCustomers}''). We prove that we can adjust the arrival
	rate of the modified system so that the throughputs in each node are
	pairwise identical to those in the original network. We also prove
	that the probabilities that the nodes with constant service rates
	are idling are pairwise identical too. Moreover, we provide a closed-form
	expression for these throughputs and probabilities of idle nodes.
	
	To approximate the external queue of the SOQN with backordering, we
	construct a reduced SOQN with backordering, where the inner network
	consists only of one node, by using Norton's theorem and results from
	the lost-customers modification. In a final step, we use the closed-form
	solution of this reduced SOQN, to estimate the performance of the
	original SOQN.
	
	We apply our results to robotic mobile fulfilment systems (RMFSs).
	These are a new type of warehousing system, which has received attention
	recently, due to increasing growth in the e-commerce sector. Instead
	of sending pickers to the storage area to search for the ordered items
	and pick them, robots carry shelves with ordered items from the storage
	area to picking stations. We model the RMFS as an SOQN, analyse its
	stability and determine the minimal number of robots for such systems
	using the results from the first part.
	
	\begingroup
	\renewcommand\thefootnote{}%
	\footnote{
		ORCID IDs and email addresses:\\
		Sonja Otten \includegraphics[width=0.8em]{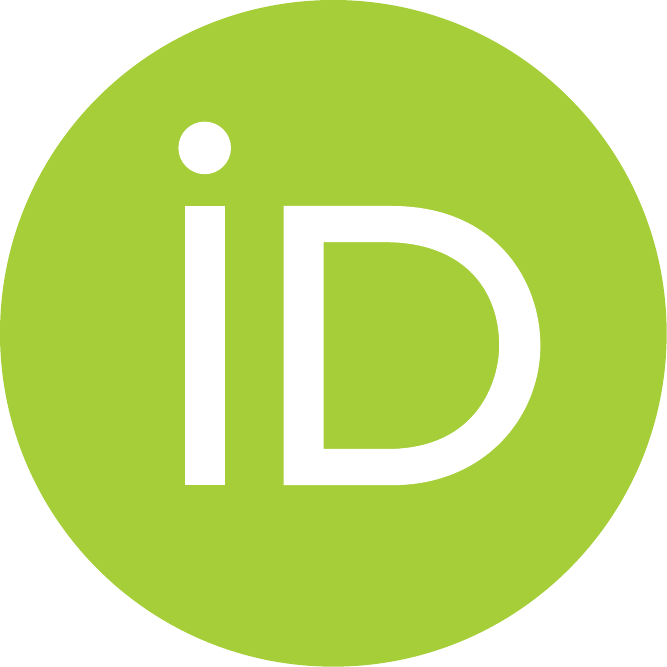}\hspace{0.4em}\url{https://orcid.org/0000-0002-3124-832X},
		sonja.otten@leuphana.de,\\
		Ruslan Krenzler \includegraphics[width=0.8em]{image/orcid_icon}\hspace{0.4em}\url{https://orcid.org/0000-0002-6637-1168},
		ruslan.krenzler@leuphana.de,\\
		Lin Xie \includegraphics[width=0.8em]{image/orcid_icon}\hspace{0.4em}\url{https://orcid.org/0000-0002-3168-4922},
		xie@leuphana.de, \\
		Hans Daduna \includegraphics[width=0.8em]{image/orcid_icon}\hspace{0.4em}\url{https://orcid.org/0000-0001-6570-3012},
		daduna@math.uni-hamburg.de, \\
		Karsten Kruse \includegraphics[width=0.8em]{image/orcid_icon}\hspace{0.4em}\url{https://orcid.org/0000-0003-1864-4915},
		karsten.kruse@tuhh.de
	}%

	\addtocounter{footnote}{-1}%
	\endgroup
	
\end{abstract}
\emph{MSC 2010 Subject Classification:} 60K25, 90B22, 90C59, 90B05

\noindent \emph{Keywords:} semi-open queueing network, backordering,
{\LostCustomers}, lost sales, product form approximation, robotic
mobile fulfilment system, warehousing

\selectlanguage{english}%
\global\long\def\phantomeq{\mathrel{\phantom{=}}}

\global\long\def\evect{\mathbf{e}}

\global\long\def\nodeSetNull{\overline{J}_{0}}
\global\long\def\nodes{J}
\global\long\def\nodeSet{\overline{J}}

\global\long\def\robotnumb{N}

\global\long\def\poolSize{N}

\global\long\def\routingMatrix{\mathcal{R}}
\global\long\def\routingProb{r}

\global\long\def\nodeProc{Y}
\global\long\def\nodeProcRV{\widehat{Y}}
\global\long\def\nodeProcALL{\mathbf{Y}}
\global\long\def\nodeQueueProcALL{\mathbf{Y}}

\global\long\def\nodeQueue{n}
\global\long\def\nodeServiceRate{\nu}
\global\long\def\nodeQueueALL{\mathbf{n}}

\global\long\def\externalQueueProc{X_{\text{ex}}}
\global\long\def\externalQueueProcRV{\widehat{X}_{\text{ex}}}
\global\long\def\externalQueue{n_{\text{ex}}}

\global\long\def\externalLength{L_{\text{ex}}}
\global\long\def\externalWaitingTime{W_{\text{ex}}}
\global\long\def\innerWaitingTime{W_{\text{in}}}

\global\long\def\piLS{\pi_{\text{LC}}}
\global\long\def\thetaLS{\theta_{\text{LC}}}
\global\long\def\normLS{C_{\text{LC}}}
\global\long\def\normB{C_{\text{}}^{\text{stb}}}
\global\long\def\normBstrich{C_{\text{BO}}^{\text{stb}'}}

\global\long\def\specialNormB{C_{\text{BO}}}

\global\long\def\piB{\pi_{\text{BO}}}
\global\long\def\thetaB{\theta_{\text{BO}}}

\global\long\def\arrival{\lambda_{\text{BO}}}
\global\long\def\effarrival{\lambda_{\text{eff}}}
\global\long\def\arrivalLS{\lambda_{\text{LC}}}
\global\long\def\maxArrival{\lambda_{\text{BO,max}}}
\global\long\def\customerOrdersArrival{\lambda_{\text{CO}}}

\global\long\def\procbackordering{X_{\text{ex}}}
\global\long\def\backordering{n_{\text{ex}}}

\global\long\def\procfreerobi{Y_{\text{idle robots}}}
\global\long\def\freerobi{k}

\global\long\def\procToPod{Y_{\text{to pod}}}
\global\long\def\toPod{m_{\text{1}}}
\global\long\def\transToPod{\nu_{1}}
\global\long\def\phiToPod{\phi_{1}}

\global\long\def\procToPick{Y_{\text{to picker}}}
\global\long\def\toPick{m_{2}}
\global\long\def\transToPicker{\nu_{\text{2}}}
\global\long\def\phiToPick{\phi_{2}}

\global\long\def\procPickerQueue{Y_{\text{picker}}}
\global\long\def\pickerQueue{n}
\global\long\def\picking{\mu}

\global\long\def\procToStorage{Y_{\text{to storage}}}
\global\long\def\toStorage{m_{3}}
\global\long\def\transToStorage{\nu_{3}}
\global\long\def\phiToStorage{\phi_{3}}

\global\long\def\etaRatio{\eta}
\global\long\def\etaRatioVect{\mathbf{\boldsymbol{\eta}}}
\global\long\def\phaseSpace{\widetilde{E}_{\text{+}}}

\global\long\def\subRate{\varphi}
\global\long\def\piBsub{\widetilde{\pi}_{\text{BO}}}

\global\long\def\etaLS{\widetilde{\eta}}
\global\long\def\lostStateSpace{E_{LC}}
\global\long\def\resourceProb{\pi_{\text{LC,0}}}

\global\long\def\procfreerobi{Y_{0}}
\global\long\def\freerobi{n_{0}}
\global\long\def\freerobiState{0}

\global\long\def\procToPod{Y_{\text{sp}}}
\global\long\def\toPod{n_{\text{sp}}}
\global\long\def\toPodState{\text{sp}}
\global\long\def\transToPod{\mu_{\text{sp}}}
\global\long\def\phiToPod{\phi_{\text{sp}}}

\global\long\def\procToPickOne{Y_{\text{pp}_{1}}}
\global\long\def\toPickOneState{\text{pp}_{1}}
\global\long\def\toPickOne{n_{\text{pp}_{1}}}
\global\long\def\transToPickerOne{\mu_{\text{pp}_{1}}}
\global\long\def\phiToPickOne{\phi_{\text{pp}_{1}}}
\global\long\def\transprobPickOne{q_{\text{pp}_{1}}}

\global\long\def\procToPickTwo{Y_{\text{pp}_{2}}}
\global\long\def\toPickTwoState{\text{pp}_{2}}
\global\long\def\toPickTwo{n_{\text{pp}_{2}}}
\global\long\def\transToPickerTwo{\mu_{\text{pp}_{2}}}
\global\long\def\phiToPickTwo{\phi_{\text{pp}_{2}}}
\global\long\def\transprobPickTwo{q_{\text{pp}_{2}}}

\global\long\def\procPickerQueueOne{Y_{\text{p}_{1}}}
\global\long\def\pickerQueueOne{n_{\text{p}_{1}}}
\global\long\def\pickerQueueOneState{\text{p}_{1}}
\global\long\def\pickingOne{\nu_{\text{p}_{1}}}
\global\long\def\transprobReplOne{q_{\text{p}_{1}\text{r}}}
\global\long\def\transprobpickReplOne{q_{\text{p}_{1}\text{s}}}

\global\long\def\procPickerQueueTwo{Y_{\text{p}_{2}}}
\global\long\def\pickerQueueTwoState{\text{p}_{2}}
\global\long\def\pickerQueueTwo{n_{\text{p}_{2}}}
\global\long\def\pickingTwo{\nu_{\text{p}_{2}}}
\global\long\def\transprobReplTwo{q_{\text{p}_{2}\text{r}}}
\global\long\def\transprobpickReplTwo{q_{\text{p}_{2}\text{s}}}

\global\long\def\procToStorageOne{Y_{\text{p}_{1}\text{s}}}
\global\long\def\toStorageOneState{\text{p}_{1}\text{s}}
\global\long\def\toStorageOne{n_{\text{p}_{1}\text{s}}}
\global\long\def\transToStorageOne{\mu_{\text{p}_{1}\text{s}}}
\global\long\def\phiToStorageOne{\phi_{\text{p}_{1}\text{s}}}

\global\long\def\procToStorageTwo{Y_{\text{p}_{2}\text{s}}}
\global\long\def\toStorageTwoState{\text{p}_{2}\text{s}}
\global\long\def\toStorageTwo{n_{\text{p}_{2}\text{s}}}
\global\long\def\transToStorageTwo{\mu_{\text{p}_{2}\text{s}}}
\global\long\def\phiToStorageTwo{\phi_{\text{p}_{2}\text{s}}}

\global\long\def\procToReplOne{Y_{\text{p}_{1}\text{r}}}
\global\long\def\toReplOneState{\text{p}_{1}\text{r}}
\global\long\def\toReplOne{n_{\text{p}_{1}\text{r}}}
\global\long\def\transToReplOne{\mu_{\text{p}_{1}\text{r}}}
\global\long\def\phiToReplOne{\phi_{\text{p}_{1}\text{r}}}

\global\long\def\procToReplTwo{Y_{\text{p}_{2}\text{r}}}
\global\long\def\toReplTwoState{\text{p}_{2}\text{r}}
\global\long\def\toReplTwo{n_{\text{p}_{2}\text{r}}}
\global\long\def\transToReplTwo{\mu_{\text{p}_{2}\text{r}}}
\global\long\def\phiToReplTwo{\phi_{\text{p}_{2}\text{r}}}

\global\long\def\procRepl{Y_{\text{r}}}
\global\long\def\replQueueState{\text{r}}
\global\long\def\replQueue{n_{\text{r}}}
\global\long\def\repl{\nu_{r}}

\global\long\def\procReplToStorage{Y_{\text{rs}}}
\global\long\def\toReplToStorageState{\text{rs}}
\global\long\def\toReplToStorage{n_{\text{rs}}}
\global\long\def\transReplToStorage{\mu_{\text{rs}}}
\global\long\def\phiToStorage{\phi_{\text{rs}}}

\global\long\def\freerobi{k_{\text{idle robots}}}

\global\long\def\procTransPick{Y_{\text{sub}}}
\global\long\def\transPickQueue{n}
\global\long\def\phiTransPick{\varphi}

\global\long\def\stabilityNetworkTh#1{TH_{#1}^{\text{stb}}}
\global\long\def\abrevC#1{b(#1)}

\global\long\def\podToOrderRatio{\sigma_{\text{pod/order}}}

\global\long\def\matchingDelay{W_{\text{alg}}}
\global\long\def\assembledTime{W_{\text{assembled}}}

\global\long\def\TOtask{TO_{\text{task}}}
\global\long\def\ThBO#1{TH_{\text{BO},#1}}
\global\long\def\ThLC#1{TH_{\text{LC},#1}}
\global\long\def\meanVisits#1{V_{#1}}
\selectlanguage{british}%

\tableofcontents{}

\section{Introduction}

Queueing networks can be broadly classified into three categories,
see e.g.~\citet[p. 21ff.]{yao2001fundamentals}, \citet{roy.2016},
\citet{azadeh2017robot}: open queueing network (OQN), closed queueing
network (CQN) and semi-open queueing network (SOQN).

In an OQN, customers arrive from an external source, request service
at several nodes and then leave the system.

In contrast, in a CQN, there are no external arrivals and departures.
The number of customers, which are cycling in the system, is fixed.

An SOQN has characteristics of both OQNs and CQNs, see e.g.~\citet{Jia.Heragu.2009}.
It is similar to an OQN in the sense that customers arrive from an
external source and leave the system after service. It is similar
to a CQN in the sense that there is an overall capacity constraint
for the inner network. A customer needs a resource from an associated
resource pool for service. If there is a resource available, the customer
is immediately served and the resource enters an inner network. If
there is no resource available, the new customer has to wait in an
external queue until one becomes available -- we call this waiting
regime ``backordering''. When a resource exits the inner network,
it returns to the resource pool and waits for the next customer. 

There are several application areas of SOQNs. For example, they are
adopted for performance analysis of manufacturing systems and service
systems (logistics, communication, warehousing and health care), see
\citet{roy.2016}.

\subsection[Contribution I: New solution approach to SOQN-BO]{Contribution I: New solution approach to SOQN with backordering}

In this paper we focus on semi-open queueing networks (SOQNs). The
literature on SOQNs is overwhelming, so we point only to the most
relevant sources for our present investigation. A detailed overview
about SOQNs and their solution methods is presented in \citet{Jia.Heragu.2009},
\citet{EKREN201478} and \citet{roy.2016}. Roy also compares the
numerical accuracy of the solution methods. A recent article, which
is not included in these reviews, is provided by \citet{Kim20181}.

The most common solution approaches are the matrix-geometric method,
aggregation method, network decomposition approach, parametric decomposition
method and performance measure bounds. 

For SOQNs with backordering and an inner network, which consists of
more than one node, closed-form expressions for the steady-state distributions
are not available. In this paper we present a new solution approach,
which is depicted in \prettyref{fig:overview-models}.

To approximate the resource network ($=$ inner network and resource
pool) of the SOQN, we consider a modification, where new arriving
customers will be lost if the resource pool is empty (``{\LostCustomers}'').
For such a modification, closed-form expressions for the steady-state
distribution in product form are available. We prove that we can adjust
the arrival rate in the modified network so that the throughputs in
each node are pairwise identical to those in the original network.
We also prove that the probabilities that the nodes with constant
service rates are idling are also pairwise identical. Moreover, we
provide a closed-form expression both for the throughputs and for
these special idle probabilities.

To approximate the external queue of the SOQN, we use an indirect
two-step approach. In the first step, we reduce the complexity of
the modified SOQN by applying Norton's theorem. In the next step,
we remove the lost-customer property of the reduced SOQN with lost
customers to get the backordering property again.

The idea to approximate backordering systems by lost-customer systems
follows \citet[Section 2.1.4, p. 34ff.]{krenzler:16}. This approximation
method is interesting because in the literature there are already
several results for queueing networks with lost customers, e.g.~lost
sales models in \citet{otten:18}, \citet{schwarz;sauer;daduna;kulik;szekli:06}
and \citet{krishnamoorthy:2011}. For an overview of the literature
on systems with lost sales, we refer to \citet{bijvank;vis:11}.

\subsection{Contribution II: Application to robotic mobile fulfilment systems}

We focus on robotic mobile fulfilment systems (RMFSs). RMFSs are a
new type of warehousing system which has received attention recently,
due to increasing growth in the e-commerce sector. Instead of sending
pickers to the storage area to search for the ordered items and pick
them, robots carry shelves -- here called \emph{pods} -- with ordered
items from the storage area to picking stations. At every picking
station there is a person -- the picker -- who takes items from
the pods and packs them into boxes according to customers' orders.
When the picker does not need the pod any more, the robot either transports
the pod directly back to the storage area or first makes a stopover
at a replenishment station. Such a fulfilment system poses many decision
problems. An overview is presented in \citet[Section 4]{RawSimDecisionRules}.
They can be classified at strategic, tactical and operational levels.
The literature on RMFS is, similar to that on SOQNs, already overwhelming,
so we only point to some references closely related to our investigations.

RMFSs are modelled as SOQNs in several articles. For an overview of
the literature, we refer to  \citet[Section 7.2 and Table 4]{azadeh2017.2}
and \citet[Section 6]{azadeh2017robot}. They classify articles according
to the decision problem of interest and the relevant methodology.
Furthermore, a recent overview is presented in \citet{dynamicPolicies}.
In our paper, we focus on decisions on the optimal number of robots.
Most of these articles analyse different decision problems from the
ones we look at in this paper. \citet{Yuan.2017} calculate the optimal
number of robots. They compare two protocols: pooled and dedicated
robots. Their investigations are based on OQNs, rather than SOQNs
and they do not consider a replenishment station in their model. \citet{zou2018}
determine the number of robots for different battery recovery strategies.
Their investigations are based on a nested SOQN.

\subsection{Structure of the paper}

In \prettyref{sec:A-general-semi-open}, we describe a general SOQN
with backordering and analyse the stability. In \prettyref{sec:throughputs-and-idle-times-LC},
we calculate the throughputs and special idle probabilities in steady
state. After that, we introduce our new approximation method for an
SOQN with backordering. Our approach is depicted in \prettyref{fig:overview-models}.
In \prettyref{sec:Approximation}, we present an approximation of
the resource network. For that, in \prettyref{sec:Lost-network},
we analyse a modification, where newly arriving customers are lost,
if the resource pool is empty. Then, in \prettyref{sec:Adjustment},
we adjust this system. In \prettyref{sect:TH_BO}, we calculate the
throughputs and special idle probabilities in steady state. In \prettyref{sec:Approximation-of-the-external-queue},
we present an approximation for the external queue. For that, in \prettyref{sec:small-modified},
we reduce the complexity of the modified SOQN with lost customers
by applying Norton's theorem. Then, in \prettyref{sub:small-with-backordering},
we remove the lost-customer property to get the backordering property
again. In \prettyref{sec:application}, we present an application
of our results. We model the RMFS as an SOQN with backordering and
apply our results to the problem: ``What is the optimal number of
robots?'' We formulate an algorithm to calculate the minimum number
of robots and present numerical examples. Finally, \prettyref{sec:Conclusion}
concludes the paper.

\begin{figure}[h]
\includegraphics[width=1\textwidth]{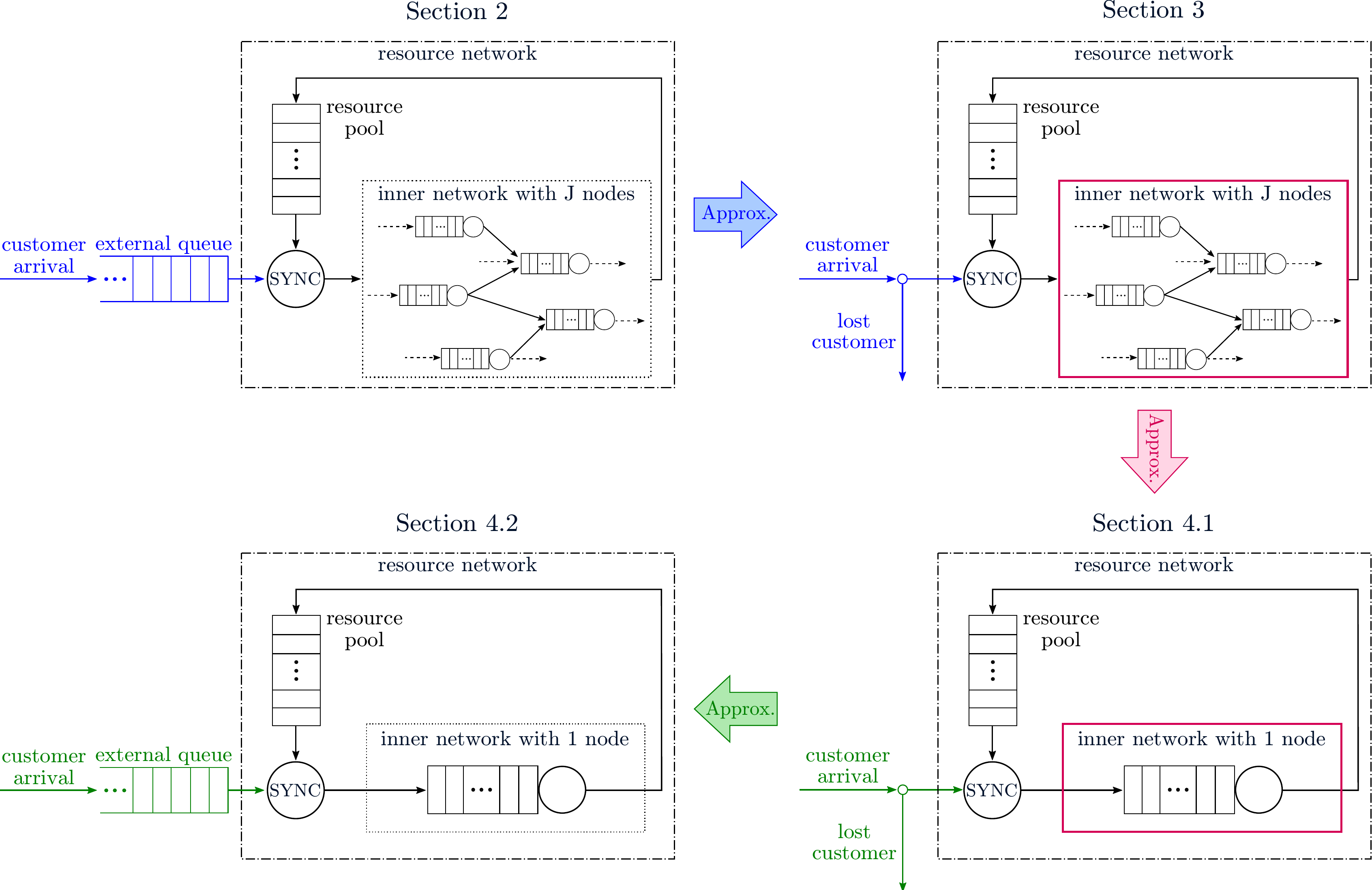}\caption{\label{fig:overview-models}Overview of the models. We use the same
colour for parts, which we change in a single approximation step.}

\end{figure}

\subsection{Notation and preliminaries}

$\mathbb{N}:=\left\{ 1,2,3,\ldots\right\} $, $\mathbb{N}_{0}:=\{0\}\cup\mathbb{N}$,
$\mathbb{R}_{0}^{+}:=[0,\infty)$ and $\mathbb{R}^{+}:=(0,\infty)$.

The vector $\mathbf{0}$ is a row vector of appropriate size with
all entries equal to $0$. The vector $\mathbf{e}$ is a column vector
of appropriate size with all entries equal to $1$. The vector $\mathbf{e}_{i}=(0,\ldots,0,\underbrace{1}_{\mathclap{i\text{-th element}}},0,\ldots,0)$
is a vector of appropriate size.

$1_{\left\{ expression\right\} }$ is the indicator function which
is $1$ if $expression$ is true and $0$ otherwise.

Empty sums are 0, and empty products are 1.

We call a ``generator'' a matrix $M\in\mathbb{R}^{K\times K}$ with
countable index set $K$ if all its off-diagonal elements are non-negative
and all its row sums are equal to zero.

Throughout this paper it is assumed that all random variables are
defined on a common probability space $(\Omega,{\cal F},P)$. Furthermore,
by ``Markov process'' we mean a time-homogeneous continuous-time strong
Markov process with discrete state space -- also known as a Markov
jump process. All Markov processes are assumed to be regular and have
cdlg paths, i.e.\ each path of a process is right-continuous and
has left limits everywhere. We call a Markov process regular if it
is non-explosive (i.e.\ the sequence of jump times of the process
diverges almost surely) and its transition intensity matrix is a generator
matrix.

\section{SOQN with backordering\label{sec:A-general-semi-open}}

\subsection{Description of the model\label{sec:RMFS-general-model}}

An SOQN with backordering is shown in \prettyref{fig:SOQN}. Henceforth,
we call it an \emph{SOQN-BO}. It represents a queueing network with
an additional resource pool. 

\begin{figure}[h]
\begin{centering}
\includegraphics[width=0.85\textwidth]{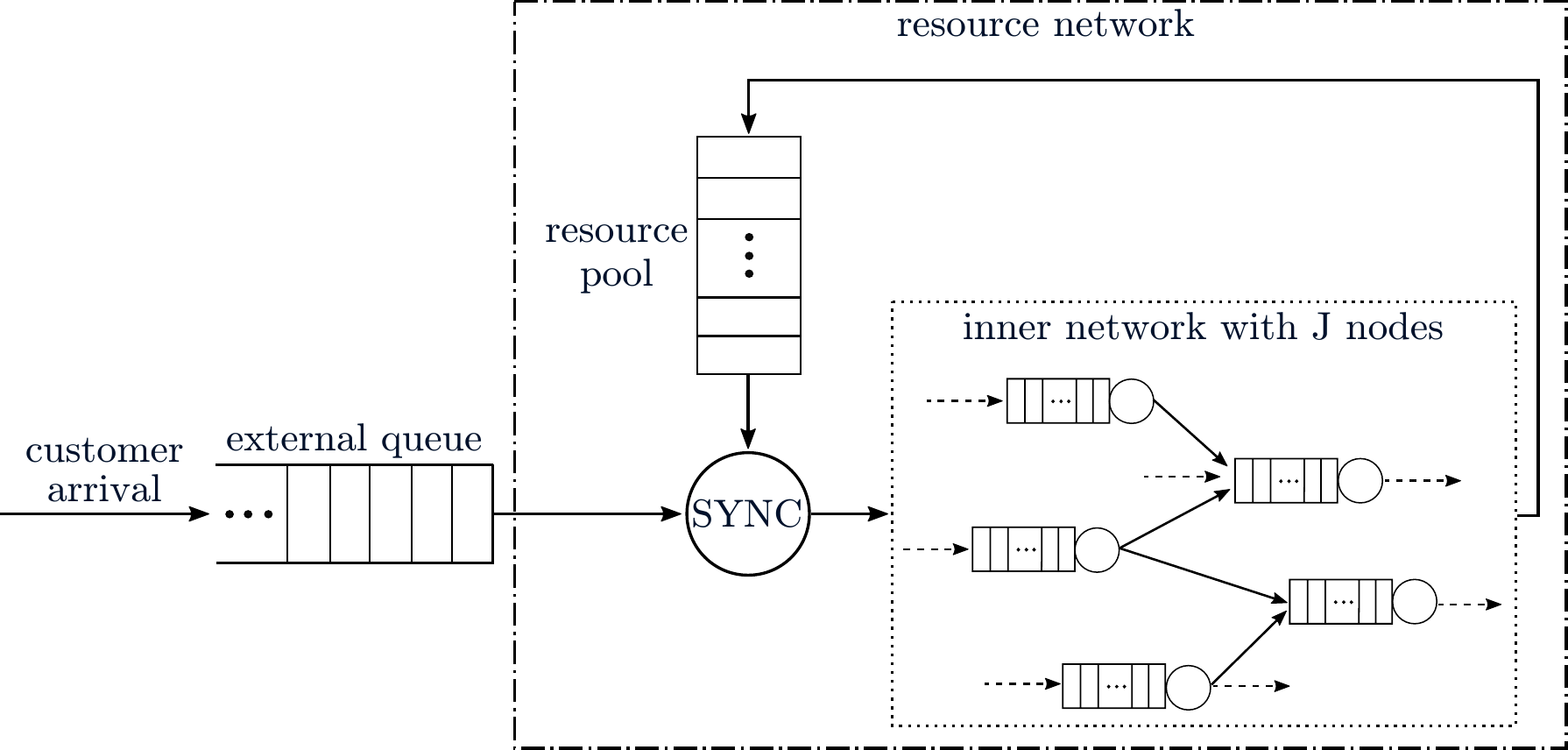}
\par\end{centering}
\centering{}\caption{\label{fig:SOQN}An SOQN with backordering.}
\end{figure}

Customers arrive one by one according to a Poisson process with rate
$\arrival>0$. Every customer requires exactly one resource from resource
pool for service. If there is a resource available, the customer is
immediately served and the resource enters an inner network. If there
is no resource available, the new customer has to wait in an external
queue under the first-come, first-served (FCFS) regime until a resource
becomes available (``backordering'').

When the resource exits the inner network, it returns to the resource
pool, which is henceforth referred to as node $0$, and waits for
the next customer. Whenever the external queue is not empty and a
resource item is returned to the resource pool, this item is instantaneously
synchronised with the customer at the head of the line. The resources
therefore move in a closed network. We will call this network \emph{resource
network}. The maximal number of resources in the resource pool is
$\poolSize$.

The inner network consists of $\nodes\geq1$ numbered service stations
(nodes), denoted by $\nodeSet:=\left\{ 1,\ldots,\nodes\right\} $.
Each station $j$ consists of a single server with infinite waiting
room under the FCFS regime or the processor sharing regime. Customers
in the network are indistinguishable. The service times are exponentially
distributed random variables with mean $1$. If there are $n_{j}>0$
customers present at node $j$, service at node $j$ is provided with
intensity $\nodeServiceRate_{j}(\nodeQueue_{j})>0$. All service and
inter-arrival times constitute an independent family of random variables.

Movements of resources in the inner network are governed by a Markovian
routing mechanism: After synchronisation with a customer, a resource
visits node $j$ with probability $\routingProb(0,j)\geq0$. A resource,
when leaving node $i$, selects with probability $\routingProb(i,j)\geq0$
to visit node $j$ next, and then enters node $j$ immediately. It
starts service if it finds the server idle, otherwise it joins the
tail of the queue at node $j$. This resource can also leave the inner
network with probability $\routingProb(i,0)\geq0$. It holds $\sum_{j=0}^{\nodes}\routingProb(i,j)=1$
with $\routingProb(0,0):=0$ for all $i\in\nodeSetNull:=\left\{ 0,1,\ldots,\nodes\right\} $.
Given the departure node $i$, the resource's routing decision is
made independently of the network's history. We assume that the routing
matrix $\routingMatrix:=\left(\routingProb(i,j):i,j\in\nodeSetNull\right)$
is irreducible.

To obtain a Markovian process description, we denote by $\externalQueueProc(t)$
the number of customers in the external queue at time $t\geq0$, by
$\nodeProc_{0}(t)$ the number of resources in the resource pool at
time $t\geq0$ and by $\nodeProc_{j}(t)$, $j\in\nodeSet$, the number
of resources present at node $j$ in the inner network at time $t\geq0$,
either waiting or in service. We call this $\nodeProc_{j}(t)$ queue
length at node $j\in\nodeSet$ at time $t\geq0$. Then $\nodeProcALL(t):=\left(\nodeProc_{j}(t):j\in\nodeSetNull\right)$
is the queue length vector of the resource network at time $t\geq0$.
We define the joint queue length process of the semi-open network
with \textbf{b}ack\textbf{o}rdering by
\[
Z_{\text{BO}}:=\left(\left(\externalQueueProc(t),\nodeProcALL(t)\right):t\geq0\right).
\]
Then, due to the independence and memorylessness assumptions,\label{independence-memorylessness}
$Z_{\text{BO}}$ is a homogeneous Markov process with state space
\begin{align*}
E & :=\big\{\left(0,\nodeQueue_{0},\nodeQueue_{1},\ldots,\nodeQueue_{J}\right):\nodeQueue_{j}\in\left\{ 0,\ldots,\poolSize\right\} \:\forall j\in\nodeSetNull,\sum_{j\in\nodeSetNull}\nodeQueue_{j}=\poolSize\big\}\\
 & \phantomeq\cup\big\{\left(\externalQueue,0,\nodeQueue_{1},\ldots,\nodeQueue_{J}\right):\externalQueue\in\mathbb{N},\:\nodeQueue_{j}\in\left\{ 0,\ldots,\poolSize\right\} \:\forall j\in\nodeSet,\sum_{j\in\nodeSet}\nodeQueue_{j}=\poolSize\big\}.
\end{align*}
$Z_{\text{BO}}$ is irreducible on $E$.

\subsection{Stability\label{sec:RMFS-general-stat-distr}}

In this section, we analyse the stability of our system. This is important
from a practical point of view: It ensures that the system processes
customers quickly enough.

The Markov process $Z_{\text{BO}}$ has an infinitesimal generator
$\mathbf{Q}:=\left(q(z;\tilde{z}):z,\tilde{z}\in E\right)$ with the
following transition rates for $\left(\externalQueue,\nodeQueueALL\right),\ \left(0,\nodeQueueALL\right)\in E$,
where $\nodeQueueALL:=\left(\nodeQueue_{j}:j\in\nodeSetNull\right)$:
\begin{align*}
 & q\left(\left(\externalQueue,\nodeQueueALL\right);\left(\externalQueue+1,\nodeQueueALL\right)\right)=\arrival\cdot1_{\left\{ \nodeQueue_{0}=0\right\} },\ \externalQueue\geq0,\\
 & q\left(\left(0,\nodeQueueALL\right);\left(0,\nodeQueueALL-\evect_{0}+\evect_{i}\right)\right)=\arrival\cdot\routingProb(0,i)\cdot1_{\left\{ \nodeQueue_{0}>0\right\} },\ i\in\nodeSet,\\
 & q\left(\left(\externalQueue,\nodeQueueALL\right);\left(\externalQueue,\nodeQueueALL-\evect_{i}+\evect_{j}\right)\right)=\nodeServiceRate_{i}(\nodeQueue_{i})\cdot\routingProb(i,j)\cdot1_{\left\{ \nodeQueue_{i}>0\right\} },\ \externalQueue\geq0,\ i,j\in\nodeSet,\\
 & q\left(\left(0,\nodeQueueALL\right);\left(0,\nodeQueueALL-\evect_{i}+\evect_{0}\right)\right)=\nodeServiceRate_{i}(\nodeQueue_{i})\cdot\routingProb(i,0)\cdot1_{\left\{ \nodeQueue_{i}>0\right\} },\ i\in\nodeSet,\\
 & q\left(\left(\externalQueue,\nodeQueueALL\right);\left(\externalQueue-1,\nodeQueueALL-\evect_{i}+\evect_{j}\right)\right)=\nodeServiceRate_{i}(\nodeQueue_{i})\cdot\routingProb(i,0)\cdot\routingProb(0,j)\cdot1_{\left\{ \externalQueue>0\right\} }\cdot1_{\left\{ \nodeQueue_{i}>0\right\} },\ \externalQueue>0,\ i\in\nodeSet.
\end{align*}
Furthermore, $q(z;\tilde{z})=0$ for any other pair $z\neq\tilde{z}$,
and
\[
q\left(z;z\right)=-\sum_{\substack{\tilde{z}\in E,\\
\tilde{z}\neq z
}
}q\left(z;\tilde{z}\right)\qquad\forall z\in E.
\]

The Markov process $Z_{\text{BO}}$ is a level-independent quasi-birth-and-death
process in the sense of \citet[Def. 1.3.1, p. 12]{latouche.ramaswami:99}.
The level is the length $\externalQueue$ of the external queue. The
phase is the state of the resource network. For a level equal to zero,
the phase space is
\[
\widetilde{E}_{0}:=\big\{\left(\nodeQueue_{0},\nodeQueue_{1},\ldots,\nodeQueue_{J}\right):\ \nodeQueue_{j}\in\left\{ 0,\ldots,\poolSize\right\} \:\forall j\in\nodeSetNull,\sum_{j\in\nodeSetNull}\nodeQueue_{j}=\poolSize\big\}.
\]
When the level is positive, there is at least one customer in the
external queue. This is only possible if the resource pool is empty.
Hence, for all positive levels the phase space is
\[
\phaseSpace:=\big\{\left(0,\nodeQueue_{1},\ldots,\nodeQueue_{J}\right):\:\nodeQueue_{j}\in\left\{ 0,\ldots,\poolSize\right\} \:\forall j\in\nodeSet,\:\sum_{j\in\nodeSet}\nodeQueue_{j}=\poolSize\big\}.
\]

Arranging the states by level, the corresponding infinitesimal generator
$\mathbf{Q}$ of $Z_{\text{BO}}$ can be written as
\[
\mathbf{Q}=\left(\begin{array}{ccccc}
\mathbf{B}_{0} & \mathbf{B}_{1}\\
\mathbf{B}_{2} & \mathbf{A}_{0} & \mathbf{A}_{1}\\
 & \mathbf{A}_{-1} & \mathbf{A}_{0} & \mathbf{A}_{1}\\
 &  &  & \ddots\ddots & \ddots
\end{array}\right),
\]
where $\mathbf{B}_{0}\in\mathbb{R}^{\widetilde{E}_{0}\times\widetilde{E}_{0}}$,
$\mathbf{B}_{1}\in\mathbb{R}^{\widetilde{E}_{0}\times\phaseSpace}$,
$\mathbf{B}_{2}\in\mathbb{R}^{\phaseSpace\times\widetilde{E}_{0}}$
and $\mathbf{A}_{-1}$, $\mathbf{A}_{0}$, $\mathbf{A}_{1}\in\mathbb{R}^{\phaseSpace\times\phaseSpace}$
are matrices.\\
$\mathbf{A}_{1}$ is a non-negative matrix with the following positive
elements:
\[
a_{1}\left(\left(0,\nodeQueue_{1},\ldots,\nodeQueue_{J}\right);\left(0,\nodeQueue_{1},\ldots,\nodeQueue_{J}\right)\right)=\arrival.
\]
$\mathbf{A}_{-1}$ is a non-negative matrix with at most the following
elements non-negative:
\begin{align}
 & a_{-1}\left(\left(0,\nodeQueue_{1},\ldots,\nodeQueue_{J}\right);\left(0,\nodeQueue_{1},\ldots,\nodeQueue_{i}-1,\ldots,\nodeQueue_{j}+1,\ldots,\nodeQueue_{J}\right)\right)\nonumber \\
 & =\nodeServiceRate_{i}(\nodeQueue_{i})\cdot\routingProb(i,0)\cdot\routingProb(0,j)\cdot1_{\left\{ \nodeQueue_{i}>0\right\} },\quad i,j\in\nodeSet.\label{eq:alpha-minus1}
\end{align}
$\mathbf{A}_{0}$ has non-negative off-diagonal elements and strictly
negative diagonals. The off-diagonal elements are
\begin{align*}
 & a_{0}\left(\left(0,\nodeQueue_{1},\ldots,\nodeQueue_{J}\right);\left(0,\nodeQueue_{1},\ldots,\nodeQueue_{i}-1,\ldots,\nodeQueue_{j}+1,\ldots,\nodeQueue_{J}\right)\right)\\
 & =\nodeServiceRate_{i}(\nodeQueue_{i})\cdot\routingProb(i,j)\cdot1_{\left\{ \nodeQueue_{i}>0\right\} },\quad i,j\in\nodeSet.
\end{align*}

Let $\boldsymbol{\pi}_{\text{BO}}:=\left(\piB\left(\externalQueue,\nodeQueueALL\right):\left(\externalQueue,\nodeQueueALL\right)\in E\right)$
be the steady-state distribution of the stochastic Markov process
$Z_{\text{BO}}$. The global balance equations $\boldsymbol{\pi}_{\text{BO}}\cdot\mathbf{Q=0}$
are given as follows.

\vspace{1\baselineskip}
\noindent For $\externalQueue=0$
\begin{align*}
 & \piB\left(0,\nodeQueueALL\right)\cdot\Big(\arrival+\sum_{i\in\nodeSet}\sum_{j\in\nodeSet\setminus\{i\}}\nodeServiceRate_{i}(\nodeQueue_{i})\cdot\routingProb(i,j)\cdot1_{\left\{ \nodeQueue_{i}>0\right\} }+\sum_{i\in\nodeSet}\nodeServiceRate_{i}(\nodeQueue_{i})\cdot\routingProb(i,0)\cdot1_{\left\{ \nodeQueue_{i}>0\right\} }\Big)\\
 & =\sum_{i\in\nodeSet}\piB\left(0,\nodeQueueALL+\evect_{0}-\evect_{i}\right)\cdot\arrival\cdot\routingProb(0,i)\cdot1_{\left\{ \nodeQueue_{i}>0\right\} }\\
 & \phantomeq+\sum_{i\in\nodeSet}\sum_{j\in\nodeSet\setminus\{i\}}\piB\left(0,\nodeQueueALL+\evect_{i}-\evect_{j}\right)\cdot\nodeServiceRate_{i}(\nodeQueue_{i}+1)\cdot\routingProb(i,j)\cdot1_{\left\{ \nodeQueue_{j}>0\right\} }\\
 & \phantomeq+\sum_{i\in\nodeSet}\piB\left(0,\nodeQueueALL+\evect_{i}-\evect_{0}\right)\cdot\nodeServiceRate_{i}(\nodeQueue_{i}+1)\cdot\routingProb(i,0)\cdot1_{\left\{ \nodeQueue_{0}>0\right\} }\\
 & \phantomeq+\sum_{i\in\nodeSet}\sum_{j\in\nodeSet\setminus\{i\}}\piB\left(1,\nodeQueueALL+\evect_{i}-\evect_{j}\right)\cdot\nodeServiceRate_{i}(\nodeQueue_{i}+1)\cdot\routingProb(i,0)\cdot\routingProb(0,j)\cdot1_{\left\{ \nodeQueue_{j}>0\right\} }\cdot1_{\left\{ \nodeQueue_{0}=0\right\} }\\
 & \phantomeq+\sum_{i\in\nodeSet}\piB\left(1,\nodeQueueALL\right)\cdot\nodeServiceRate_{i}(\nodeQueue_{i})\cdot\routingProb(i,0)\cdot\routingProb(0,i)\cdot1_{\left\{ \nodeQueue_{i}>0\right\} }\cdot1_{\left\{ \nodeQueue_{0}=0\right\} }.
\end{align*}
For $\externalQueue>0$ which implies $\nodeQueue_{0}=0$
\begin{align*}
 & \piB\left(\externalQueue,\nodeQueueALL\right)\cdot\Big(\arrival+\sum_{i\in\nodeSet}\sum_{j\in\nodeSet\setminus\{i\}}\nodeServiceRate_{i}(\nodeQueue_{i})\cdot\routingProb(i,j)\cdot1_{\left\{ \nodeQueue_{i}>0\right\} }+\sum_{i\in\nodeSet}\nodeServiceRate_{i}(\nodeQueue_{i})\cdot\routingProb(i,0)\cdot1_{\left\{ \nodeQueue_{i}>0\right\} }\Big)\\
 & =\piB\left(\externalQueue-1,\nodeQueueALL\right)\cdot\arrival\\
 & \phantomeq+\sum_{j\in\nodeSet}\sum_{i\in\nodeSet\setminus\{i\}}\piB\left(\externalQueue,\nodeQueueALL+\evect_{i}-\evect_{j}\right)\cdot\nodeServiceRate_{i}(\nodeQueue_{i}+1)\cdot\routingProb(i,j)\cdot1_{\left\{ \nodeQueue_{j}>0\right\} }\\
 & \phantomeq+\sum_{j\in\nodeSet}\sum_{i\in\nodeSet\setminus\{i\}}\piB\left(\externalQueue+1,\nodeQueueALL+\evect_{i}-\evect_{j}\right)\cdot\nodeServiceRate_{i}(\nodeQueue_{i}+1)\cdot\routingProb(i,0)\cdot\routingProb(0,j)\cdot1_{\left\{ \nodeQueue_{j}>0\right\} }\\
 & \phantomeq+\sum_{i\in\nodeSet}\piB\left(\externalQueue+1,\nodeQueueALL\right)\cdot\nodeServiceRate_{i}(\nodeQueue_{i})\cdot\routingProb(i,0)\cdot\routingProb(0,i)\cdot1_{\left\{ \nodeQueue_{i}>0\right\} }.
\end{align*}

There is no closed-form expression for the steady-state distribution
for the case $J>1$. Latouche and Ramaswami developed a logarithmic
reduction algorithm for the level-independent quasi-birth-and-death
process to compute the steady-state distribution (\citet{latouche;ramaswami:93},
\citep[Theorem 6.4.1 and Lemma 6.4.3, p. 142ff.]{latouche.ramaswami:99}).
For the case $J=1$, we calculate a closed-form expression for the
steady-state distribution, which we present in \prettyref{sec:Special-case}.

Next we determine the stability condition of the system. For that
we define the following traffic equation:
\begin{equation}
\etaRatio_{j}=\sum_{i\in\nodeSetNull}\etaRatio_{i}\cdot\routingProb(i,j),\qquad j\in\nodeSetNull.\label{eq:traffic-equation-LS}
\end{equation}
In matrix notation, the above equation can be expressed as $\etaRatioVect\cdot R=\etaRatioVect$
with $\etaRatioVect:=\left(\etaRatio_{j}\in\mathbb{R}^{+}:j\in\nodeSetNull\right)$. 

\citet{lavenberg-1978} showed that the system with backordering is
stable if the arrival rate $\arrival$ is smaller than the maximal
arrival rate $\maxArrival$ where $\maxArrival$ is the throughput
through node $0$ of a closed network in \prettyref{fig:SOQN-stability-network}.
In this network every resource which goes through node~$0$, spends
zero time at node $0$ and goes immediately to the next node according
to a branching vector $(\routingProb(0,j):j\in\nodeSet)$. \citeauthor{lavenberg-1978}
calls this network \emph{saturated} -- it is obtained when there
are infinitely many customers in the external queue.

We will call this network in \prettyref{fig:SOQN-stability-network}
\emph{stability} network. \citeauthor{lavenberg-1978} did not only
prove that the system is stable for $\arrival<\maxArrival$, he also
proved that it is unstable if $\arrival>$ $\maxArrival$, but he
did not show what will happen if $\arrival=\maxArrival$. In our stability
analysis we use matrix geometrical methods which cover all the cases
``$<$'', ``$=$'' and ``$>$''. 

In order to simplify notation we define the following constant:
\[
\normB(\nodeSet,\poolSize):=\sum_{\sum_{j\in\nodeSet}\nodeQueue_{j}=\poolSize}\prod_{j=1}^{\nodes}\left(\prod_{i=1}^{\nodeQueue_{j}}\frac{\etaRatio_{j}}{\nodeServiceRate_{j}(i)}\right).
\]

\begin{prop}
\label{prop:stability-condition}The system is stable if and only
if 
\[
\arrival<\maxArrival
\]
with
\begin{equation}
\maxArrival=\etaRatio_{0}\cdot\frac{\normB(\nodeSet,\poolSize-1)}{\normB(\nodeSet,\poolSize)}.\label{eq:max-arrival-from-norm-constants}
\end{equation}
\end{prop}

\begin{proof}
We use matrix-geometric methods to prove the stability condition.
According to \citet[Theorem 1]{latouche:10}: Given the irreducible
inter-level generator matrix $\mathbf{A}:=\mathbf{A}_{-1}+\mathbf{A}_{0}+\mathbf{A}_{1}$
of $Z_{\text{BO}}$ and the stochastic solution $\mathbf{\boldsymbol{\alpha}}:=(\alpha(\widetilde{\nodeQueueALL}):\widetilde{\nodeQueueALL}\in\phaseSpace)$
of the equation $\mathbf{\boldsymbol{\alpha}}\cdot\mathbf{A}=\mathbf{0}$,
then the process $Z_{\text{BO}}$ is stable if and only if
\begin{equation}
\mathbf{\boldsymbol{\alpha}}\cdot\mathbf{A}_{1}\cdot\mathbf{e}<\mathbf{\boldsymbol{\alpha}}\cdot\mathbf{A}_{-1}\cdot\mathbf{e}.\label{eq:stability-equation}
\end{equation}

Because $\mathbf{A}_{1}$ is a diagonal matrix with $\arrival$ on
its diagonal and $\mathbf{\boldsymbol{\alpha}}$ is a stochastic vector,
we have immediately for the left-hand side of \prettyref{eq:stability-equation}
\begin{equation}
\mathbf{\boldsymbol{\alpha}}\cdot\mathbf{A}_{1}\cdot\mathbf{e}=\arrival.\label{eq:left-hand-side-of-the-stability-equation}
\end{equation}
Therefore, in a stable system, the arrival rate $\arrival$ must be
strictly less than the right-hand side of inequation \prettyref{eq:stability-equation}.
We define
\[
\maxArrival:=\mathbf{\boldsymbol{\alpha}}\cdot\mathbf{A}_{-1}\cdot\mathbf{e}.
\]
To calculate $\maxArrival$, we first need to calculate the stochastic
vector $\mathbf{\boldsymbol{\alpha}}$. The non-negative non-diagonal
elements of the generator $\mathbf{A}$ are of the form
\begin{align*}
 & a\left(\left(0,\nodeQueue_{1},\ldots,\nodeQueue_{J}\right);\left(0,\nodeQueue_{1},\ldots,\nodeQueue_{i}-1,\ldots,\nodeQueue_{j}+1,\ldots,\nodeQueue_{J}\right)\right)\\
 & =\nodeServiceRate_{i}(\nodeQueue_{i})\cdot\Big(\routingProb(i,j)+\routingProb(i,0)\cdot\routingProb(0,j)\Big)\cdot1_{\left\{ \nodeQueue_{i}>0\right\} },
\end{align*}
for $i\neq j$. For its diagonal elements it holds
\[
a\left(z;z\right)=-\sum_{\substack{\tilde{z}\in E,\\
z\neq\tilde{z}
}
}a\left(z;\tilde{z}\right)\qquad\forall z\in E.
\]
We now solve for all $\widetilde{\nodeQueueALL}:=\left(0,\nodeQueue_{1},\ldots,\nodeQueue_{J}\right)\in\phaseSpace$
\begin{align}
 & \alpha\left(\widetilde{\nodeQueueALL}\right)\cdot\sum_{i\in\nodeSet}\nodeServiceRate_{i}(\nodeQueue_{i})\cdot\sum_{j\in\nodeSet}\Big(\routingProb(i,j)+\routingProb(i,0)\cdot\routingProb(0,j)\Big)\cdot1_{\left\{ \nodeQueue_{i}>0\right\} }\nonumber \\
 & =\sum_{j\in\nodeSet}\sum_{i\in\nodeSet}\alpha\left(\widetilde{\nodeQueueALL}+\evect_{i}-\evect_{j}\right)\cdot\nodeServiceRate_{i}(\nodeQueue_{i}+1)\cdot\Big(\routingProb(i,j)+\routingProb(i,0)\cdot\routingProb(0,j)\Big)\cdot1_{\left\{ \nodeQueue_{j}>0\right\} }.\label{eq:envirenment-generalized-gordon-newell-gbe}
\end{align}
\begin{figure}[h]
\centering{}\includegraphics[width=0.5\textwidth]{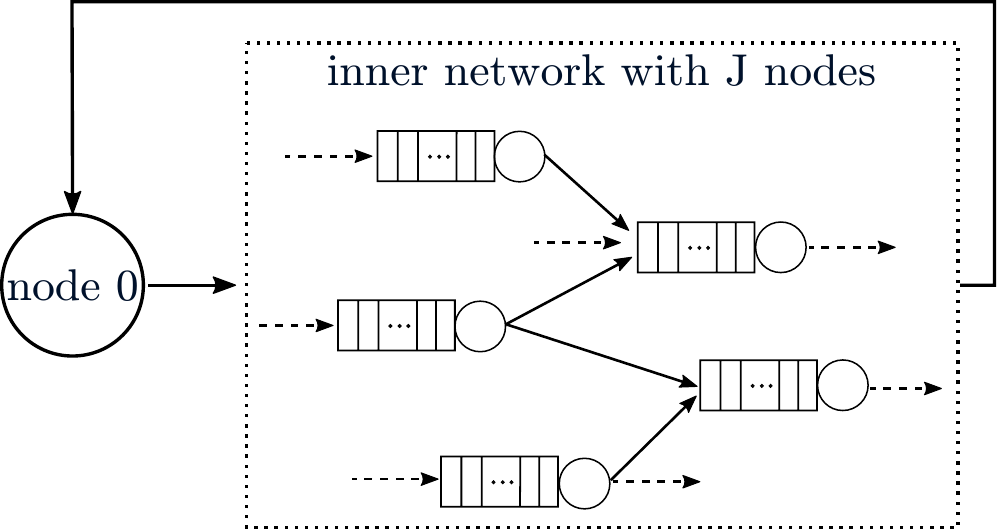}\caption{\label{fig:SOQN-stability-network}{\StabilityNetwork} -- a closed
network described by the inter-level generator matrix $\mathbf{A}$
from the proof of \prettyref{prop:stability-condition}.\protect \\
Interpretation I: a generalised Gordon-Newell network with zero service
time at node~$0$.\protect \\
Interpretation II: a classical Gordon-Newell network obtained after
rerouting at node~$0$.}
\end{figure}

Equation \eqref{eq:envirenment-generalized-gordon-newell-gbe} is
the global balance equation of a generalised Gordon-Newell network
with node set $\nodeSetNull$, $\poolSize$ customers and zero service
time at node $0$. We will call this network a \emph{\stabilityNetwork},
see \prettyref{fig:SOQN-stability-network}.

Equation \eqref{eq:envirenment-generalized-gordon-newell-gbe} also
has another important interpretation, which allows us to use standard
algorithms from performance analysis of Gordon-Newell networks. Note
that the state of node $0$ never changes, therefore we define
\begin{equation}
\alpha'(\nodeQueue_{1},\ldots,\nodeQueue_{J}):=\alpha\left(0,\nodeQueue_{1},\ldots,\nodeQueue_{J}\right)\label{eq:def-alpha-prime}
\end{equation}
and we define the routing matrix $\routingMatrix':=\left(\routingProb'(i,j):i,j\in\nodeSet\right)$
with
\begin{equation}
\routingProb'(i,j):=\routingProb(i,j)+\routingProb(i,0)\cdot\routingProb(0,j).\label{eq:jump-over-routing-matrix}
\end{equation}
Routing matrix $\routingMatrix'$ results from routing matrix $\routingMatrix$
when every resource which wants to go to node~$0$ skips this node
and goes immediately to the next node according to $(\routingProb(0,j):j\in\nodeSet)$.

Now equation \eqref{eq:envirenment-generalized-gordon-newell-gbe}
can be written as
\begin{align}
 & \alpha'\left(\nodeQueue_{1},\ldots,\nodeQueue_{J}\right)\cdot\sum_{i\in\nodeSet}\nodeServiceRate_{i}(\nodeQueue_{i})\sum_{j\in\nodeSet}\routingProb'(i,j)\cdot1_{\left\{ \nodeQueue_{i}>0\right\} }\nonumber \\
 & =\sum_{j\in\nodeSet}\sum_{i\in\nodeSet}\alpha'\left((\nodeQueue_{1},\ldots,\nodeQueue_{J})+\evect_{i}-\evect_{j}\right)\cdot\nodeServiceRate_{i}(\nodeQueue_{i}+1)\cdot\routingProb'(i,j)\cdot1_{\left\{ \nodeQueue_{j}>0\right\} }.\label{eq:envirenment-normal-gordon-newell-gbe}
\end{align}
\eqref{eq:envirenment-normal-gordon-newell-gbe} is the global balance
equation of a Gordon-Newell network with nodes $\nodeSet=\left\{ 1,2,\ldots,\nodes\right\} $,
with $\poolSize$ customers, and with routing matrix $\routingMatrix'$.

The steady-state distribution of this network is
\begin{equation}
\alpha'(\nodeQueue_{1},\ldots,\nodeQueue_{J})=\left[\normBstrich(\nodeSet,\poolSize)\right]^{-1}\prod_{j=1}^{\nodes}\left(\prod_{i=1}^{\nodeQueue_{j}}\frac{\etaRatio'_{j}}{\nodeServiceRate_{j}(i)}\right)\label{eq:stability-distribution-explicite-with-jumps}
\end{equation}
where $\etaRatioVect':=\left(\etaRatio'_{j}:j\in\nodeSet\right)$
is a solution of the traffic equation $\etaRatioVect'\cdot\routingMatrix'=\etaRatioVect'$
and
\[
\normBstrich(\nodeSet,\poolSize):=\sum_{\sum_{j\in\nodeSet}\nodeQueue_{j}=\poolSize}\prod_{j=1}^{\nodes}\left(\prod_{i=1}^{\nodeQueue_{j}}\frac{\etaRatio'_{j}}{\nodeServiceRate_{j}(i)}\right)
\]
is the normalisation constant. Because of the special structure \eqref{eq:jump-over-routing-matrix}
of matrix $\routingMatrix'$, $\etaRatio'_{j}:=\etaRatio_{j}$ for
all $j\in J$ is a solution of $\etaRatioVect'\cdot\routingMatrix'=\etaRatioVect'$,
see \citet[Proposition 2.1]{krenzler-daduna-otten:2016}. Consequently,
$\normBstrich(\nodeSet,\poolSize)=\normB(\nodeSet,\poolSize)$. Now
we can switch between both interpretations without recalculating $\etaRatioVect'$
and $\normBstrich(\nodeSet,\poolSize)$:
\[
\alpha'(\nodeQueue_{1},\ldots,\nodeQueue_{J})=\left[\normB(\nodeSet,\poolSize)\right]^{-1}\prod_{j=1}^{\nodes}\left(\prod_{i=1}^{\nodeQueue_{j}}\frac{\etaRatio{}_{j}}{\nodeServiceRate_{j}(i)}\right).
\]

We now calculate  $\maxArrival$ explicitly.\begingroup \allowdisplaybreaks
\begin{align*}
 & \phantomeq\maxArrival=\alpha\cdot\mathbf{A}_{-1}\cdot\mathbf{e}\\
 & =\sum_{(0,m_{1},\ldots,m_{J})\in\phaseSpace}\big(\alpha\cdot\mathbf{A}_{-1}\big)\left(0,m_{1},\ldots,m_{J}\right)\\
 & =\sum_{(0,m_{1},\ldots,m_{J})\in\phaseSpace}\Bigg[\sum_{\left(0,\nodeQueue_{1},\ldots,\nodeQueue_{J}\right)\in E_{1}}\alpha\left(0,\nodeQueue_{1},\ldots,\nodeQueue_{J}\right)\cdot a_{-1}\left(\left(0,\nodeQueue_{1},\ldots,\nodeQueue_{J}\right);\left(0,m_{1},\ldots,m_{J}\right)\right)\Bigg]\\
 & =\sum_{\left(0,\nodeQueue_{1},\ldots,\nodeQueue_{J}\right)\in\phaseSpace}\alpha\left(0,\nodeQueue_{1},\ldots,\nodeQueue_{J}\right)\sum_{(0,m_{1},\ldots,m_{J})\in\phaseSpace}a_{-1}\left(\left(0,\nodeQueue_{1},\ldots,\nodeQueue_{J}\right);\left(0,m_{1},\ldots,m_{J}\right)\right)\\
 & =\sum_{\left(0,\nodeQueue_{1},\ldots,\nodeQueue_{J}\right)\in\phaseSpace}\alpha\left(0,\nodeQueue_{1},\ldots,\nodeQueue_{J}\right)\cdot\Bigg[\sum_{i=1}^{\nodes}\sum_{j=1}^{\nodes}\nodeServiceRate_{i}(\nodeQueue_{i})\cdot\routingProb(i,0)\cdot\routingProb(0,j)\cdot1_{\left\{ \nodeQueue_{i}>0\right\} }\Bigg]\\
 & =\sum_{\left(0,\nodeQueue_{1},\ldots,\nodeQueue_{J}\right)\in\phaseSpace}\underbrace{\alpha\left(0,\nodeQueue_{1},\ldots,\nodeQueue_{J}\right)}_{=\alpha'(\nodeQueue_{1},\ldots,\nodeQueue_{J})}\cdot\Bigg[\sum_{i=1}^{\nodes}\nodeServiceRate_{i}(\nodeQueue_{i})\cdot\routingProb(i,0)\cdot1_{\left\{ \nodeQueue_{i}>0\right\} }\cdot\underbrace{\sum_{j=1}^{\nodes}\routingProb(0,j)}_{=1}\Bigg]\\
 & =\sum_{\left(\nodeQueue_{1},\ldots,\nodeQueue_{J}\right)\in\phaseSpace}\alpha'(\nodeQueue_{1},\ldots,\nodeQueue_{J})\cdot\Bigg[\sum_{i=1}^{\nodes}\nodeServiceRate_{i}(\nodeQueue_{i})\cdot1_{\left\{ \nodeQueue_{i}>0\right\} }\cdot\routingProb(i,0)\Bigg]\\
 & =\sum_{i=1}^{\nodes}\Bigg[\underbrace{\sum_{\nodeQueue_{i}=0}^{\poolSize}\sum_{\stackrel{\nodeQueue_{j}\in\left\{ 0,\ldots,\poolSize\right\} ,\ j\in\nodeSet\backslash\left\{ i\right\} }{\sum_{j\in\nodeSet\backslash\left\{ i\right\} }\nodeQueue_{j}=\poolSize-\nodeQueue_{i}}}\alpha'(\nodeQueue_{1},\ldots,\nodeQueue_{J})\cdot\nodeServiceRate_{i}(\nodeQueue_{i})\cdot1_{\left\{ \nodeQueue_{i}>0\right\} }}_{(*)}\Bigg]\cdot\routingProb(i,0).
\end{align*}
\endgroup

The expression ($*$) is the throughput through node $i$ in the Gordon-Newell
network with routing matrix $\routingMatrix'$:
\[
\stabilityNetworkTh i(\poolSize):=\sum_{\nodeQueue_{i}=0}^{\poolSize}\sum_{\stackrel{\nodeQueue_{j}\in\left\{ 0,\ldots,\poolSize\right\} ,\ j\in\nodeSet\backslash\left\{ i\right\} }{\sum_{j\in\nodeSet\backslash\left\{ i\right\} }\nodeQueue_{j}=\poolSize-\nodeQueue_{i}}}\alpha'(\nodeQueue_{1},\ldots,\nodeQueue_{J})\cdot\nodeServiceRate_{i}(\nodeQueue_{i})\cdot1_{\left\{ \nodeQueue_{i}>0\right\} },\qquad i\in\nodeSet.
\]
We can now write
\begin{equation}
\maxArrival=\sum_{i=1}^{\nodes}\stabilityNetworkTh i(\poolSize)\cdot\routingProb(i,0).\label{eq:rhs-stability-with-thi}
\end{equation}

According to \citet[p. 374, (8.14)]{Bolch:1998:QNM:289350}
\[
\stabilityNetworkTh i(\poolSize)=\etaRatio_{i}\cdot\frac{\normB(\nodeSet,\poolSize-1)}{\normB(\nodeSet,\poolSize)},
\]
therefore,
\begin{align*}
\maxArrival & =\sum_{i=1}^{\nodes}\etaRatio_{i}\cdot\frac{\normB(\nodeSet,\poolSize-1)}{\normB(\nodeSet,\poolSize)}\cdot\routingProb(i,0)=\frac{\normB(\nodeSet,\poolSize-1)}{\normB(\nodeSet,\poolSize)}\underbrace{\sum_{i=1}^{\nodes}\etaRatio_{i}\cdot\routingProb(i,0)}_{=\etaRatio_{0}}.
\end{align*}
\end{proof}
\begin{rem}
\label{rem:lambdamax-th0}The right-hand side of equation \eqref{eq:rhs-stability-with-thi}
is the throughput through node $0$ in the {\stabilityNetwork} in
\prettyref{fig:SOQN-stability-network}. Formally, this means
\begin{equation}
\maxArrival=\stabilityNetworkTh 0(\poolSize)\text{ with }\stabilityNetworkTh 0(\poolSize):=\sum_{i=1}^{\nodes}\stabilityNetworkTh i(\poolSize)\cdot\routingProb(i,0).\label{eq:th0-by-thi}
\end{equation}
The advantage of representation \eqref{eq:th0-by-thi} is that it
uses throughputs $\stabilityNetworkTh i(\poolSize)$, $i\in\nodeSet$,
of a classical Gordon-Newell network with routing matrix $\routingMatrix'$.
We can calculate these throughputs very efficiently with standard
methods, like, for example, mean value analysis (MVA).

With \eqref{eq:max-arrival-from-norm-constants} we have another representation
of $\stabilityNetworkTh 0(\poolSize)$:
\begin{equation}
\stabilityNetworkTh 0(\poolSize)=\etaRatio_{0}\cdot\frac{\normB(\nodeSet,\poolSize-1)}{\normB(\nodeSet,\poolSize)}.\label{eq:stable-th0-by-thi}
\end{equation}

To calculate efficiently the constants on the right-hand side of \eqref{eq:stable-th0-by-thi},
we can use, for example, the convolution algorithm.

Both algorithms are illustrated in \citet[p. 371ff., Section 8.1 and p. 384ff., Section 8.2]{Bolch:1998:QNM:289350}.

From a practical point of view, it is important to know the long-term
behaviour of the system. This behaviour is equal to the behaviour
of the system in steady state. We analyse it in the following sections,
\ref{sec:throughputs-and-idle-times-LC} and \ref{sec:Special-case}.
\end{rem}

\subsection{Throughputs and idle times\label{sec:throughputs-and-idle-times-LC}}

We consider a stable SOQN-BO in steady state. Let $\etaRatioVect:=(\etaRatio_{j}:j\in\nodeSetNull)$
be a solution of $\boldsymbol{x}=\boldsymbol{x}\cdot\routingMatrix$.
Recall that $\etaRatioVect$ is uniquely defined up to a constant.
\begin{prop}
\label{prop:THL-BO} The local throughput at the nodes $j\in\nodeSetNull$
is
\begin{equation}
\ThBO j=\arrival\cdot\frac{\etaRatio_{j}}{\etaRatio_{0}}.\label{eq:TH-BO1}
\end{equation}
\end{prop}

Note, formula \eqref{eq:TH-BO1} occurs as an approximation in a different
setting in \citep{dallery:1990} as (26). We give a proof for correctness
in our setting.
\begin{proof}
We define for $j\in\nodeSetNull$ in steady state:
\begin{itemize}
\item the mean number of departures from $j$ per time unit is $D_{j}$,
and
\item the mean number of arrivals at $j$ per time unit is $\meanVisits j$.
\end{itemize}
From steady state assumption follows
\[
\ThBO j=\meanVisits j=D_{j}.
\]
For any $j\in\nodeSetNull$ holds
\[
\meanVisits j=\sum_{i\in\nodeSetNull}D_{i}\cdot\routingProb(i,j).
\]
Therefore, the vector $\meanVisits{}:=(\meanVisits j:j\in\nodeSetNull)$
fulfils the set of equations
\[
\meanVisits j=\sum_{i\in\nodeSetNull}\meanVisits i\cdot\routingProb(i,j),\quad j\in\nodeSetNull,~~~\text{which is }~~~\meanVisits{}=\meanVisits{}\cdot\routingMatrix.
\]
This implies that $\meanVisits j=\etaRatio_{j}\cdot K$ for some constant
$K>0$.

Because of $\arrival=\meanVisits 0=\etaRatio_{0}\cdot K$ we have
\[
K=\frac{\arrival}{\etaRatio_{0}},
\]
and therefore,
\[
\meanVisits j=\arrival\cdot\frac{\etaRatio_{j}}{\etaRatio_{0}},\quad j\in\nodeSetNull.
\]
\end{proof}
Note, that \eqref{eq:TH-BO1} shows that $\ThBO j$ does not depend
on normalisation of $\etaRatioVect$.
\begin{cor}
\label{cor:idle-times-BO}Let $Y_{\text{BO}}:=(Y_{\text{BO},j}:j\in\nodeSet)$
denote a random vector which is distributed according to the stationary
queue length at the nodes in $\nodeSet$ of the SOQN-BO. If the service
rate at node $j$ does not depend on the queue length, i.e. $\nodeServiceRate_{j}(\cdot)=\nodeServiceRate_{j},j\in\nodeSet$,
then the probability that node $j$ is idling is
\[
P(Y_{\text{BO},j}=0)=1-\arrival\cdot\frac{\etaRatio_{j}}{\etaRatio_{0}}\cdot\nodeServiceRate_{j}^{-1}.
\]
Note, that the formula above also describes the proportion of time
that node $j$ is idle.
\end{cor}

\pagebreak{}
\begin{proof}
We define for $j\in\nodeSet$ in steady state:
\begin{itemize}
\item the mean number of customers in service is $B_{j}$,
\item the mean service time is $S_{j}$, and
\item the arrival intensity is $\lambda_{j}$ . 
\end{itemize}
According to Little's formula, for every node $j$ it holds $B_{j}=\lambda_{j}\cdot S_{j}$. 

In steady state the arrival rate $\lambda_{j}$ at node $j$ is equal
to its throughput $\ThBO j$. Hence, from \prettyref{prop:THL-LC},
we know that $\lambda_{j}=\ThBO j=\arrival\cdot\frac{\etaRatio_{j}}{\etaRatio_{0}}$.

In every node $j$, with constant service rate $\nodeServiceRate_{j}$,
the mean service time $S_{j}$ is $\nodeServiceRate_{j}^{-1}$.

Now we substitute these known results for $\lambda_{j}$ and $S_{j}$
in Little's formula and get for the mean number of customers in service
\[
B_{j}=\arrival\cdot\frac{\etaRatio_{j}}{\etaRatio_{0}}\cdot\nodeServiceRate_{j}^{-1}.
\]
Consequently, the probability that node $j$ is idling is
\[
P(Y_{\text{BO},j}=0)=1-P(Y_{\text{BO},j}>0)=1-E\left[1_{\left\{ Y_{\text{BO},j}>0\right\} }\right]=1-B_{j}=1-\arrival\cdot\frac{\etaRatio_{j}}{\etaRatio_{0}}\cdot\nodeServiceRate_{j}^{-1}.
\]
\end{proof}

\subsection{Special case: $\protect\nodes=1$\label{sec:Special-case}}

In this section, we consider the special case where the inner network
consists only of one node ($J=1$) as shown in \prettyref{fig:SOQN-small-1}.
For this special case, \citeauthor{avi-itzhak-heyman-1973} calculated
the steady-state distribution in \citep{avi-itzhak-heyman-1973}.
\begin{figure}[h]
\centering{}\includegraphics[width=0.65\textwidth]{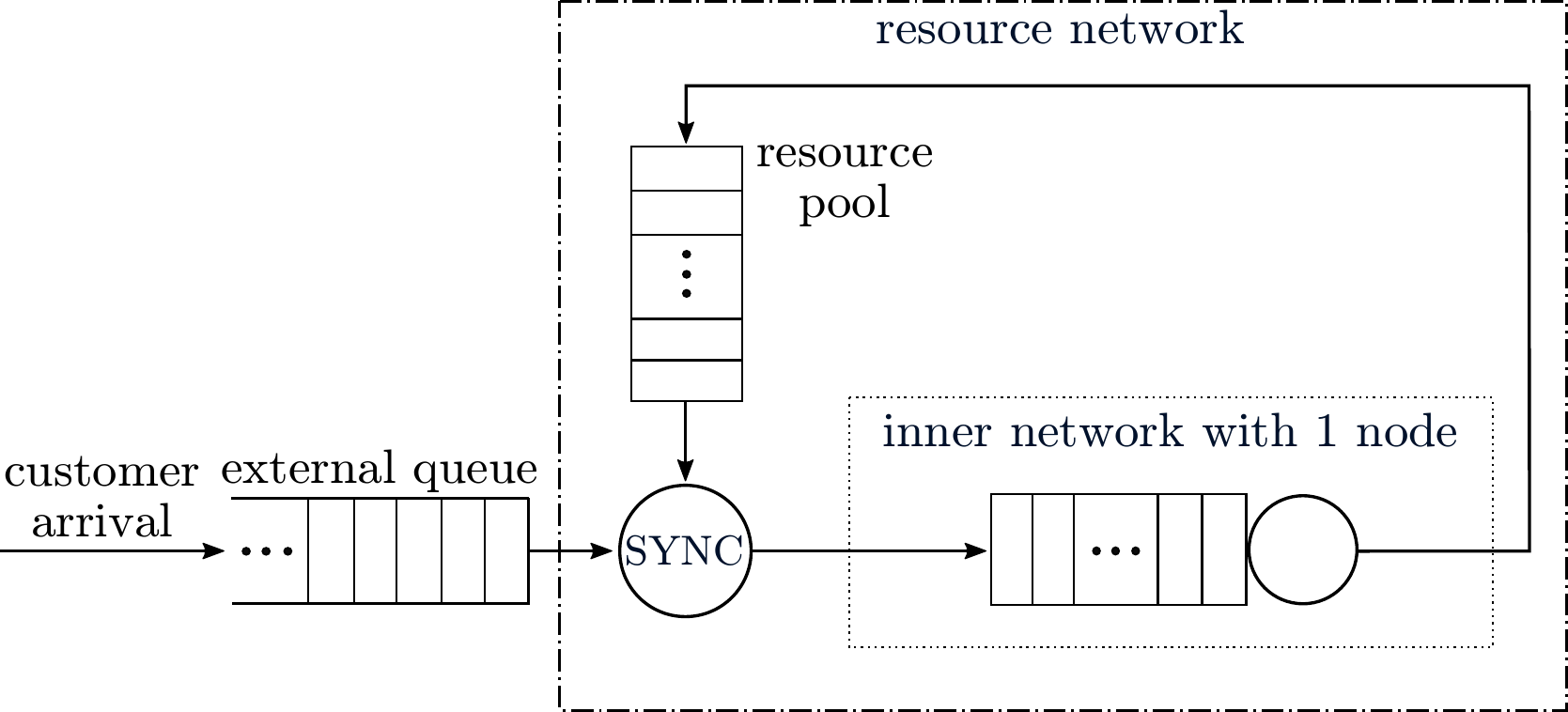}\caption{\label{fig:SOQN-small-1}An SOQN-BO with $\protect\nodes=1$.}
\end{figure}

\begin{thm}
For $\nodes=1$, $Z_{\text{BO}}$ is stable if and only if $\arrival<\nodeServiceRate_{1}(\poolSize)$.
If $Z_{\text{BO}}$ is stable, the limiting and steady-state distribution
$\boldsymbol{\pi}_{\text{BO}}:=\left(\piB\left((\externalQueue,\nodeQueue_{0},\nodeQueue_{1})\right):(\externalQueue,\nodeQueue_{0},\nodeQueue_{1})\in E\right)$
of the process $Z_{\text{BO}}$ is
\begin{equation}
\piB(\externalQueue,\nodeQueue_{0},\nodeQueue_{1})=\piB(\externalQueue,\poolSize-\nodeQueue_{1},\nodeQueue_{1})=\left[\specialNormB(\left\{ 1\right\} ,\poolSize)\right]^{-1}\cdot\prod_{m=1}^{\nodeQueue_{1}}\frac{\arrival}{\nodeServiceRate_{1}(m)}\cdot\left(\frac{\arrival}{\nodeServiceRate_{1}(\poolSize)}\right)^{\externalQueue}\label{eq:MODEL-ALL-TOGETHER-stationary-distribution-1}
\end{equation}
with normalisation constant
\begin{equation}
\specialNormB(\left\{ 1\right\} ,\poolSize)\coloneqq\sum_{m=0}^{\poolSize-1}\prod_{\ell=1}^{m}\frac{\arrival}{\nodeServiceRate_{1}(\ell)}+\prod_{\ell=1}^{\poolSize}\frac{\arrival}{\nodeServiceRate_{1}(\ell)}\cdot\left(\frac{1}{1-\frac{\arrival}{\nodeServiceRate_{1}(\poolSize)}}\right).\label{eq:Modell-all-together-normierungskonstante}
\end{equation}
\end{thm}

\begin{proof}
See \citep[equations (20) and (21)]{avi-itzhak-heyman-1973}. In that
paper, the authors calculated the steady-state probabilities for $P\left(\externalQueueProc+\nodeProc_{1}=m\right)$.
From these results we can easily obtain the probabilities of the states
of all queues, because for $m\leq N$ it holds $\externalQueueProc+\nodeProc_{1}=m\Leftrightarrow\externalQueueProc=0\land\nodeProc_{1}=m\land Y_{0}=N-m$
and otherwise $\externalQueueProc+\nodeProc_{1}=m\Leftrightarrow\externalQueueProc=m-N\land\nodeProc_{1}=N\land Y_{0}=0$.
\end{proof}
\begin{prop}
\label{prop:marginal-distribution-of-a-simple-system}Let $(\externalQueueProcRV,\nodeProcRV_{0},\nodeProcRV_{1})$
denote a random vector which is distributed according to the limiting
and steady-state distribution of $Z_{\text{BO}}$ in the SOQN-BO with
$\nodes=1$.
\begin{enumerate}
\item [(i)]For the marginal distributions it holds:
\begin{equation}
P(\externalQueueProcRV=0)=\left[\specialNormB(\left\{ 1\right\} ,\poolSize)\right]^{-1}\cdot\sum_{\nodeQueue_{1}=0}^{\robotnumb}\prod_{m=1}^{\nodeQueue_{1}}\frac{\arrival}{\nodeServiceRate_{1}(m)}\label{eq:eq:small-backordering-p-X-ex-0}
\end{equation}
and for $\externalQueue>0$
\begin{equation}
P(\externalQueueProcRV=\externalQueue)=\left[\specialNormB(\left\{ 1\right\} ,\poolSize)\right]^{-1}\cdot\prod_{m=1}^{\poolSize}\frac{\arrival}{\nodeServiceRate_{1}(m)}\cdot\left(\frac{\arrival}{\nodeServiceRate_{1}(\poolSize)}\right)^{\externalQueue}.\label{eq:small-backordering-p-X-ex-n}
\end{equation}
For $0\leq\nodeQueue_{1}<\poolSize$
\begin{equation}
P(\nodeProcRV_{0}=\poolSize-\nodeQueue_{1},\nodeProcRV_{1}=\nodeQueue_{1})=\left[\specialNormB(\left\{ 1\right\} ,\poolSize)\right]^{-1}\cdot\prod_{m=1}^{\nodeQueue_{1}}\frac{\arrival}{\nodeServiceRate_{1}(m)}\label{eq:small-backordering-p-Y-n}
\end{equation}
and for $\nodeQueue_{1}=\poolSize$
\begin{equation}
P(\nodeProcRV_{0}=0,\nodeProcRV_{1}=\poolSize)=\left[\specialNormB(\left\{ 1\right\} ,\poolSize)\right]^{-1}\cdot\prod_{m=1}^{\poolSize}\frac{\arrival}{\nodeServiceRate_{1}(m)}\cdot\frac{1}{1-\frac{\arrival}{\nodeServiceRate_{1}(\poolSize)}}.\label{eq:randvtlg-y}
\end{equation}
\item [(ii)]The average external queue length is
\begin{align}
\externalLength & =P(\nodeProcRV_{0}=0,\nodeProcRV_{1}=\poolSize)\cdot\frac{\arrival}{\nodeServiceRate_{1}(\poolSize)-\arrival}\label{eq:small-backordering-L-ex-0}
\end{align}
and the average waiting time of customers in the external queue is
\begin{equation}
\externalWaitingTime=\frac{\externalLength}{\arrival}=P(\nodeProcRV_{0}=0,\nodeProcRV_{1}=\poolSize)\cdot\frac{1}{\nodeServiceRate_{1}(\poolSize)-\arrival}.\label{eq:waiting-time-ex}
\end{equation}
\end{enumerate}
\end{prop}

\begin{proof}
$ $
\begin{enumerate}
\item [(i)]For \eqref{eq:eq:small-backordering-p-X-ex-0} we calculate
\begin{align*}
P(\externalQueueProcRV=0) & \overset{\hphantom{\eqref{eq:MODEL-ALL-TOGETHER-stationary-distribution-1}}}{=}\sum_{\nodeQueue_{1}=0}^{\robotnumb}P(\externalQueueProcRV=0,\nodeProcRV_{0}=\poolSize-\nodeQueue_{1},\nodeProcRV_{1}=\nodeQueue_{1})\\
 & \overset{\eqref{eq:MODEL-ALL-TOGETHER-stationary-distribution-1}}{=}\left[\specialNormB(\left\{ 1\right\} ,\poolSize)\right]^{-1}\cdot\sum_{\nodeQueue_{1}=0}^{\robotnumb}\prod_{m=1}^{\nodeQueue_{1}}\frac{\arrival}{\nodeServiceRate_{1}(m)}.
\end{align*}
For \eqref{eq:small-backordering-p-X-ex-n} we calculate
\begin{align*}
P(\externalQueueProcRV=\externalQueue) & \overset{\hphantom{\eqref{eq:MODEL-ALL-TOGETHER-stationary-distribution-1}}}{=}P(\externalQueueProcRV=\externalQueue,\nodeProcRV_{0}=0,\nodeProcRV_{1}=\poolSize)\\
 & \overset{\eqref{eq:MODEL-ALL-TOGETHER-stationary-distribution-1}}{=}\left[\specialNormB(\left\{ 1\right\} ,\poolSize)\right]^{-1}\cdot\prod_{m=1}^{\poolSize}\frac{\arrival}{\nodeServiceRate_{1}(m)}\cdot\left(\frac{\arrival}{\nodeServiceRate_{1}(\poolSize)}\right)^{\externalQueue}.
\end{align*}
For \eqref{eq:small-backordering-p-Y-n} we calculate
\begin{align*}
P(\nodeProcRV_{0}=\poolSize-\nodeQueue_{1},\nodeProcRV_{1}=\nodeQueue_{1}) & \overset{\hphantom{\eqref{eq:MODEL-ALL-TOGETHER-stationary-distribution-1}}}{=}P(\externalQueueProcRV=0,\nodeProcRV_{0}=\poolSize-\nodeQueue_{1},\nodeProcRV_{1}=\nodeQueue_{1})\\
 & \overset{\eqref{eq:MODEL-ALL-TOGETHER-stationary-distribution-1}}{=}\left[\specialNormB(\left\{ 1\right\} ,\poolSize)\right]^{-1}\cdot\prod_{m=1}^{\nodeQueue_{1}}\frac{\arrival}{\nodeServiceRate_{1}(m)}.
\end{align*}
For \eqref{eq:randvtlg-y} we calculate
\begin{align*}
P(\nodeProcRV_{0}=0,\nodeProcRV_{1}=\poolSize) & \overset{\hphantom{\eqref{eq:MODEL-ALL-TOGETHER-stationary-distribution-1}}}{=}\sum_{\externalQueue=0}^{\infty}P(\externalQueueProcRV=\externalQueue,\nodeProcRV_{0}=0,\nodeProcRV_{1}=\poolSize)\\
 & \overset{\eqref{eq:MODEL-ALL-TOGETHER-stationary-distribution-1}}{=}\left[\specialNormB(\left\{ 1\right\} ,\poolSize)\right]^{-1}\cdot\prod_{m=1}^{\poolSize}\frac{\arrival}{\nodeServiceRate_{1}(m)}\cdot\sum_{\externalQueue=0}^{\infty}\left(\frac{\arrival}{\nodeServiceRate_{1}(\poolSize)}\right)^{\externalQueue}.
\end{align*}
\item [(ii)]We use the marginal distributions in (i) to calculate the
average external queue length in steady state.
\begin{align*}
\externalLength & =\sum_{\externalQueue=0}^{\infty}\externalQueue\cdot P(\externalQueueProcRV=\externalQueue)=\left[\specialNormB(\left\{ 1\right\} ,\poolSize)\right]^{-1}\cdot\prod_{m=1}^{\poolSize}\frac{\arrival}{\nodeServiceRate_{1}(m)}\cdot\sum_{\externalQueue=1}^{\infty}\externalQueue\cdot\left(\frac{\arrival}{\nodeServiceRate_{1}(\poolSize)}\right)^{\externalQueue}\\
 & =\left[\specialNormB(\left\{ 1\right\} ,\poolSize)\right]^{-1}\cdot\prod_{m=1}^{\poolSize}\frac{\arrival}{\nodeServiceRate_{1}(m)}\cdot\frac{\frac{\arrival}{\nodeServiceRate_{1}(\poolSize)}}{\left(1-\frac{\arrival}{\nodeServiceRate_{1}(\poolSize)}\right)^{2}}\overset{\eqref{eq:randvtlg-y}}{=}P(\nodeProcRV_{0}=0,\nodeProcRV_{1}=\poolSize)\cdot\frac{\frac{\arrival}{\nodeServiceRate_{1}(\poolSize)}}{1-\frac{\arrival}{\nodeServiceRate_{1}(\poolSize)}}\\
 & =P(\nodeProcRV_{0}=0,\nodeProcRV_{1}=\poolSize)\cdot\frac{\frac{\arrival}{\nodeServiceRate_{1}(\poolSize)}}{\frac{\nodeServiceRate_{1}(\poolSize)-\arrival}{\nodeServiceRate_{1}(\poolSize)}}=P(\nodeProcRV_{0}=0,\nodeProcRV_{1}=\poolSize)\cdot\frac{\arrival}{\nodeServiceRate_{1}(\poolSize)-\arrival}.
\end{align*}
Equation \eqref{eq:waiting-time-ex} follows from Little's law, see,
for example, \citet{little-graves:2008}.
\end{enumerate}
\end{proof}
An important practical question about SOQN systems in this paper is:
``Do more resources allow more system throughput?'' \prettyref{prop:monoton-for-1}
shows that, even for the simple system with $\nodes=1$, the surprising
answer is: ``It depends.'' We also know that in systems with $\nodes\geq1$
whose service rates are non-decreasing in the number of customers,
more resources lead to equal or higher throughput.
\begin{prop}
\label{prop:monoton-for-more-1}Let $\nodes\geq1$ and all service
rates be non-decreasing in the number of customers, i.e.~is $\nodeServiceRate_{j}(n+1)\geq\nodeServiceRate_{j}(n)$,
$n\in\{0,\ldots,N-1\}$, for all $j\in\nodeSet$. Then $\maxArrival=\stabilityNetworkTh 0(\cdot)$
is non-decreasing on $\mathbb{N}$.
\end{prop}

\begin{proof}
See Van der Wal \Citep{vanderWal1989}.
\end{proof}
\begin{prop}
\label{prop:monoton-for-1}Let $\nodes=1$. Then the following are
equivalent:
\begin{itemize}
\item [(i)]$\maxArrival=\stabilityNetworkTh 0(\cdot)$ is non-decreasing
on $\mathbb{N}$.
\item [(ii)]$\nodeServiceRate_{1}$ is non-decreasing on $\mathbb{N}$.
\end{itemize}
\end{prop}

\begin{proof}
Because of \eqref{eq:max-arrival-from-norm-constants} and \eqref{eq:th0-by-thi},
(i) is equivalent to
\begin{equation}
\forall\;\poolSize\in\mathbb{N}:\;\normB(\left\{ 1\right\} ,\poolSize-1)\cdot\normB(\left\{ 1\right\} ,\poolSize+1)\leq\normB(\left\{ 1\right\} ,\poolSize)^{2}.\label{eq:1}
\end{equation}
We note that
\begin{align*}
\normB(\left\{ 1\right\} ,\poolSize) & =\prod_{i=1}^{N}\frac{\etaRatio_{1}}{\nodeServiceRate_{1}(i)},\\
\normB(\left\{ 1\right\} ,\poolSize+1) & =\normB(\left\{ 1\right\} ,\poolSize)\cdot\frac{\etaRatio_{1}}{\nodeServiceRate_{1}(\poolSize+1)},\\
\normB(\left\{ 1\right\} ,\poolSize-1) & =\normB(\left\{ 1\right\} ,\poolSize)\cdot\frac{\nodeServiceRate_{1}(\poolSize)}{\etaRatio_{1}}
\end{align*}
implying that $\prettyref{eq:1}$ is equivalent to
\[
\forall\;\poolSize\in\mathbb{N}:\;\frac{\nodeServiceRate_{1}(\poolSize)}{\nodeServiceRate_{1}(\poolSize+1)}\cdot\normB(\left\{ 1\right\} ,\poolSize)^{2}\leq\normB(\left\{ 1\right\} ,\poolSize)^{2}
\]
and because $\normB(\left\{ 1\right\} ,\poolSize)^{2}>0$, it is equivalent
to
\[
\forall\;\poolSize\in\mathbb{N}:\;\nodeServiceRate_{1}(\poolSize)\leq\nodeServiceRate_{1}(\poolSize+1),
\]
which is (ii).
\end{proof}

\section{Approximation of the resource network\label{sec:Approximation}}

Unfortunately, the steady-state distribution of the SOQN-BO for $J>1$
is not known. However, for a modified system, we get closed-form expressions
and product-form results for $\nodes\geq1$, as shown in the next
section.

\subsection{SOQN {\withLostCustomers}\label{sec:Lost-network}}

In this section, we consider a modification of the model from \prettyref{sec:RMFS-general-model},
where newly arriving customers are lost if the resource pool is empty.
Henceforth, we call it \emph{SOQN-LC}. In the literature, this behaviour
is called ``lost customers'', ``lost sales'', ``lost arrivals'' or
``loss systems''. These customers arrive one by one according to a
Poisson process with rate $\arrivalLS>0$. But because of lost customers,
the actual arrival rate to the synchronisation node is smaller. We
call this effective arrival rate $\effarrival(\arrivalLS)$. The SOQN-LC
is on the right of \prettyref{fig:SOQN-lostsales}, and the original
SOQN-BO from \prettyref{sec:RMFS-general-model} is on the left of
\prettyref{fig:SOQN-lostsales}. 

It is well known that such an SOQN-LC turns formally into a closed
Gordon-Newell network which consists only of the resource network,
see \citep[p. 21]{yao2001fundamentals}.

\begin{figure}[h]
\begin{centering}
\includegraphics[width=1\textwidth]{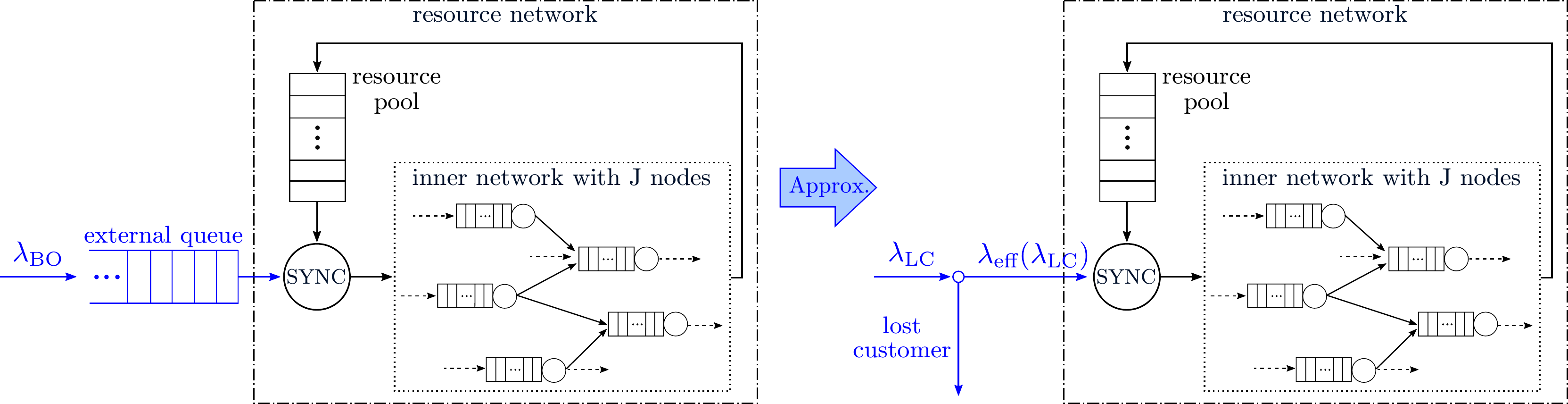}
\par\end{centering}
\centering{}\caption{\label{fig:SOQN-lostsales}Transition from SOQN-BO to SOQN-LC.}
\end{figure}

To obtain a Markovian process description, we denote by $\nodeProc_{0}(t)$
the number of resources in the resource pool at time $t\geq0$ and
by $\nodeProc_{j}(t)$, $j\in\nodeSet$, the number of resources present
at node $j$ in the inner network at time $t\geq0$, either waiting
or in service. We call this number queue length at node $j\in\nodeSet$.

Then $\nodeProcALL(t):=\left(\nodeProc_{j}(t):j\in\nodeSetNull\right)$
is the local queue length vector of the resource network at time $t\geq0$.
We define the joint queue length process of the semi-open network
by
\[
Z_{\text{LC}}:=\left(\nodeProcALL(t):t\geq0\right).
\]
Then, due to the usual independence and memorylessness assumptions
(see the assumptions \vpageref{independence-memorylessness}), $Z_{\text{LC}}$
is a homogeneous Markov process, which is irreducible on the state
space
\begin{align*}
\lostStateSpace & :=\big\{\left(\nodeQueue_{0},\nodeQueue_{1},\ldots,\nodeQueue_{J}\right):\:\nodeQueue_{j}\in\left\{ 0,\ldots,\poolSize\right\} \:\forall j\in\nodeSetNull,\:\sum_{j\in\nodeSetNull}\nodeQueue_{j}=\poolSize\big\}.
\end{align*}

For such a system, closed-form expressions for the steady-state distribution\\
$\boldsymbol{\pi}_{\text{LC}}:=\left(\piLS\left(\nodeQueueALL\right):\nodeQueueALL\in\lostStateSpace\right)$
in product form are available, see \citet[p. 22, Theorem 2.5]{yao2001fundamentals}.
It is for $\nodeQueueALL:=\left(\nodeQueue_{j}:j\in\nodeSetNull\right)\in\lostStateSpace$
\begin{equation}
\piLS\left(\nodeQueueALL\right)=\left[\normLS(\nodeSetNull,\poolSize)\right]^{-1}\cdot\left(\frac{\etaRatio_{0}}{\arrivalLS}\right)^{\nodeQueue_{0}}\cdot\prod_{j=1}^{\nodes}\left(\prod_{i=1}^{\nodeQueue_{j}}\frac{\etaRatio_{j}}{\nodeServiceRate_{j}(i)}\right)\label{eq:BO-GEN-APROX-stationary-distribution-1}
\end{equation}
with normalisation constant
\begin{equation}
\normLS(\nodeSetNull,\poolSize):=\sum_{\sum_{j\in\nodeSetNull}\nodeQueue_{j}=\poolSize}\left(\frac{\etaRatio_{0}}{\arrivalLS}\right)^{\nodeQueue_{0}}\cdot\prod_{j=1}^{\nodes}\left(\prod_{i=1}^{\nodeQueue_{j}}\frac{\etaRatio_{j}}{\nodeServiceRate_{j}(i)}\right).\label{eq:BO-GEN-normalize}
\end{equation}

\subsection{Adjustment\label{sec:Adjustment}}

We want to use the modified system (SOQN-LC) with known results to
approximate the SOQN-BO. But before we do so, we have to ensure that
both systems process in the mean the same number of external customers.
That means they have the same throughput through synchronisation node
$0$. Our main idea is: In order to compensate customers' loss, we
will adjust the input rate $\arrivalLS$ of the {\similar} system
until it reaches the desired throughput. But is it even possible?
Yes, it is! We will prove it in the next \prettyref{thm:theorem-ls}.
But first we need to calculate the throughput of both systems.
\begin{lem}
\label{lem:effektive-throuput-with-lost-customrers}The throughput
of the SOQN-BO in steady state is $\arrival$. The throughput of the
SOQN-LC in steady state is
\[
\effarrival(\arrivalLS)=\arrivalLS\cdot\Bigg(1-\underbrace{\frac{\normB(\nodeSet,\poolSize)}{\normLS(\nodeSetNull,\poolSize)}}_{=\resourceProb(0)}\Bigg)
\]
where $\resourceProb(0)$ is the probability that there are no resources
in the resource pool.
\end{lem}

\begin{proof}
The throughput of the SOQN-BO in steady state is equal to the arrival
rate $\arrival$ because all customers pass through the system.

In contrast, in the SOQN-LC, the proportion $\resourceProb(0)$ of
the customers is lost.

From the steady-state distribution \prettyref{eq:BO-GEN-APROX-stationary-distribution-1},
we directly calculate for the SOQN-LC the probability $\resourceProb(0)$
of the resource pool being empty as
\begin{align}
\resourceProb(0) & :=\sum_{\sum_{j\in\nodeSet}\nodeQueue_{j}=\poolSize}\piLS\left(0,\nodeQueue_{1},\ldots,\nodeQueue_{J}\right)\overset{\prettyref{eq:BO-GEN-APROX-stationary-distribution-1}}{=}\frac{\normB(\nodeSet,\poolSize)}{\normLS(\nodeSetNull,\poolSize)}.\label{eq:BO-GEN-LS-theta-null-1}
\end{align}
Then the effective arrival rate $\effarrival(\arrivalLS)$ and the
throughput $\ThLC j$ of the system is
\[
\effarrival(\arrivalLS)=\ThLC j=\arrivalLS\cdot\left(1-\resourceProb(0)\right).
\]
\end{proof}
Now we can adjust $\arrivalLS$ in such a way that both systems have
the same throughput. We assume that both systems -- with backordering
and {\withLostCustomers} -- are stable. For the SOQN-BO, according
to \prettyref{prop:stability-condition}, stability is equivalent
to $\arrival\in\left(0,\maxArrival\right)$. For the SOQN-LC, stability
is granted for any arrival rate $\arrivalLS\in\left(0,\infty\right)$,
because the state space $\lostStateSpace$ is finite.
\begin{thm}
\label{thm:theorem-ls}For every stable SOQN-BO, there exists an SOQN-LC
with arrival rate $\arrivalLS$ such that both systems have the same
throughput in steady state. Formally, this means:
\[
\text{For all }\arrival\in\left(0,\maxArrival\right)\ \text{exists }\arrivalLS\in\left(0,\infty\right)\ \text{with}\ \effarrival(\arrivalLS)=\arrival,
\]
where $\maxArrival$ is given in \prettyref{eq:max-arrival-from-norm-constants}.
\end{thm}

\begin{proof}
The idea of the proof is to show that for any $\arrival\in\left(0,\maxArrival\right)$,
the function $\effarrival(\arrivalLS)$ from \prettyref{lem:effektive-throuput-with-lost-customrers}
can have values larger than a prescribed $\arrival$, and smaller
than $\arrival$, and is continuous. Thus, by the intermediate value
theorem, there exists $\arrivalLS$ such that $\effarrival(\arrivalLS)=\arrival$.

We first show that $\effarrival(\arrivalLS)=\arrivalLS\cdot\left(1-\frac{\normB(\nodeSet,\poolSize)}{\normLS(\nodeSetNull,\poolSize)}\right)$
can be larger than any given $\arrival$. We analyse the normalisation
constant $\normLS(\nodeSetNull,\poolSize)$:
\begin{align}
\normLS(\nodeSetNull,\poolSize) & =\sum_{\sum_{j\in\nodeSetNull}\nodeQueue_{j}=\poolSize}\left(\frac{\etaRatio_{0}}{\arrivalLS}\right)^{\nodeQueue_{0}}\cdot\prod_{j=1}^{\nodes}\left(\prod_{i=1}^{\nodeQueue_{j}}\frac{\etaRatio_{j}}{\nodeServiceRate_{j}(i)}\right)\nonumber \\
 & =\sum_{\nodeQueue_{0}=0}^{\poolSize}\left(\frac{\etaRatio_{0}}{\arrivalLS}\right)^{\nodeQueue_{0}}\cdot\sum_{\sum_{j\in\nodeSet}\nodeQueue_{j}=\poolSize-\nodeQueue_{0}}\prod_{j=1}^{\nodes}\left(\prod_{i=1}^{\nodeQueue_{j}}\frac{\etaRatio_{j}}{\nodeServiceRate_{j}(i)}\right)\nonumber \\
 & =\sum_{\nodeQueue_{0}=0}^{\poolSize}\left(\frac{\etaRatio_{0}}{\arrivalLS}\right)^{\nodeQueue_{0}}\cdot\normB(\nodeSet,\poolSize-\nodeQueue_{0}).\label{eq:C_LC-depends-from-CB}
\end{align}

To simplify the notation, we define constant $\abrevC{\nodeQueue_{0}}:=\normB(\nodeSet,\poolSize-\nodeQueue_{0})$.
Then $\effarrival(\arrivalLS)$ can be expressed as
\begin{align*}
\effarrival(\arrivalLS) & =\arrivalLS\cdot\left(1-\frac{\normB(\nodeSet,\poolSize)}{\normLS(\nodeSetNull,\poolSize)}\right)=\arrivalLS\cdot\left(1-\frac{\abrevC 0}{\sum_{\nodeQueue_{0}=0}^{\poolSize}\left(\frac{\etaRatio_{0}}{\arrivalLS}\right)^{\nodeQueue_{0}}\abrevC{\nodeQueue_{0}}}\right)\\
 & =\arrivalLS\cdot\left(\frac{\sum_{\nodeQueue_{0}=1}^{\poolSize}\left(\frac{\etaRatio_{0}}{\arrivalLS}\right)^{\nodeQueue_{0}}\abrevC{\nodeQueue_{0}}}{\sum_{\nodeQueue_{0}=0}^{\poolSize}\left(\frac{\etaRatio_{0}}{\arrivalLS}\right)^{\nodeQueue_{0}}\abrevC{\nodeQueue_{0}}}\right)=\etaRatio_{0}\cdot\left(\frac{\sum_{\nodeQueue_{0}=1}^{\poolSize}\left(\frac{\etaRatio_{0}}{\arrivalLS}\right)^{\nodeQueue_{0}-1}\abrevC{\nodeQueue_{0}}}{\abrevC 0+\sum_{\nodeQueue_{0}=1}^{\poolSize}\left(\frac{\etaRatio_{0}}{\arrivalLS}\right)^{\nodeQueue_{0}}\abrevC{\nodeQueue_{0}}}\right)\\
 & =\etaRatio_{0}\cdot\left(\frac{\abrevC 1+\sum_{\nodeQueue_{0}=2}^{\poolSize}\left(\frac{\etaRatio_{0}}{\arrivalLS}\right)^{\nodeQueue_{0}-1}\abrevC{\nodeQueue_{0}}}{\abrevC 0+\sum_{\nodeQueue_{0}=1}^{\poolSize}\left(\frac{\etaRatio_{0}}{\arrivalLS}\right)^{\nodeQueue_{0}}\abrevC{\nodeQueue_{0}}}\right).
\end{align*}
Hence, it holds
\[
\lim_{\arrivalLS\rightarrow\infty}\effarrival(\arrivalLS)=\etaRatio_{0}\cdot\frac{\abrevC 1}{\abrevC 0}=\maxArrival.
\]
Therefore, $\effarrival(\arrivalLS)$ can be larger than any stable
arrival rate $\arrival\in\left(0,\maxArrival\right)$.

Now we show that $\effarrival(\arrivalLS)$ can be smaller than any
stable $\arrival\in\left(0,\maxArrival\right)$. It follows from
\[
\lim_{\arrivalLS\rightarrow0}\effarrival(\arrivalLS)=\lim_{\arrivalLS\rightarrow0}\arrivalLS\cdot\underbrace{\left(1-\resourceProb(0)\right)}_{>0\ \text{and }<1}=0.
\]

Finally, from $\effarrival(\arrivalLS)=\arrivalLS\cdot\left(1-\frac{\abrevC 0}{\sum_{\nodeQueue_{0}=0}^{\poolSize}\left(\frac{\etaRatio_{0}}{\arrivalLS}\right)^{\nodeQueue_{0}}\cdot\abrevC{\nodeQueue_{0}}}\right)$
follows that $\effarrival$ is a continuous function in $\arrivalLS\in\left(0,\infty\right)$,
which proves our claim by the intermediate value theorem.
\end{proof}
Henceforth, we will call $\arrivalLS$ with $\effarrival(\arrivalLS)=\arrival$
\emph{adjusted} arrival rate for $\arrival$. 
\begin{prop}
\label{prop:unique}If the service rates $\nodeServiceRate_{j}(\cdot)$,
$j\in\nodeSet$, are non-decreasing, $\arrivalLS$ in \prettyref{thm:theorem-ls}
is unique.
\end{prop}

The proof of \prettyref{prop:unique} is presented in \prettyref{appx:omitted-calculation}.

Interested readers will find the explicit results for adjusted $\arrivalLS$
in the special cases $\poolSize=1$ and $\poolSize=2$ in \prettyref{rem:lambda-eff-N-1-N-2}
in \prettyref{appx:omitted-calculation}.

\paragraph*{Some arguments for applying the resource network of the SOQN-LC as
approximation for that of the SOQN-BO.}

The result of \prettyref{thm:theorem-ls} only guarantees that for
a stable SOQN-BO with a prescribed external arrival rate $\arrival$
there exists an SOQN-LC with the same resource network where the resource
pool has the same throughput. Because the SOQN-LC is a standard Gordon-Newell
network, the local throughputs can be computed directly by standard
algorithms. We can compare these local throughputs with those of the
SOQN-BO. Surprisingly, not only the throughputs at the queues of the
resource network are the same by construction. All throughput pairs
in the respective nodes of the resource network are identical, too.
This observation suggests to use the local characteristics of the
queues in the resource network of the SOQN-LC as approximation for
the respective performance measures of the SOQN-BO. The point is:
The performance characteristics of the SOQN-BO are not directly accessible,
while the performance characteristics of the SOQN-LC are explicitly
known from product-form network theory, and even more, well-established
algorithmic procedures are at hand to evaluate these.

In the next section, we prove the coincidence of the respective local
throughputs. Thereafter, we show that in nodes with constant service
rates even the probabilities for an empty queue in both resource networks
are pairwise identical. 

\subsection{Throughputs and idle times\label{sect:TH_BO}}
\begin{prop}
\label{prop:THL-LC}The local throughput $\ThLC j$ at nodes $j\in\nodeSetNull$
in the SOQN-LC with adjusted arrival rate is pairwise the same as
that of the respective nodes in the SOQN-BO given in \prettyref{prop:THL-BO}.
With
\begin{equation}
\normLS(\nodeSetNull,N)=\sum_{n_{0}=0}^{N}\left(\frac{\etaRatio_{0}}{\arrivalLS}\right)^{n_{0}}\sum_{\sum_{j\in\nodeSet}n_{j}=N-n_{0}}\prod_{j=1}^{J}\left(\prod_{\ell=1}^{n_{j}}\frac{\etaRatio_{j}}{\nodeServiceRate_{j}(\ell)}\right)\label{eq:NormingConst}
\end{equation}
it holds

\[
\ThLC j=\etaRatio_{j}\cdot\frac{\normLS(\nodeSetNull,N-1)}{\normLS(\nodeSetNull,N)}=\arrival\cdot\frac{\etaRatio_{j}}{\etaRatio_{0}},\quad j\in\nodeSetNull.
\]
\end{prop}

\begin{proof}
It was shown in the proof of \prettyref{thm:theorem-ls} that it holds
\[
\ThLC 0=\effarrival(\arrivalLS)=\eta_{0}\cdot\frac{\normLS(\nodeSetNull,N-1)}{\normLS(\nodeSetNull,N)}
\]
which is the standard formula for the throughput at node $0$ in the
resource network, which is a standard Gordon-Newell network. Therefore,
\[
\arrival=\etaRatio_{0}\cdot\frac{\normLS(\nodeSetNull,N-1)}{\normLS(\nodeSetNull,N)}
\]
and by using the standard formula for local throughputs in Gordon-Newell
networks follows
\[
\ThBO j=\arrival\cdot\frac{\etaRatio_{j}}{\etaRatio_{0}}=\etaRatio_{j}\cdot\frac{\normLS(\nodeSetNull,N-1)}{\normLS(\nodeSetNull,N)}=\ThLC j\quad j\in\nodeSetNull.
\]
\end{proof}
The explicit formula for the throughputs in \prettyref{prop:THL-LC}
has a further interesting consequence. It allows us to determine efficiently
the steady-state marginal distribution of the queue length at every
node $j\in\nodeSet$ without knowing the adjusted value $\effarrival(\arrivalLS)$
of $\arrivalLS$. 
\begin{prop}
Let $Y_{\text{LC}}:=(Y_{\text{LC},j}:j\in\nodeSet)$ denote a random
vector which is distributed according to the stationary queue length
at the nodes in $\nodeSet$ of the SOQN-LC with adjusted arrival rate.

If the service rate at node $j$ does not depend on the queue length,
i.e. $\nu_{j}(\cdot)=\nu_{j},j\in\nodeSet$, then the probabilities
that the nodes $j\in\nodeSetNull$ in the SOQN-LC with adjusted arrival
rate are idling are pairwise the same as those of the respective nodes
in the SOQN-BO given in \prettyref{cor:idle-times-BO}:
\[
P(Y_{\text{LC},j}=0)=1-\arrival\cdot\frac{\etaRatio_{j}}{\etaRatio_{0}}\cdot\nodeServiceRate_{j}^{-1}.
\]
\end{prop}

\begin{proof}
According to \prettyref{prop:THL-LC}, the throughputs are equal.
The rest of the proof is the same as the proof of \prettyref{cor:idle-times-BO}.
\end{proof}

\section{Approximation of the external queue\label{sec:Approximation-of-the-external-queue}}

Although, after adjusting $\arrivalLS$, the behaviour of the resources
in both systems -- the original SOQN-BO and the {\similar} SOQN-LC
-- is very similar, their external queues are still very different.
No matter how highly we increase $\arrivalLS$, the external queue
length of the SOQN-LC is always zero, while the external queue of
the SOQN-BO has strictly positive mean. Therefore, we cannot use the
{\similar} system directly to estimate the external queue of the
original system. Instead, in the following two sections, we will use
a two-step approach to approximate the external queue:
\begin{description}
\item [{step~1}] In \prettyref{sec:small-modified}, we will reduce the
modified system to a simple system {\withLostCustomers}.
\item [{step~2}] In \prettyref{sub:small-with-backordering}, we will
combine the results from \prettyref{sec:small-modified} and the results
from \prettyref{sec:Special-case} for a simple system with backordering
to approximate the external queue.
\end{description}

\subsection[Reduced SOQN-LC]{Reduced SOQN {\withLostCustomers}\label{sec:small-modified}}

Because the SOQN-LC from \prettyref{sec:Lost-network} is a Gordon-Newell
network, we can reduce complexity further by applying Norton's theorem
proved by \citet{chandy;herzog;woo75:} to construct a two-node Gordon-Newell
with the same throughput.

\begin{figure}[h]
\includegraphics[width=1\textwidth]{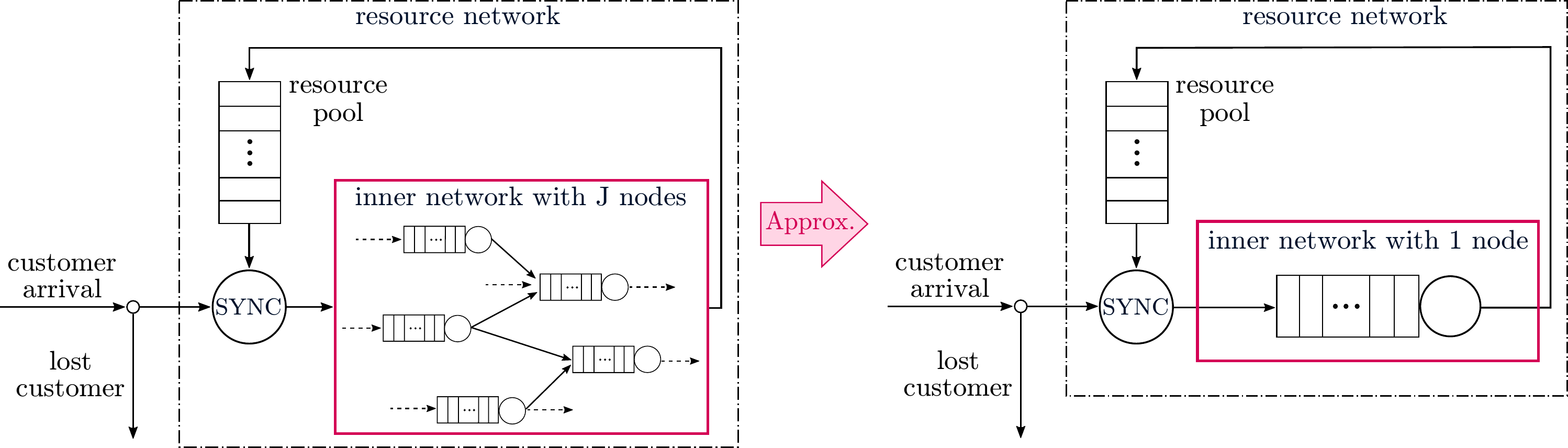}

\caption{Step 1: Reduction of complexity\label{fig:Reduction-of-complexity}.}
\end{figure}

The inner network is replaced by only one composite node ($\nodeSet:=\left\{ 1\right\} $),
which consists of a single server with infinite waiting room under
the FCFS regime. The service time is exponentially distributed with
mean $1$. The service speed is determined by a queue-length-dependent
service intensity. According to \citet[p. 39, eq. (20)]{chandy;herzog;woo75:},
the service intensity $\subRate$ is given by
\begin{align}
\subRate(0) & =0,\nonumber \\
\subRate(m) & =\etaRatio_{0}\cdot\frac{\normB(\nodeSet,m-1)}{\normB(\nodeSet,m)},\quad m\in\left\{ 1,\ldots,\poolSize\right\} .\label{eq:small-model-phi}
\end{align}
Remarkably, $\subRate(N)$ is the same as $\maxArrival$ in \eqref{eq:max-arrival-from-norm-constants}.
We deduce from \eqref{eq:stable-th0-by-thi} that
\begin{equation}
\subRate(m)=\stabilityNetworkTh 0(m),\quad m\in\left\{ 1,\ldots,\poolSize\right\} ,\label{eq:small-model-phi-from-th-0}
\end{equation}
and it does not depend on $\arrivalLS$.

The normalisation constants $\normB(\nodeSet,m)$, $m\in\left\{ 0,\ldots,\poolSize\right\} $,
can be calculated by the convolution algorithm or mean value analysis
(MVA). They are illustrated in \citet[p. 371ff., Section 8.1 and p. 384ff., Section 8.2]{Bolch:1998:QNM:289350}.

\subsection{Back to backordering\label{sub:small-with-backordering}}

We can use the result of \prettyref{thm:theorem-ls}, that for every
stable SOQN-BO, there exists an SOQN-LC with adjusted arrival rate
$\arrivalLS$ such that both systems have the same throughput in the
steady state. So we can remove the lost-customer property to get again
the backordering property as shown in \prettyref{fig:Transition-from-lost-to-back}.

\begin{figure}[h]
\includegraphics[width=1\textwidth]{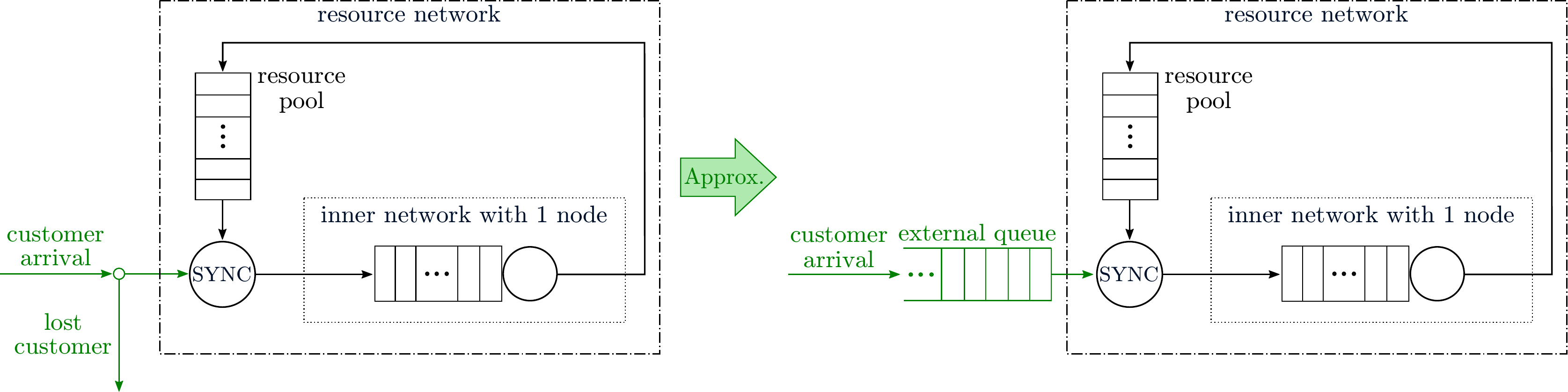}\caption{\label{fig:Transition-from-lost-to-back}Step 2: Transition from reduced
SOQN-LC to reduced SOQN-BO.}
\end{figure}

Having done this, we can eventually  approximate the external queue
of the large SOQN-BO with $\nodes>1$ by a reduced SOQN-BO with $\nodes=1$.
We calculate $\stabilityNetworkTh 0(m)$, $m\in\left\{ 1,\ldots,\poolSize\right\} $,
of the large system  and substitute these throughputs in the reduced
system like in \eqref{eq:small-model-phi-from-th-0}: $\nodeServiceRate_{1}(m):=\subRate(m)=\stabilityNetworkTh 0(m)$.
Then we insert these rates $\nodeServiceRate_{1}(m)$ into the formulas
for $P(\externalQueueProcRV=\externalQueue)$, $\externalLength$
and $\externalWaitingTime$ in \prettyref{prop:marginal-distribution-of-a-simple-system},
which we know to be exact for $\nodes=1$. For $\nodes>1$ we expect
that the results are close to the real values, but we cannot give
error bounds at the present.

\[
P(\externalQueueProcRV=0)\overset{\eqref{eq:eq:small-backordering-p-X-ex-0}}{\text{\ensuremath{\approx}}}\left[\specialNormB(\left\{ 1\right\} ,\poolSize)\right]^{-1}\cdot\sum_{\nodeQueue_{1}=0}^{\robotnumb}\prod_{m=1}^{\nodeQueue_{1}}\frac{\arrival}{\stabilityNetworkTh 0(m)}
\]
and for $\externalQueue>0$
\[
P(\externalQueueProcRV=\externalQueue)\overset{\eqref{eq:small-backordering-p-X-ex-n}}{\text{\ensuremath{\approx}}}\left[\specialNormB(\left\{ 1\right\} ,\poolSize)\right]^{-1}\cdot\prod_{m=1}^{\poolSize}\frac{\arrival}{\stabilityNetworkTh 0(m)}\cdot\left(\frac{\arrival}{\stabilityNetworkTh 0(\poolSize)}\right)^{\externalQueue}
\]
with
\[
\specialNormB(\left\{ 1\right\} ,\poolSize)\overset{\eqref{eq:Modell-all-together-normierungskonstante}}{=}\sum_{\nodeQueue_{1}=0}^{\poolSize-1}\prod_{m=1}^{\nodeQueue_{1}}\frac{\arrival}{\stabilityNetworkTh 0(m)}+\prod_{m=1}^{N}\frac{\arrival}{\stabilityNetworkTh 0(m)}\cdot\frac{1}{1-\frac{\arrival}{\stabilityNetworkTh 0(\poolSize)}}.
\]
We approximate the average number of customers in the external queue
with \eqref{eq:randvtlg-y} and \eqref{eq:small-backordering-L-ex-0}
\begin{equation}
\externalLength\approx\externalLength^{\text{apprx}}:=\left[\specialNormB(\left\{ 1\right\} ,\poolSize)\right]^{-1}\cdot\prod_{m=1}^{\poolSize}\frac{\arrival}{\stabilityNetworkTh 0(m)}\cdot\frac{1}{1-\frac{\arrival}{\stabilityNetworkTh 0(\poolSize)}}\cdot\frac{\arrival}{\stabilityNetworkTh 0(\poolSize)-\arrival}.\label{eq:external-length}
\end{equation}

Note, with our approximation method we arrive to the same formula
for $\externalLength^{\text{apprx}}$ as \citet[eq. (22)]{dallery:1990}
and thus our method produces the same formula as the aggregation technique
of \citet[Section 6]{dallery:1990}.

We approximate the average waiting time of customers in the external
queue with \eqref{eq:waiting-time-ex}
\begin{equation}
\externalWaitingTime\text{\ensuremath{\approx}}\frac{\externalLength^{\text{apprx}}}{\arrival}.\label{eq:waiting-time-ex-approx}
\end{equation}

\section{Application to RMFS\label{sec:application}}

In this section, we use our approximation algorithm to calculate the
minimal number of robots for a robotic mobile fulfilment system (RMFS).
In an RMFS, robots are expensive resources. Therefore, we want to
keep their number small while still maintaining the necessary quality
of service.

We will model an RMFS as SOQN-BO and evaluate its performance analytically
because this is much faster than evaluating a simulation model. Because
of the large state space, it is impractical to solve this SOQN-BO
exactly with matrix-geometric methods, even for a small number of
robots. For example, for 10 robots, we need to calculate ca.~$\left(9\cdot10^{4}\right)^{2}$
entries of a special matrix. That is why we will use our approximation
methods instead, in order to quickly estimate the main performance
metrics.

\subsection{\label{sec:description-of-RMFS}Description of RMFS}

Firstly, we define the components and the order fulfilment processes
in an RMFS together with an illustrated example in \prettyref{fig:rmfs-example}.
The central components are:
\begin{itemize}
\item movable shelves, called \emph{pods}, on which items are stored,
\item \emph{storage area} -- the area where the pods are stored,
\item workstations, where
\begin{itemize}
\item the items are picked from pods by pickers (\emph{picking stations})
or
\item the items are stored to pods (\emph{replenishment stations}),
\end{itemize}
\item mobile \emph{robots}, which can move underneath pods and carry them
to workstations.
\end{itemize}
\prettyref{fig:rmfs-example} illustrates an example of order fulfilment
processes in an RMFS. On the upper left hand we have three customers'
orders. The orders contain different items, which are illustrated
with different colours. To fulfil customers' orders, we send them
or parts of them to picking stations. To the same station we send
pods with all the necessary items. Each pod is carried by a robot.
In this way, customers' orders generate tasks for robots. The robots,
with their pods, queue up in front of the picking stations. A picker
takes all the necessary items from the pod at the head of the queue.
Then he sends it, with its robot, back to the storage area. As soon
as the customer's order or part of it is fulfilled, we remove it from
the picking station. The order in which we send the customers' orders,
how we split them apart, and which pod we send, is a complex topic.
See, for example, \citet{orderpicking}. In the present paper, we
focus on the generated robots' tasks, which we will call just \emph{tasks}.

In this example, each customer's order is split into two parts. Three
parts are sent to picking station~1, and three other parts are sent
to picking station~2. To fulfil these partial orders, a robot transports
one pod to picking station~1, and another robot transports one pod
to picking station~2, from the storage area. 

From time to time, we need to refill the pods. To do this, we send
these pods to the replenishment station. There, employees will refill
these pods and send them back to the storage area.

In this example, after picking, pod~2 is sent to the replenishment
station to refill it with the blue items.

\begin{figure}[h]
\centering{}\includegraphics[width=1\textwidth]{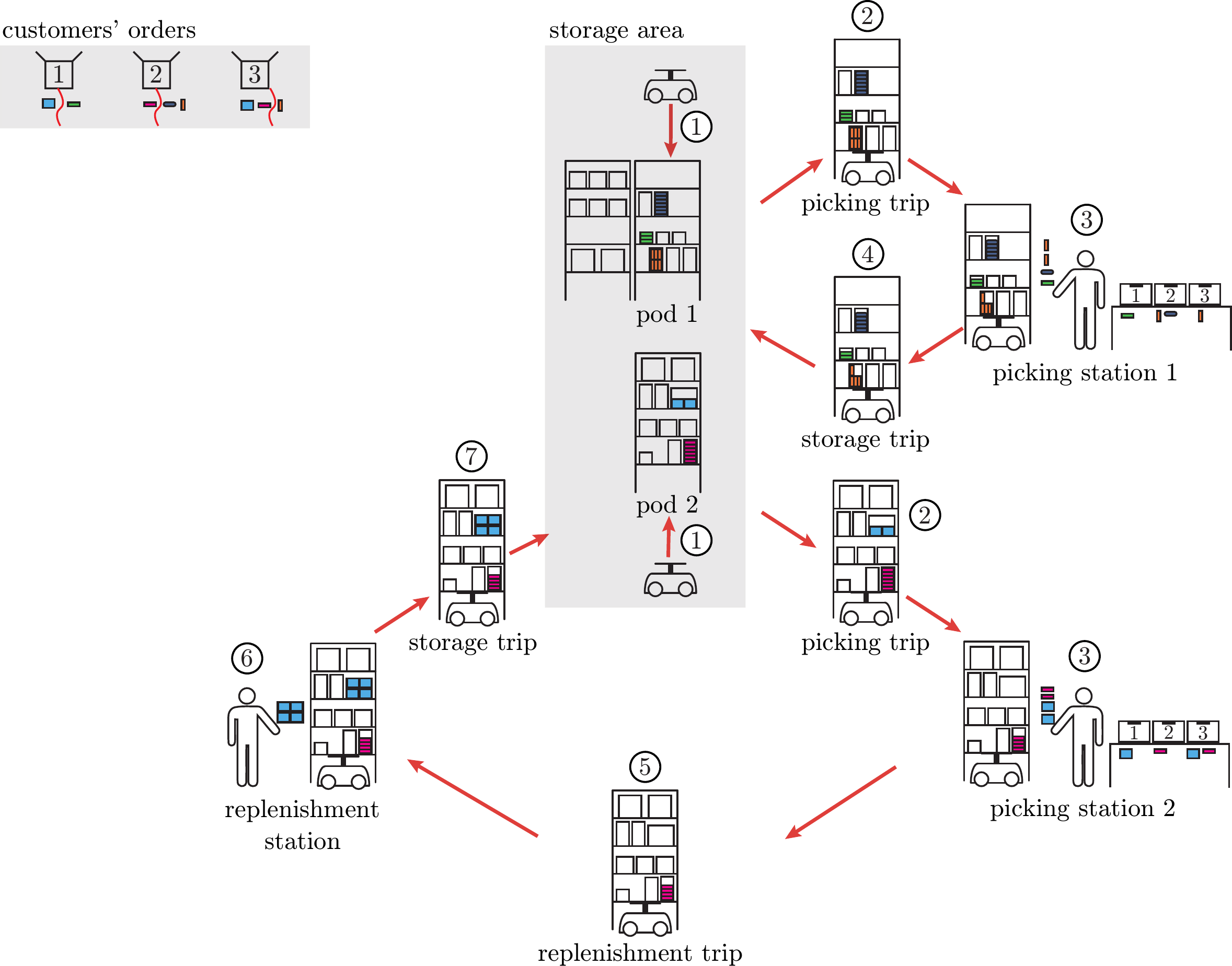}\caption{\label{fig:rmfs-example}Order fulfilment processes in RMFS. The circled
numbers refer to the processes in \prettyref{fig:RMFS-image}.}
\end{figure}

\subsection{Modelling as SOQN}

A robotic mobile fulfilment system can be modelled as an SOQN-BO.
This is depicted in \prettyref{fig:RMFS-image}. An RMFS is open with
respect to tasks and closed with respect to robots, which are the
resources in this model.

In this section, we consider only an RMFS with two picking stations
and one replenishment station, but the results from Sections \ref{sec:A-general-semi-open}
and \ref{sec:Approximation} about general SOQN can also be applied
to an RMFS with more than two picking stations and more than one replenishment
station.

\begin{figure}[h]
\includegraphics[width=1\textwidth]{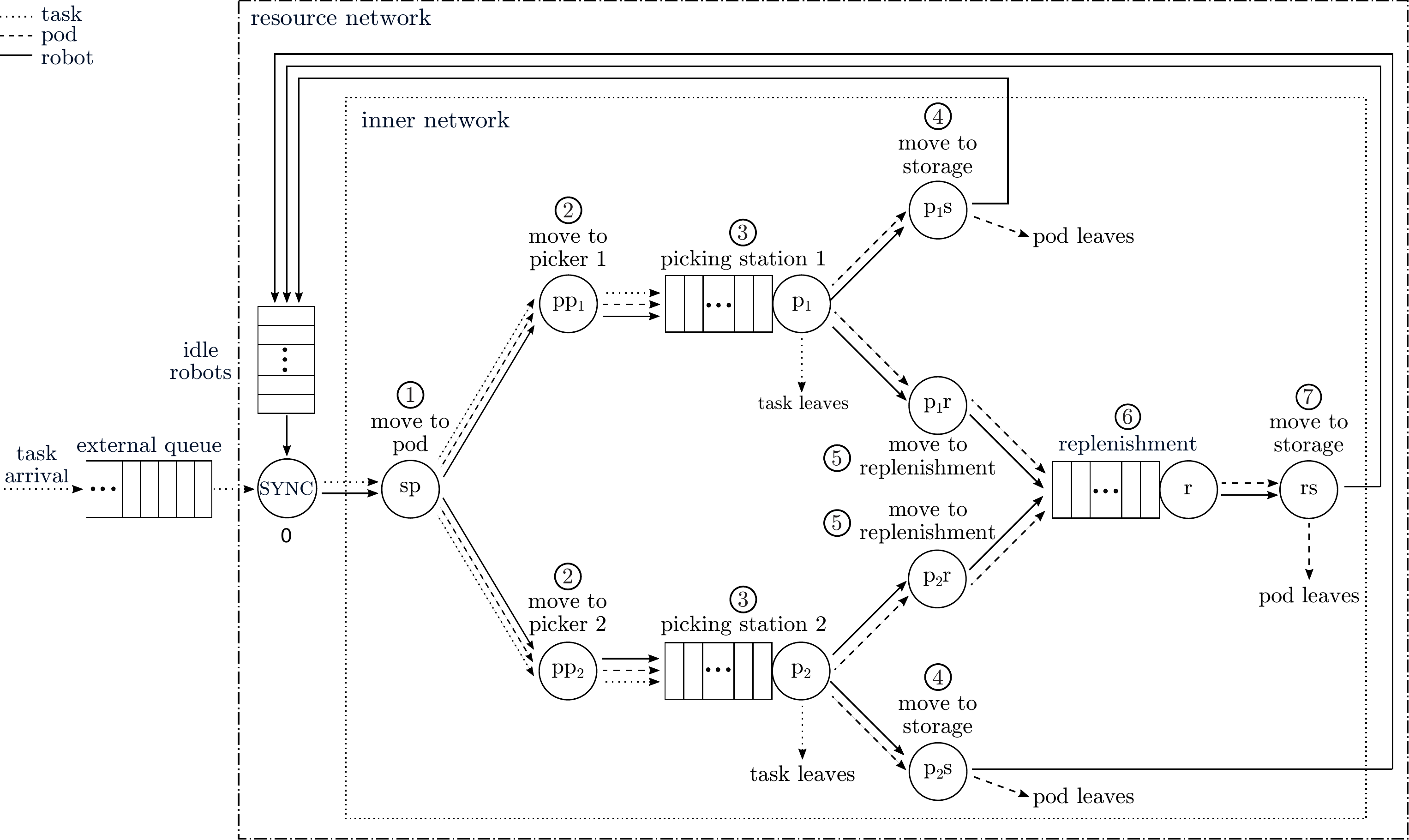}

\caption{\label{fig:RMFS-image}RMFS modelled as an SOQN-BO. The circled numbers
refer to the processes in \prettyref{fig:rmfs-example}.}
\end{figure}

Customers' orders arrive at the RMFS one by one with rate $\customerOrdersArrival$
and generate tasks. The number of tasks, that a single customer order
can generate depends on many parameters. In particular, it depends
on the efficiency of the algorithm which tries to find an optimal
match between customers' orders and pods. The matching problem is
NP-hard, and, to the best of our knowledge, there are no formulas
known which determine how many pods an order will require. Therefore,
we assume that there exists some average pod/order ratio $\podToOrderRatio$
which we can find empirically for a particular RMFS. We assume that
this ratio depends only on the pods' contents and customers' order
contents, and it does not depend on the number of robots.

The matching algorithm also adds some delay in order to assign pods
to orders. We assume that this delay depends only on the pods' contents,
customers' order contents and order input rate and it does not depend
on the number of robots. We assume that we can find this delay empirically
for our particular RMFS and it is on average $\matchingDelay$.

Thus, the customers' orders generate a stream of ``bring a pod to
a picking station'' tasks. This stream has the rate $\arrival=\customerOrdersArrival\cdot\podToOrderRatio>0$.
The delay, introduced by the matching algorithm, does not change this
rate.

We simplify all the complexity that occurs until each  task is created
and model the task stream as a Poisson arrival stream with rate $\arrival=\customerOrdersArrival\cdot\podToOrderRatio$.
To be processed ($=$ to enter the inner network), each such task
requires exactly one idle robot from the robot pool (resource pool),
which is henceforth referred to as node $0$. If there is no idle
robot available, the new task has to wait in an external queue until
a robot becomes available. The maximal number of robots in the resource
pool is $\robotnumb$. The inner network in the example in \prettyref{fig:RMFS-image}
consists of 11 nodes, denoted by
\[
\nodeSet:=\left\{ \toPodState,\toPickOneState,\toPickTwoState,\pickerQueueOneState,\pickerQueueTwoState,\toStorageOneState,\toStorageTwoState,\toReplOneState,\toReplTwoState,\replQueueState,\toReplToStorageState\right\} .
\]
The notations of the nodes are presented in Table \ref{tab:RMFS-nodes-overview}
\vpageref{tab:RMFS-nodes-overview}.

The robot with assigned task moves through the network.

The following processes occur from the perspective of a robot:
\begin{itemize}
\item The idle robot waits to be assigned to a task (bring a particular
pod).
\item The robot moves with the assigned task to a pod.
\item With this pod, the robot moves to one of the picking stations, more
precisely, with probability $\transprobPickOne\in\left(0,1\right)$
to picking station $1$ and with probability $\transprobPickTwo\in\left(0,1\right)$
to picking station $2$, whereby $\transprobPickOne+\transprobPickTwo=1$.
\item The robot queues with the pod at the picking stations.
\item After picking at picking station $1$ resp.~picking station 2, the
robot:
\begin{itemize}
\item either
\begin{itemize}
\item carries the pod directly back to the storage area with probability
$\transprobpickReplOne\in\left(0,1\right)$ resp.~$\transprobpickReplTwo\in\left(0,1\right)$,
or
\item moves to the replenishment station with probability $\transprobReplOne\in\left(0,1\right)$
resp.~$\transprobReplTwo\in\left(0,1\right)$, whereby $\transprobpickReplOne+\transprobReplOne=1$
resp.~$\transprobpickReplTwo+\transprobReplTwo=1$,
\end{itemize}
\item queues at the replenishment station and
\item carries the pod back to the storage area and waits for the next task.
\end{itemize}
\end{itemize}
Each of these processes is modelled as a queue.

All the movements of the robots are modelled by processor-sharing
nodes with exponentially distributed service times. Their intensities
$\nu_{j}(n_{j}):=\mu_{j}\cdot\phi_{j}(n_{j})$, $j\in\nodeSet\setminus\left\{ \pickerQueueOneState,\pickerQueueTwoState,\replQueueState\right\} $,
are presented in Table \ref{tab:RMFS-nodes-overview}.

The two picking stations and the replenishment station, which are
referred to as node $\pickerQueueOneState$, node $\pickerQueueTwoState$
resp.~node $\replQueueState$, consist of a single server with waiting
room under the FCFS regime. The picking times and the replenishment
times are exponentially distributed with rates $\pickingOne$, $\pickingTwo$
resp.~$\repl$.

\begin{table}[h]
\caption{\label{tab:RMFS-nodes-overview}Overview of the nodes in the network}

\begin{tabular}{ccccc}
\hline 
\multirow{2}{*}{\textbf{Node}} & \textbf{Service} & \textbf{Random} & \multirow{2}{*}{\textbf{State}} & \textbf{Description}\tabularnewline
 & \textbf{intensity} & \textbf{variable} &  & \textbf{(number of robots at time $t\geq0$)}\tabularnewline
\hline 
\multirow{2}{*}{$\toPodState$} & \multirow{2}{*}{$\transToPod\cdot\phiToPod(\toPod)$} & \multirow{2}{*}{$\procToPod(t)$} & \multirow{2}{*}{$\toPod$} & \multirow{2}{*}{moving in the storage area to a pod}\tabularnewline
 &  &  &  & \tabularnewline
\hline 
\multirow{2}{*}{$\toPickOneState$} & \multirow{2}{*}{$\transToPickerOne\cdot\phiToPickOne(\toPickOne)$} & \multirow{2}{*}{$\procToPickOne(t)$} & \multirow{2}{*}{$\toPickOne$} & moving a pod from the storage area\tabularnewline
 &  &  &  & to picking station $1$\tabularnewline
\hline 
\multirow{2}{*}{$\toPickTwoState$} & \multirow{2}{*}{$\transToPickerTwo\cdot\phiToPickTwo(\toPickTwo)$} & \multirow{2}{*}{$\procToPickTwo(t)$} & \multirow{2}{*}{$\toPickTwo$} & moving a pod from the storage area\tabularnewline
 &  &  &  & to picking station $2$\tabularnewline
\hline 
\multirow{2}{*}{$\pickerQueueOneState$} & \multirow{2}{*}{$\pickingOne$} & \multirow{2}{*}{$\procPickerQueueOne(t)$} & \multirow{2}{*}{$\pickerQueueOne$} & \multirow{2}{*}{in the queue of picking station $1$}\tabularnewline
 &  &  &  & \tabularnewline
\hline 
\multirow{2}{*}{$\pickerQueueTwoState$} & \multirow{2}{*}{$\pickingTwo$} & \multirow{2}{*}{$\procPickerQueueTwo(t)$} & \multirow{2}{*}{$\pickerQueueTwo$} & \multirow{2}{*}{in the queue of picking station $2$}\tabularnewline
 &  &  &  & \tabularnewline
\hline 
\multirow{2}{*}{$\toStorageOneState$} & \multirow{2}{*}{$\transToStorageOne\cdot\phiToStorageOne(\toStorageOne)$} & \multirow{2}{*}{$\procToStorageOne(t)$} & \multirow{2}{*}{$\toStorageOne$} & moving a pod from picking station $1$\tabularnewline
 &  &  &  & to the storage area and entering node $0$\tabularnewline
\hline 
$\toStorageTwoState$ & $\transToStorageTwo\cdot\phiToStorageTwo(\toStorageTwo)$ & $\procToStorageTwo(t)$ & $\toStorageTwo$ & moving a pod from picking station $2$\tabularnewline
 &  &  &  & to the storage area and entering node $0$\tabularnewline
\hline 
\multirow{2}{*}{$\toReplOneState$} & \multirow{2}{*}{$\transToReplOne\cdot\phiToReplOne(\toReplOne)$} & \multirow{2}{*}{$\procToReplOne(t)$} & \multirow{2}{*}{$\toReplOne$} & moving a pod from picking station $1$\tabularnewline
 &  &  &  & to the replenishment station\tabularnewline
\hline 
\multirow{2}{*}{$\toReplTwoState$} & \multirow{2}{*}{$\transToReplTwo\cdot\phiToReplTwo(\toReplTwo)$} & \multirow{2}{*}{$\procToReplTwo(t)$} & \multirow{2}{*}{$\toReplTwo$} & moving a pod from picking station $2$\tabularnewline
 &  &  &  & to the replenishment station\tabularnewline
\hline 
\multirow{2}{*}{$\replQueueState$} & \multirow{2}{*}{$\repl$} & \multirow{2}{*}{$\procRepl(t)$} & \multirow{2}{*}{$\replQueue$} & \multirow{2}{*}{in the queue of the replenishment station}\tabularnewline
 &  &  &  & \tabularnewline
\hline 
\multirow{2}{*}{$\toReplToStorageState$} & \multirow{2}{*}{$\transReplToStorage\cdot\phiToStorage(\toReplToStorage)$} & \multirow{2}{*}{$\procReplToStorage(t)$} & \multirow{2}{*}{$\toReplToStorage$} & moving a pod from the replenishment station\tabularnewline
 &  &  &  & to the storage area and entering node $0$\tabularnewline
\hline 
\end{tabular}
\end{table}

The robots travel among the nodes following a fixed routing matrix
$\routingMatrix:=\left(\routingProb(i,j):i,j\in\nodeSetNull\right)$,
whereby $\nodeSetNull:=\left\{ 0\right\} \cup\nodeSet$, which is
given by
\[
\routingMatrix=\left(\begin{array}{c|cccccccccccc}
 & \freerobiState & \toPodState & \toPickOneState & \toPickTwoState & \pickerQueueOneState & \pickerQueueTwoState & \toStorageOneState & \toStorageTwoState & \toReplOneState & \toReplTwoState & \replQueueState & \toReplToStorageState\\
\hline \freerobiState &  & 1\\
\toPodState &  &  & \transprobPickOne & \transprobPickTwo\\
\toPickOneState &  &  &  &  & 1\\
\toPickTwoState &  &  &  &  &  & 1\\
\pickerQueueOneState &  &  &  &  &  &  & \transprobpickReplOne &  & \transprobReplOne\\
\pickerQueueTwoState &  &  &  &  &  &  &  & \transprobpickReplTwo &  & \transprobReplTwo\\
\toStorageOneState & 1\\
\toStorageTwoState & 1\\
\toReplOneState &  &  &  &  &  &  &  &  &  &  & 1\\
\toReplTwoState &  &  &  &  &  &  &  &  &  &  & 1\\
\replQueueState &  &  &  &  &  &  &  &  &  &  &  & 1\\
\toReplToStorageState & 1
\end{array}\right).
\]

The routing matrix $\routingMatrix$ is irreducible by construction.

We define the joint stochastic process $Z$ of this system by
\begin{align*}
Z & :=\Big(\Big(\procbackordering(t),\procfreerobi(t),\procToPod(t),\procToPickOne(t),\procToPickTwo(t),\procPickerQueueOne(t),\procPickerQueueTwo(t),\procToStorageOne(t),\procToStorageTwo(t),\\
 & \phantomeq\quad\ \ \procToReplOne(t),\procToReplTwo(t),\procRepl(t),\procReplToStorage(t)\Big):t\geq0\Big).
\end{align*}

Due to the usual independence and memorylessness assumptions (see
the assumptions \vpageref{independence-memorylessness}), $Z$ is
a homogeneous Markov process with state space
\begin{align*}
E & :=\phantom{\cup}\big\{\left(0,\freerobi,\toPod,\toPickOne,\toPickTwo,\pickerQueueOne,\pickerQueueTwo,\toStorageOne,\toStorageTwo,\toReplOne,\toReplTwo,\replQueue,\toReplToStorage\right):\\
 & \phantomeq\phantom{\cup\big\{\:}n_{j}\in\left\{ 0,\ldots,\robotnumb\right\} \ \forall j\in\nodeSetNull,\ \sum_{j\in\nodeSetNull}n_{j}=\robotnumb\big\}\\
 & \phantomeq\cup\big\{\left(\backordering,0,\toPod,\toPickOne,\toPickTwo,\pickerQueueOne,\pickerQueueTwo,\toStorageOne,\toStorageTwo,\toReplOne,\toReplTwo,\replQueue,\toReplToStorage\right):\\
 & \phantomeq\phantom{\cup\big\{\:}\backordering\in\mathbb{N},\:n_{j}\in\left\{ 0,\ldots,\robotnumb\right\} \ \forall j\in\nodeSet,\ \sum_{j\in\nodeSet}n_{j}=\robotnumb\big\}.
\end{align*}
$Z$ is irreducible on $E$.

\subsection{Determine the minimal number of robots}

We see that the throughput $\stabilityNetworkTh 0(\poolSize)$ depends
on $\poolSize$, the number of robots. This raises the question: ``How
many robots do we need to stabilize a given system?'' Consequently,
to find the minimal number of robots, we first check the stability
criterion from \prettyref{prop:stability-condition}.
The maximal number of robots $N^{\max}$ is equal to the number of
pods or is determined by financial restrictions. In the following
\prettyref{alg:algorithm-stable}, we determine the set $\overline{N}^{*}$of
feasible numbers of robots for a stable system.
\begin{algorithm}
\begin{algorithmic}[1]
\Function{StableRobotsSet}{}
\State $\overline{N}^{*}=\{N\in \{1,\ldots,N^{\max}\}:\lambda_{\text{BO}}<TH^{\text{stb}}_{0}(N)\}$
 \State \Return $\overline{N}^{*}$
\EndFunction
\end{algorithmic}\caption{\label{alg:algorithm-stable}Calculate set of feasible numbers of
robots for a stable system}
\end{algorithm}

\begin{rem}
If $\stabilityNetworkTh 0(\cdot)$ is non-decreasing in $\mathbb{N}$,
then the algorithm can be improved: the algorithm does not have to
check all possible numbers of robots . It starts with one robot and
adds a new robot in each new step until the stability criterion is
satisfied for the first time. You will find some simple sufficient
conditions for non-decreasing $\stabilityNetworkTh 0(\cdot)=\maxArrival$
in \prettyref{prop:monoton-for-more-1} and \ref{prop:monoton-for-1}.
\end{rem}

Unfortunately, stability does not say anything about the turnover
time of an order. We do not know if only 2 minutes are required to
satisfy a customer\textquoteright s order or maybe the order requires
2 years to be processed. In case of the latter turnover time, the
system is stable, but customers may not be happy.

Now, in addition to stability, we want to consider the turnover time
of a customer's order for the quality of service.

The turnover time of a customer's order can be split into three main
parts:
\begin{enumerate}
\item Waiting time until the matching algorithm has assigned all required
pods to that order. By assumption, this time does not depend on the
number of robots and is on average $\matchingDelay>0$.
\item Waiting time of the first matched pod for an idle robot, time for
transport to the picking station, waiting time for the picker at the
picking station. During all these times, the order is coupled with
at least one  task. We call this turnover time for the task $\TOtask(\arrivalLS,\poolSize)$.
\item Time of an order between start of picking and its completion. This
time is complex and depends on many factors. Like, for example: How
many orders can a picker complete with the same pod? Will all completed
orders wait until a pod leaves? Is the order's content in multiple
pods? Will these pods arrive right after each other, or will there
be many pods for other orders in between? Is there any complex merging
procedure outside of the picking station? In our model, we use a simplifying
assumption that the order needs on average $\assembledTime>0$ from
the time its first pod arrives at the picking station until the time
picking for this order is completed.
\end{enumerate}
With all these assumptions, we can assume that the turnover time of
an order is 
\[
TO_{\text{order}}(\arrivalLS,\poolSize):=\matchingDelay+\TOtask(\arrivalLS,\poolSize)+\assembledTime.
\]
Even in the case where $\matchingDelay$ and $\assembledTime$ are
not known, we can still use $\TOtask(\arrivalLS,\poolSize)$ as a
lower bound for $TO_{\text{order}}(\arrivalLS,\poolSize)$.

Because of the simplifying assumption about $\matchingDelay$ and
$\assembledTime$, only the turnover time $\TOtask(\arrivalLS,\poolSize)$
of a task depends on $\poolSize$, and for the minimal number of robots
we can focus on this.

The turnover time $\TOtask(\arrivalLS,\poolSize)$ of a task is measured
from the time the task is received to the time the picker starts to
process it:
\[
\TOtask(\arrivalLS,\poolSize):=\externalWaitingTime(\poolSize)+\innerWaitingTime(\arrivalLS,\poolSize).
\]
$\externalWaitingTime(\poolSize)$ is the average time which a task
spends waiting in the external queue until it enters the inner network.
We can calculate it with \eqref{eq:waiting-time-ex-approx}.

$\innerWaitingTime(\arrivalLS,\poolSize)$ is the average time which
a task spends in the inner network until a picker starts to process
it at one of the picking stations. Given the average waiting times
$W_{j}(\arrivalLS,\poolSize)$ at nodes $j\in\nodeSet=\left\{ \toPodState,\toPickOneState,\toPickTwoState,\pickerQueueOneState,\pickerQueueTwoState,\toStorageOneState,\toStorageTwoState,\toReplOneState,\toReplTwoState,\replQueueState,\toReplToStorageState\right\} $
from arrival until service completion, and constant service rates
$\nu_{j}$ at nodes $j\in\{\pickerQueueOneState,\pickerQueueTwoState$\},
then
\begin{align*}
\innerWaitingTime(\arrivalLS,\poolSize) & :=W_{\toPodState}(\arrivalLS,\poolSize)+\routingProb(\toPodState,\toPickOneState)\cdot\left(W_{\toPickOneState}(\arrivalLS,\poolSize)+W_{\pickerQueueOneState}(\arrivalLS,\poolSize)-1/\pickingOne\right)\\
 & \phantomeq+\routingProb(\toPodState,\toPickTwoState)\cdot\left(W_{\toPickTwoState}(\arrivalLS,\poolSize)+W_{\pickerQueueTwoState}(\arrivalLS,\poolSize)-1/\pickingTwo\right).
\end{align*}
We calculate $W_{j}(\arrivalLS,\poolSize)$, $j\in\nodeSet$, with
MVA.

So, the question is: \textquotedblleft How many additional robots
do we need, so that the turnover time of a task is also acceptable?\textquotedblright{}
In the following \prettyref{alg:algorithm-minimal-to}, we determine
the minimum number of robots for an acceptable turnover time of a
task. We will call this time $TO_{task}^{max}$.

\begin{algorithm}[h]
\begin{algorithmic}[1]
\Function{MinimalRobots}{$\overline{N}^{*}$, $TO^{\max}_{\text{task}}$}
\While{$\overline{N}^{*} \neq \{\}$}
	\State $N\gets\min{(\overline{N}^{*})}$
	\State calculate $\lambda_{LC}$ with $\lambda_{\text{eff}}(\lambda_{\text{LC}})=\arrival$
	\If{$TO_{\text{task}}(\lambda_{LC}, N)\leq TO^{\max}_{\text{task}}$}
		\State \Return $N$
	\Else
		\State $\overline{N}^{*} \gets \overline{N}^{*}\setminus \{N\}$
	\EndIf
\EndWhile
\State \Return ``no solution''
\EndFunction
\end{algorithmic}

\caption{\label{alg:algorithm-minimal-to}Calculate the minimal number of robots
for acceptable turnover time of a task}
\end{algorithm}

\subsection{Numerical experiments\label{sec:numerical-experiments}}

We use parameters from \citet[Table 5.3 and Table 5.4]{dynamicPolicies}
in our experiments. We set the number of pods $\robotnumb^{\max}=550$,
arrival rate of tasks\footnote{Note, in \citet[Table 5.3 and Table 5.4]{dynamicPolicies}, the arrival
rates are in {[}order/hour{]}, but because in that paper each order
generates one task, we use {[}task/hour{]} directly.} $468\:\frac{\text{tasks}}{\text{h}}=0.13\:\frac{\text{tasks}}{\text{s}}$,
average travel time at node $\toPodState$ $\transToPod^{-1}=\SI{18.4}{\second}$,
average travel time at node $\toPickOneState$ $\transToPickerOne^{-1}=\SI{34.5}{\second}$,
average travel time at node $\toPickTwoState$ $\transToPickerTwo^{-1}=\SI{34.5}{\second}$,
average pick time of picking station $1$ $\pickingOne^{-1}=\SI{10}{\second}$,
average pick time of picking station $2$ $\pickingTwo^{-1}=\SI{10}{\second}$,
average travel time at node $\toStorageOneState$ $\transToStorageOne^{-1}=\SI{34.5}{\second}$,
average travel time at node $\toStorageTwoState$ $\transToStorageTwo^{-1}=\SI{34.5}{\second}$,
average travel time at node $\toReplTwoState$ $\transToReplTwo^{-1}=\SI{34.5}{\second}$,
average travel time at node $\toReplOneState$ $\transToReplOne^{-1}=\SI{34.5}{\second}$,
average replenishment time (node $\replQueueState$) $\repl^{-1}=\SI{30}{\second}$,
and average travel time at node $\toReplToStorageState$ $\transReplToStorage^{-1}=\SI{34.5}{\second}$.

For our numerical example, we assume that the robots do not interfere
when they move. Hence, our processor-sharing queues are infinite server
queues. That means $\phi_{j}(n_{j})=n_{j}$ for all $j\in\nodeSet\setminus\left\{ \pickerQueueOneState,\pickerQueueTwoState,\replQueueState\right\} $.

We implemented our algorithm in R and used the package\emph{ queueing},
see \citet{Rqueueing}. The minimal number of robots for this system
to be stable is 18. The minimal number of robots with some additional
waiting time requirements depends on this waiting time. In the worst
case scenario -- when we need to try all ($550-18+1$) of the robots
-- our implementation takes on average 83 seconds on a notebook with
an i7-7600U CPU processor, 2.80GHz and 16GB RAM. We plotted some important
parameters of the system in Figures \ref{fig:stable-arrivals} to
\ref{fig:waiting-times}. For better readability, we plotted data
for a limited number of robots; due to the asymptotic behaviour, one
can estimate how these parameters look with more robots.

\prettyref{fig:stable-arrivals} shows maximal arrival rates $\arrival$
for given numbers of robots to keep the system stable. We see asymptotic
behaviour of the rate $\arrival$ for $\robotnumb\rightarrow\infty$
. In particular, after about 40 robots, additional robots do not allow
significantly higher arrival rates.

\prettyref{fig:throughputs} shows the throughputs for each node.
In \prettyref{prop:THL-BO}, we showed that these throughputs do not
depend on the number of robots, and they are pairwise the same for
the original system with backordering and for the adjusted lost-customers
approximation.

The probabilities that the nodes $\pickerQueueOneState$, $\pickerQueueTwoState$
and $\replQueueState$ are idling are 0.35, 0.35, and 0.22. We calculate
them with \prettyref{cor:idle-times-BO}.

\prettyref{fig:adjusted-arrivals} shows the adjusted arrival rate
$\arrivalLS$ for a system with lost customers, so that the effective
arrival rate is $\arrival$. We see an asymptotic behaviour $\lim_{\robotnumb\rightarrow\infty}\arrivalLS\approx\arrival$.
This happens because, in our test system with lost customers, when
there are many robots in the system, the probability of an empty resource
pool is almost $0$, therefore only a few customers are lost. When
only a few customers are lost, we do not need to adjust $\arrivalLS$
much.

\prettyref{fig:waiting-times} shows average waiting times for an
order, after it has arrived into the system and until it is completed
at a picking station. We see that an order spends a lot of time in
a system with only 18 robots, even if this system is stable. We also
see how dramatically the waiting time improves with only one additional
robot.

\begin{figure}[h]
\centering{}\includegraphics{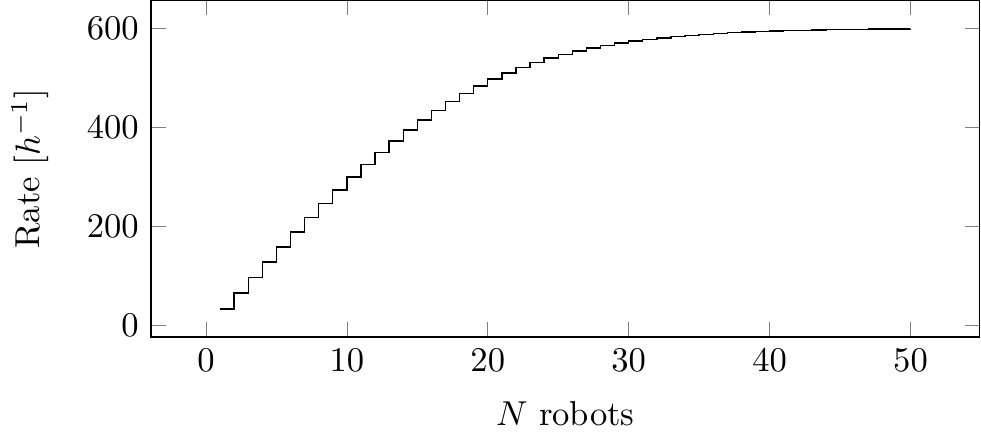}\caption{Maximal arrival rates $\protect\arrival$ for given numbers of robots
to keep the system stable.\label{fig:stable-arrivals}}
\end{figure}

\begin{figure}[h]
\centering{}\includegraphics{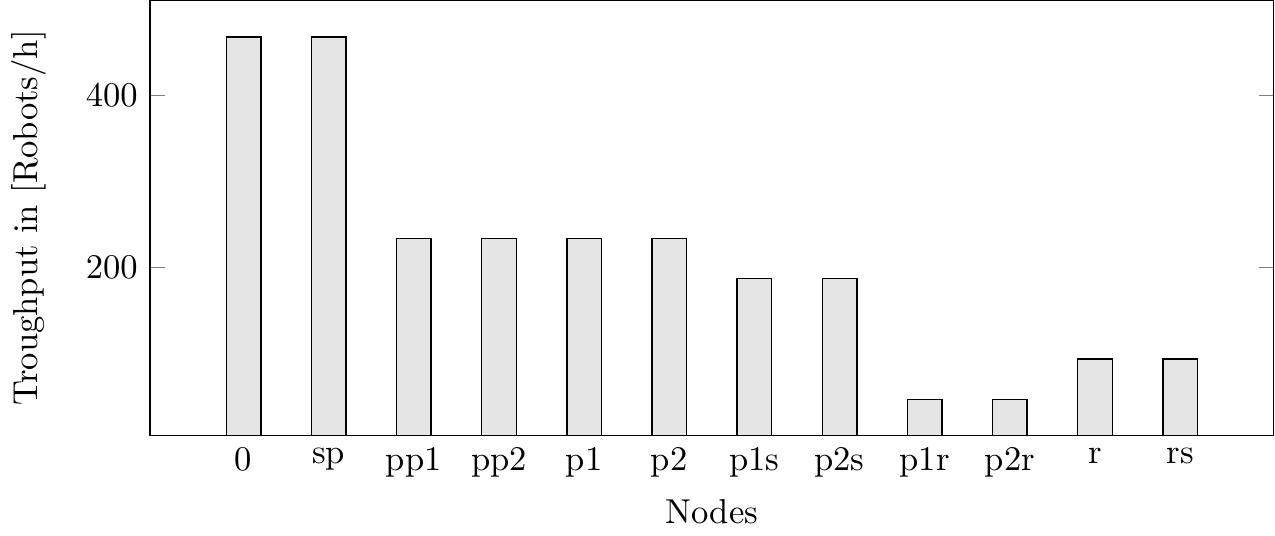}\caption{Throughputs for each node of the RMFS example.\label{fig:throughputs}}
\end{figure}

\begin{figure}[H]
\centering{}\includegraphics{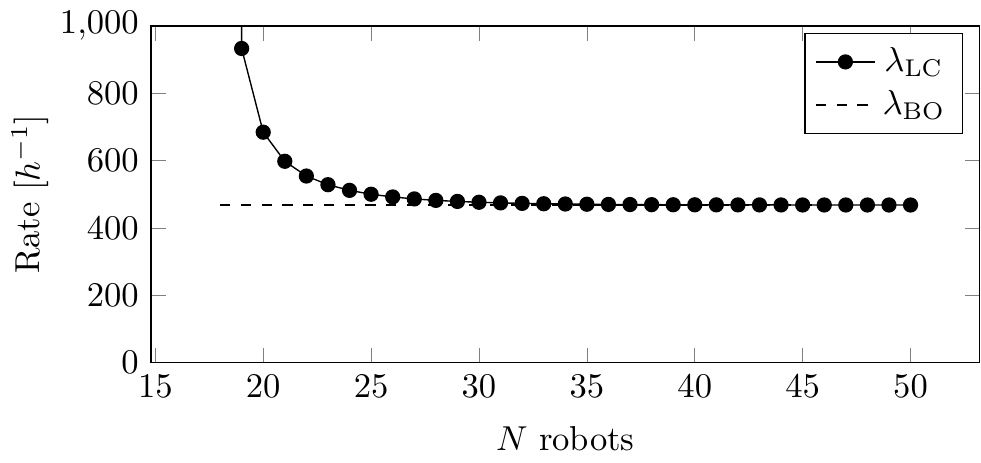}\caption{Adjusted arrival rate $\protect\arrivalLS$ for a system with lost
customers such that the effective arrival rate is $\protect\arrival$.\label{fig:adjusted-arrivals}}
\end{figure}

\begin{figure}[H]
\centering{}\includegraphics{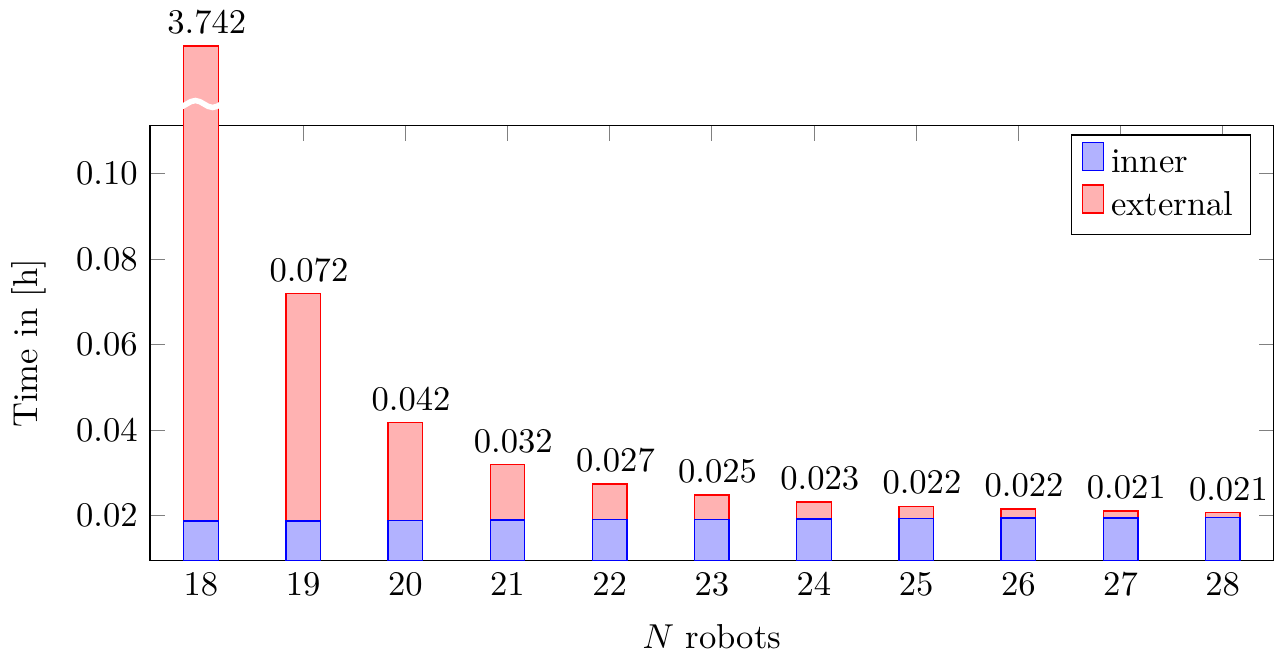}\caption{Turnover times of a task $\protect\TOtask(\protect\arrivalLS,\protect\poolSize)$,
which are the average delay times of a task until it is completed.\label{fig:waiting-times}}
\end{figure}

\paragraph{Simulation}

We simulated the RMFS with backordering for 365 days 20 times for
each number of robots. For simulation we used SimPy 3.0. \prettyref{fig:simulation-external-waiting-times}
shows results for waiting times beginning with 19 robots. The approximation
shows the same qualitative behaviour as the original system with backordering.
In the interesting region 19-25 robots, the approximation is not very
precise but it answers the essential question ``how many robots we
need'{}' quite well. Beginning with 26 robots, the approximation
reflects the asymptotic behaviour of the original system well.

In \prettyref{fig:simulation-external-waiting-times} we have omitted
the results for 18 robots because they have very large mean value
and very large standard deviation. This is because the system with
18 robots operates on the edge of instability. This behaviour of the
system under these settings is interesting from theoretical point
of view, that is why we ran more intensive tests of this system with
200 simulations. The results are in \prettyref{fig:simulation-external-18-histogram}.
They show how different the average waiting time can be. From practical
point of view, we do not recommend to operate a real system under
these conditions. Also in this case, we recommend not to trust simulation
results if they were obtained by only few simulations.

\prettyref{fig:simulation-turn-over-times} shows that our approximation
approximates very well the turnover times. To better judge the quality
of the approximation, we need to consider that the turnover times
consist of transportation times and waiting times for a picker. The
average transportation times are pairwise equal in the original system
and in the approximation. We can easily calculate them from service
times on appropriate nodes without any approximation: $\transToPod^{-1}+\routingProb(\toPodState,\toPickOneState)\cdot\transToPickerOne^{-1}+\routingProb(\toPodState,\toPickTwoState)\cdot\transToPickerOne^{-1}$
$=0.0147\text{h}$. The hard part is to estimate the average waiting
times for the picker. \prettyref{fig:simulation-waiting-for-picker}
shows the results. Our approximation is still good.

Other waiting times, which we cannot calculate directly, are the waiting
times for the replenishment, which are shown in \prettyref{fig:simulation-waiting-for-replenisher}.
We do not need these times for our optimisation problem. However,
it demonstrates how well our algorithm estimates other parts of the
network.

Although we are happy with these results, we remark that this approximation
worked well for our test system but systems with less impressive results
are possible, too.

\begin{figure}[H]
\centering{}\includegraphics{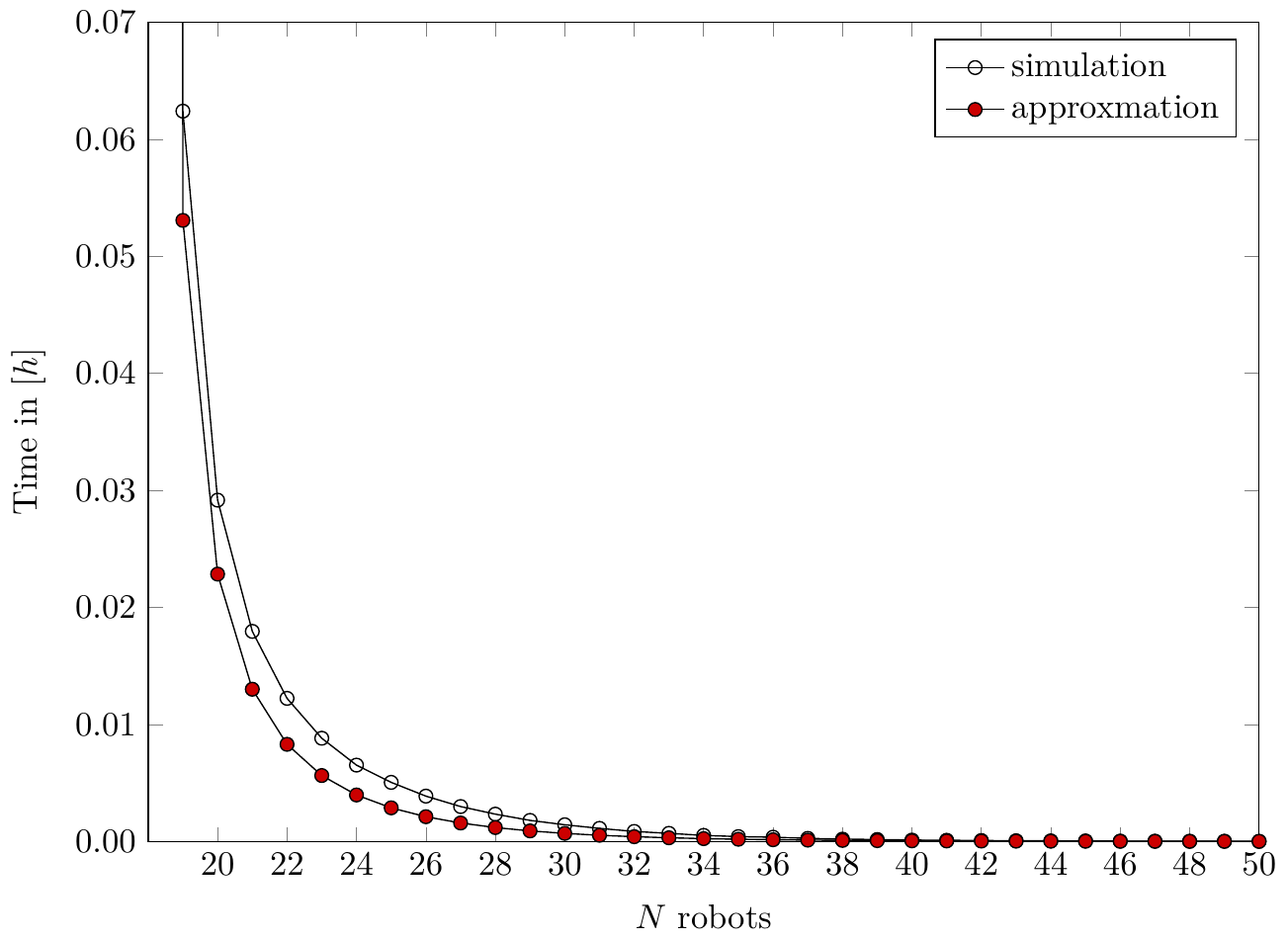}\caption{Average waiting times in the external queue for different number of
robots, simulation vs. approximation. \label{fig:simulation-external-waiting-times}}
\end{figure}

\begin{figure}[H]
\centering{}\includegraphics[width=0.65\textwidth]{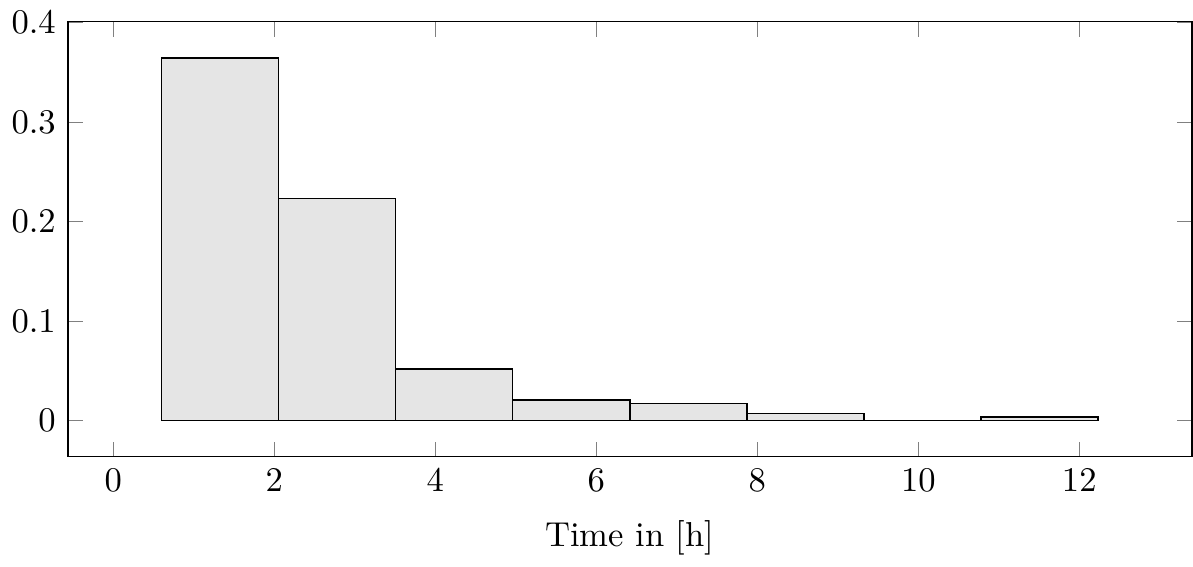}\caption{Distribution of waiting times in the external queue for a system with
18 robots, obtained with 200 simulations. \label{fig:simulation-external-18-histogram}}
\end{figure}

\begin{figure}[H]
\centering{}\includegraphics{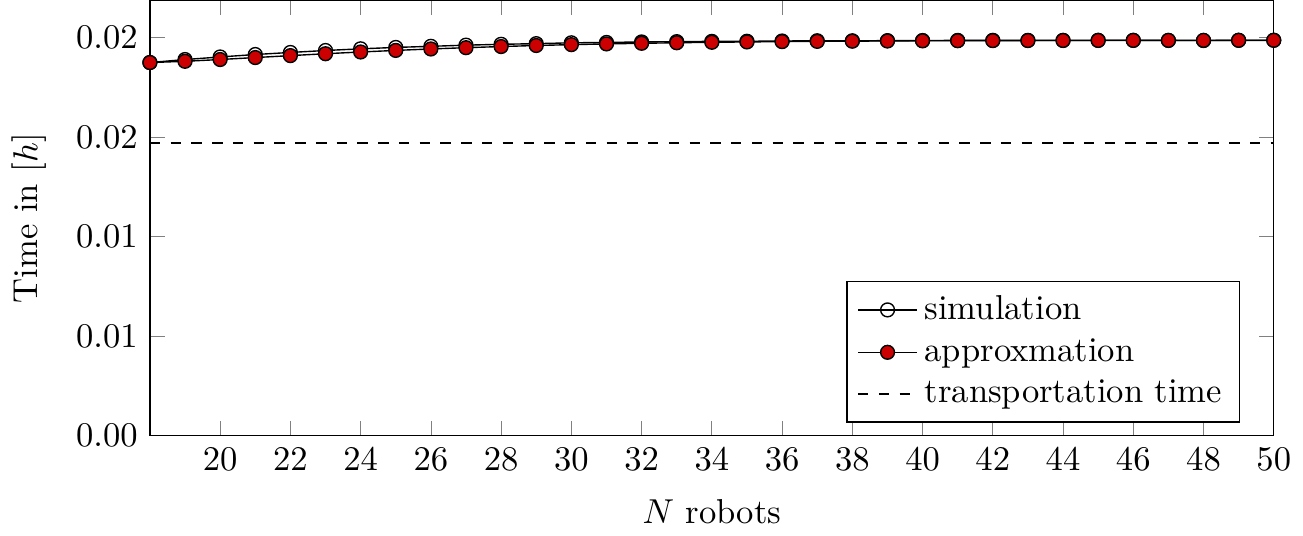}\caption{Average turnover times for different number of robots, simulation
vs. approximation. The graph for approximation is right on the graph
for simulation. The large fraction of turnover times for the transportation
is the same in both systems. It is equal in both systems by construction.
\label{fig:simulation-turn-over-times}}
\end{figure}

\begin{figure}[H]
\centering{}\includegraphics{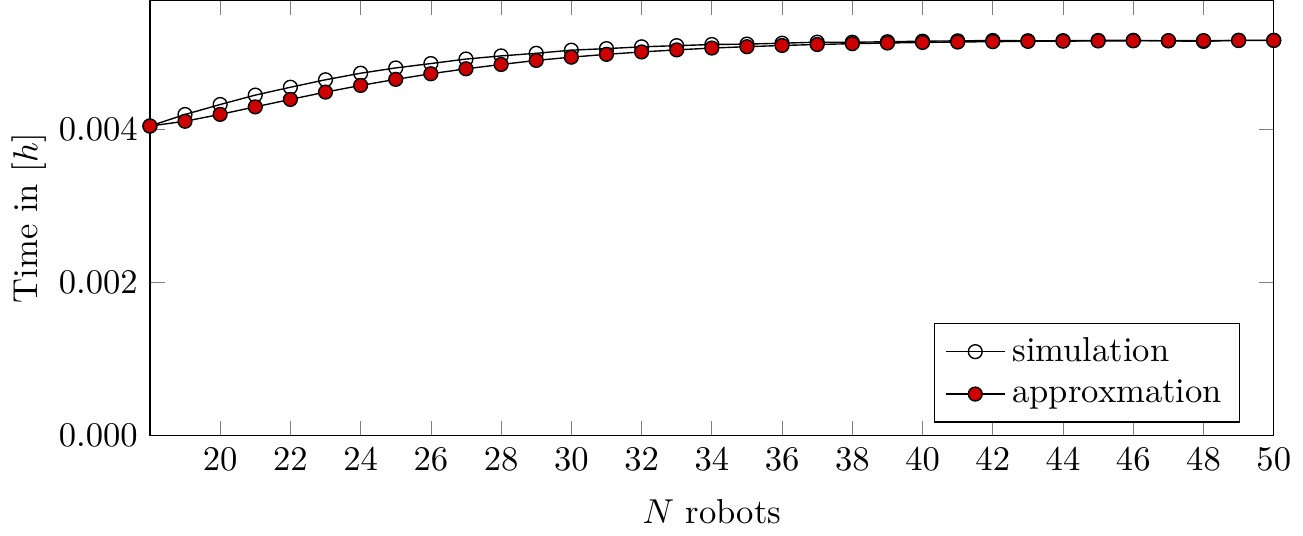}\caption{Waiting times for pickers with different number of robots, simulation
vs. approximation. \label{fig:simulation-waiting-for-picker}}
\end{figure}

\enlargethispage{1cm}
\begin{figure}[H]
\centering{}\includegraphics{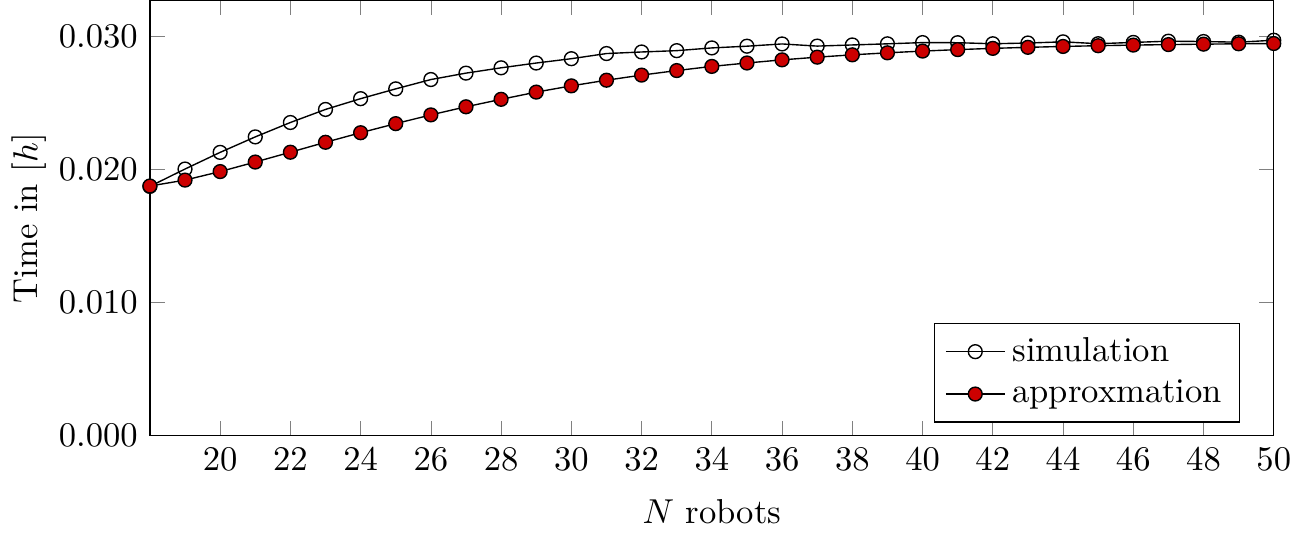}\caption{Waiting times until replenishment begins with different number of
robots, simulation vs. approximation. \label{fig:simulation-waiting-for-replenisher}}
\end{figure}

\section{Conclusion\label{sec:Conclusion}}

In this paper, we have focused on semi-open queueing networks (SOQNs),
which include an external queue for customers and a resource network
which consists of an inner network and a resource pool. We have determined
closed-form expressions for stability, throughputs and idle probabilities
of some nodes. For other performance metrics, we have proposed a new
approximation approach to solve problems of an SOQN with backordering. 

To approximate the resource network of the SOQN with backordering,
in \prettyref{sec:Approximation} we have considered a modification,
where newly arriving customers will decide not to join the external
queue and are lost if the resource pool is empty (``{\LostCustomers}'').
We have proved that we can adjust the arrival rate so that the throughputs
in each node are pairwise identical to those in the original network.
We also have proved that the idle probabilities of the nodes with
constant service rate are pairwise identical.

To approximate the external queue of the SOQN with backordering, in
\prettyref{sec:Approximation-of-the-external-queue} we have used
a two-step approach. In step one, we have constructed a reduced SOQN
with lost customers, where the inner network consists only of one
node, by using Norton's theorem. In step two, we have used the closed-form
solution of this reduced SOQN, to estimate the performance of the
original SOQN with backordering.

Based on the theoretical foundation above, we have modelled a real-world
automatic warehousing system (called robotic mobile fulfilment system,
in short: RMFS) as an SOQN with backordering. We have selected for
our experiment an RMFS with two picking stations and one replenishment
station. Based on the stability analysis in \prettyref{sec:RMFS-general-stat-distr},
we got the minimum number of robots for a stable system. However,
the stability does not say anything about the turnover time of an
order ($=$ waiting time in the external queue + processing time in
the inner network), since customers' orders are expected to be processed
as quickly as possible in the e-commerce sector. Therefore, we have
calculated the processing time in the network based on the approximation
model in \prettyref{sec:Approximation} and the waiting time in the
external queue based on the approximation model in \prettyref{sec:Approximation-of-the-external-queue}.
Due to the short computational time for trying different numbers of
robots, we got all the important metrics within a couple of seconds.

We have plotted the data to see the relationship between the number
of required robots and the average waiting time of a task. Based on
our results, we have seen the dramatic reduction in waiting time in
the external queue with only one more robot, while we have seen the
stagnation of average waiting time by adding more robots. Therefore,
it provides an interesting insight for practitioners and researchers
to make a trade-off between low investment cost and good service for
customers. We made a simulation to analyse the quality of our approximation
method. For our test system, it shows good results.

\section*{Acknowledgments}

Ruslan Krenzler and Sonja Otten are funded by the industrial project
``Robotic Mobile Fulfillment System'', which is financially supported
by Ecopti GmbH (Paderborn, Germany) and Beijing Hanning Tech Co.,
Ltd.~(Beijing, China).

\appendix
\begin{appendices} 
	
\section{Omitted proofs\label{appx:omitted-calculation}}
\begin{proof}
[Proof of \prettyref{prop:unique}]We will show the strict isotonicity
of the effective arrival rate $\effarrival(\arrivalLS)$ for any $\poolSize\geq1$
by induction in $\poolSize$. \textit{\emph{To distinguish systems
with different numbers of}} resources,\textit{\emph{ we will use the
notation $\effarrival^{(\poolSize)}(\arrivalLS)$ for a system with
}}$\poolSize$\textit{\emph{ }}resources\textit{\emph{.}}

We recall the following constants from \prettyref{eq:BO-GEN-normalize}
for $L=1,2,\ldots,\poolSize$:
\[
\normLS(\nodeSetNull,L)=\sum_{n_{0}=0}^{L}\left(\frac{\etaRatio_{0}}{\arrivalLS}\right)^{\nodeQueue_{0}}\sum_{\sum_{j\in\nodeSet}\nodeQueue_{j}=L-\nodeQueue_{0}}\prod_{j=1}^{\nodes}\left(\prod_{i=1}^{\nodeQueue_{j}}\frac{\etaRatio_{j}}{\nodeServiceRate_{j}(i)}\right).
\]
We set $C_{\arrivalLS}(\nodeSetNull,L):=\normLS(\nodeSetNull,L)$
to emphasise that these constants are functions in $\arrivalLS$.

It holds
\begin{align*}
\effarrival^{(\poolSize)}(\arrivalLS) & =\arrivalLS\cdot\left(1-\frac{\normB(\nodeSet,\poolSize)}{\normLS(\nodeSetNull,\poolSize)}\right)\\
 & =\arrivalLS\cdot\left(1-\frac{\sum_{\sum_{j\in\nodeSet}\nodeQueue_{j}=\poolSize}\prod_{j=1}^{\nodes}\left(\prod_{i=1}^{\nodeQueue_{j}}\frac{\etaRatio_{j}}{\nodeServiceRate_{j}(i)}\right)}{\sum_{\sum_{j\in\nodeSetNull}\nodeQueue_{j}=\poolSize}\left(\frac{\etaRatio_{0}}{\arrivalLS}\right)^{\nodeQueue_{0}}\cdot\prod_{j=1}^{\nodes}\left(\prod_{i=1}^{\nodeQueue_{j}}\frac{\etaRatio_{j}}{\nodeServiceRate_{j}(i)}\right)}\right)\\
 & =\etaRatio_{0}\cdot\frac{\sum_{n_{0}=1}^{N}\left(\frac{\etaRatio_{0}}{\arrivalLS}\right)^{\nodeQueue_{0}-1}\sum_{\sum_{j\in\nodeSet}\nodeQueue_{j}=\poolSize-\nodeQueue_{0}}\prod_{j=1}^{\nodes}\left(\prod_{i=1}^{\nodeQueue_{j}}\frac{\etaRatio_{j}}{\nodeServiceRate_{j}(i)}\right)}{\sum_{n_{0}=0}^{N}\left(\frac{\etaRatio_{0}}{\arrivalLS}\right)^{\nodeQueue_{0}}\sum_{\sum_{j\in\nodeSet}\nodeQueue_{j}=\poolSize-\nodeQueue_{0}}\prod_{j=1}^{\nodes}\left(\prod_{i=1}^{\nodeQueue_{j}}\frac{\etaRatio_{j}}{\nodeServiceRate_{j}(i)}\right)}\\
 & =\etaRatio_{0}\cdot\frac{\sum_{n_{0}=0}^{N-1}\left(\frac{\etaRatio_{0}}{\arrivalLS}\right)^{\nodeQueue_{0}}\sum_{\sum_{j\in\nodeSet}\nodeQueue_{j}=\poolSize-1-\nodeQueue_{0}}\prod_{j=1}^{\nodes}\left(\prod_{i=1}^{\nodeQueue_{j}}\frac{\etaRatio_{j}}{\nodeServiceRate_{j}(i)}\right)}{\sum_{n_{0}=0}^{N}\left(\frac{\etaRatio_{0}}{\arrivalLS}\right)^{\nodeQueue_{0}}\sum_{\sum_{j\in\nodeSet}\nodeQueue_{j}=\poolSize-\nodeQueue_{0}}\prod_{j=1}^{\nodes}\left(\prod_{i=1}^{\nodeQueue_{j}}\frac{\etaRatio_{j}}{\nodeServiceRate_{j}(i)}\right)}\\
 & =\etaRatio_{0}\cdot\frac{C_{\arrivalLS}(\nodeSetNull,\poolSize-1)}{C_{\arrivalLS}(\nodeSetNull,\poolSize)}.
\end{align*}

We further define for $L=1,2,\ldots,\poolSize$
\[
C(\nodeSet,L):=\sum_{\sum_{j\in\nodeSet}\nodeQueue_{j}=L}\prod_{j=1}^{\nodes}\left(\prod_{i=1}^{\nodeQueue_{j}}\frac{\etaRatio_{j}}{\nodeServiceRate_{j}(i)}\right)
\]
and obtain
\begin{equation}
C_{\arrivalLS}(\nodeSetNull,L)=C(\nodeSet,L)+\frac{\etaRatio_{0}}{\arrivalLS}\cdot C_{\arrivalLS}(\nodeSetNull,L-1),\quad L=1,2,\ldots,\poolSize.\label{eq:C-gleichung}
\end{equation}
Now we prove by induction
\[
\effarrival^{(\poolSize)}(\arrivalLS+\varepsilon)>\effarrival^{(\poolSize)}(\arrivalLS)\qquad\forall\arrivalLS>0,\ \varepsilon>0,\ \forall\poolSize=1,2,\ldots
\]
\textit{Base step}: For $\poolSize=1$:
\[
\effarrival^{(1)}(\arrivalLS+\varepsilon)-\effarrival^{(1)}(\arrivalLS)=\frac{\etaRatio_{0}}{\frac{\etaRatio_{0}}{\arrivalLS+\varepsilon}+\sum_{j=1}^{\nodes}\frac{\etaRatio_{j}}{\nodeServiceRate_{j}(1)}}-\frac{\etaRatio_{0}}{\frac{\etaRatio_{0}}{\arrivalLS}+\sum_{j=1}^{\nodes}\frac{\etaRatio_{j}}{\nodeServiceRate_{j}(1)}}>0.
\]
\textit{Induction step:} Assume the inequality holds for all $1\leq L\leq\poolSize-1$.
Then for $L=N$ we obtain
\begin{align*}
 & \frac{1}{\etaRatio_{0}}\cdot\left(\effarrival^{(\poolSize)}(\arrivalLS+\varepsilon)-\effarrival^{(\poolSize)}(\arrivalLS)\right)=\frac{C_{\arrivalLS+\varepsilon}(\nodeSetNull,\poolSize-1)}{C_{\arrivalLS+\varepsilon}(\nodeSetNull,\poolSize)}-\frac{C_{\arrivalLS}(\nodeSetNull,\poolSize-1)}{C_{\arrivalLS}(\nodeSetNull,\poolSize)}\\
 & =\frac{C_{\arrivalLS+\varepsilon}(\nodeSetNull,\poolSize-1)\cdot C_{\arrivalLS}(\nodeSetNull,\poolSize)-C_{\arrivalLS}(\nodeSetNull,\poolSize-1)\cdot C_{\arrivalLS+\varepsilon}(\nodeSetNull,\poolSize)}{C_{\arrivalLS+\varepsilon}(\nodeSetNull,\poolSize)\cdot C_{\arrivalLS}(\nodeSetNull,\poolSize)}.
\end{align*}
Because the denominator is strictly positive, it suffices to show
that the numerator is strictly positive. Using the induction assumption
we obtain from \prettyref{eq:C-gleichung} the following:
\begin{align*}
 & C_{\arrivalLS+\varepsilon}(\nodeSetNull,\poolSize-1)\cdot C_{\arrivalLS}(\nodeSetNull,\poolSize)-C_{\arrivalLS}(\nodeSetNull,\poolSize-1)\cdot C_{\arrivalLS+\varepsilon}(\nodeSetNull,\poolSize)\\
 & \stackrel{}{=}\left(C(\nodeSet,\poolSize-1)+\frac{\etaRatio_{0}}{\arrivalLS+\varepsilon}\cdot C_{\arrivalLS+\varepsilon}(\nodeSetNull,\poolSize-2)\right)\\
 & \phantomeq\phantomeq\cdot\left(C(\nodeSet,\poolSize)+\frac{\etaRatio_{0}}{\arrivalLS}\cdot C_{\arrivalLS}(\nodeSetNull,\poolSize-1)\right)\\
 & \phantomeq-\left(C(\nodeSet,\poolSize-1)+\frac{\etaRatio_{0}}{\arrivalLS}\cdot C_{\arrivalLS}(\nodeSetNull,\poolSize-2)\right)\\
 & \phantomeq\phantomeq\cdot\left(C(\nodeSet,\poolSize)+\frac{\etaRatio_{0}}{\arrivalLS+\varepsilon}\cdot C_{\arrivalLS+\varepsilon}(\nodeSetNull,\poolSize-1)\right)\\
 & =C(\nodeSet,\poolSize-1)\cdot C(\nodeSet,\poolSize)+C(\nodeSet,\poolSize-1)\cdot\frac{\etaRatio_{0}}{\arrivalLS}\cdot C_{\arrivalLS}(\nodeSetNull,\poolSize-1)\\
 & \phantomeq\quad+\frac{\etaRatio_{0}}{\arrivalLS+\varepsilon}\cdot C_{\arrivalLS+\varepsilon}(\nodeSetNull,\poolSize-2)\cdot C(\nodeSet,\poolSize)\\
 & \phantomeq\quad+\frac{\etaRatio_{0}}{\arrivalLS+\varepsilon}\cdot C_{\arrivalLS+\varepsilon}(\nodeSetNull,\poolSize-2)\cdot\frac{\etaRatio_{0}}{\arrivalLS}\cdot C_{\arrivalLS}(\nodeSetNull,\poolSize-1)\\
 & \phantomeq-C(\nodeSet,\poolSize-1)\cdot C(\nodeSet,\poolSize)-C(\nodeSet,\poolSize-1)\cdot\frac{\etaRatio_{0}}{\arrivalLS+\varepsilon}\cdot C_{\arrivalLS+\varepsilon}(\nodeSetNull,\poolSize-1)\\
 & \phantomeq\quad-\frac{\etaRatio_{0}}{\arrivalLS}\cdot C_{\arrivalLS}(\nodeSetNull,\poolSize-2)\cdot C(\nodeSet,\poolSize)\\
 & \phantomeq\quad-\frac{\etaRatio_{0}}{\arrivalLS}\cdot C_{\arrivalLS}(\nodeSetNull,\poolSize-2)\cdot\frac{\etaRatio_{0}}{\arrivalLS+\varepsilon}\cdot C_{\arrivalLS+\varepsilon}(\nodeSetNull,\poolSize-1)\\
 & =\frac{\etaRatio_{0}}{\arrivalLS+\varepsilon}\cdot\frac{\etaRatio_{0}}{\arrivalLS}\cdot\Big[C_{\arrivalLS+\varepsilon}(\nodeSetNull,\poolSize-2)\cdot C_{\arrivalLS}(\nodeSetNull,\poolSize-1)\\
 & \phantomeq\underbrace{\qquad\qquad\qquad\qquad\qquad\qquad\qquad\qquad\quad-C_{\arrivalLS}(\nodeSetNull,\poolSize-2)\cdot C_{\arrivalLS+\varepsilon}(\nodeSetNull,\poolSize-1)\Big]}_{=:A_{\poolSize}}\\
 & \phantomeq+\frac{\etaRatio_{0}}{\arrivalLS}\cdot C(\nodeSet,\poolSize-1)\cdot C_{\arrivalLS}(\nodeSetNull,\poolSize-1)+\frac{\etaRatio_{0}}{\arrivalLS+\varepsilon}\cdot C_{\arrivalLS+\varepsilon}(\nodeSetNull,\poolSize-2)\cdot C(\nodeSet,\poolSize)\\
 & \phantomeq-\frac{\etaRatio_{0}}{\arrivalLS+\varepsilon}\cdot C(\nodeSet,\poolSize-1)\cdot C_{\arrivalLS+\varepsilon}(\nodeSetNull,\poolSize-1)-\frac{\etaRatio_{0}}{\arrivalLS}\cdot C_{\arrivalLS}(\nodeSetNull,\poolSize-2)\cdot C(\nodeSet,\poolSize)\\
 & =A_{\poolSize}+\frac{\etaRatio_{0}}{\arrivalLS}\cdot\left[C(\nodeSet,\poolSize-1)\cdot C_{\arrivalLS}(\nodeSetNull,\poolSize-1)-C_{\arrivalLS}(\nodeSetNull,\poolSize-2)\cdot C(\nodeSet,\poolSize)\right]\\
 & \phantomeq+\frac{\etaRatio_{0}}{\arrivalLS+\varepsilon}\cdot\left[C_{\arrivalLS+\varepsilon}(\nodeSetNull,\poolSize-2)\cdot C(\nodeSet,\poolSize)-C(\nodeSet,\poolSize-1)\cdot C_{\arrivalLS+\varepsilon}(\nodeSetNull,\poolSize-1)\right].
\end{align*}
$A_{\poolSize}$ is a positive factor $\frac{\etaRatio_{0}}{\arrivalLS+\varepsilon}\cdot\frac{\etaRatio_{0}}{\arrivalLS}$
multiplied by the numerator of $\effarrival^{(\poolSize-1)}(\arrivalLS+\varepsilon)-\effarrival^{(\poolSize-1)}(\arrivalLS)$.
Therefore, $A_{\poolSize}>0$ by the induction assumption. Thus, it
suffices to prove
\begin{align*}
 & \frac{\etaRatio_{0}}{\arrivalLS}\cdot\left[C(\nodeSet,\poolSize-1)\cdot C_{\arrivalLS}(\nodeSetNull,\poolSize-1)-C_{\arrivalLS}(\nodeSetNull,\poolSize-2)\cdot C(\nodeSet,\poolSize)\right]\\
 & \phantomeq+\frac{\etaRatio_{0}}{\arrivalLS+\varepsilon}\cdot\left[C_{\arrivalLS+\varepsilon}(\nodeSetNull,\poolSize-2)\cdot C(\nodeSet,\poolSize)-C(\nodeSet,\poolSize-1)\cdot C_{\arrivalLS+\varepsilon}(\nodeSetNull,\poolSize-1)\right]\\
 & \geq0,
\end{align*}
which is equivalent to
\begin{align*}
 & \frac{\etaRatio_{0}}{\arrivalLS}\cdot\left[\frac{C(\nodeSet,\poolSize-1)}{C(\nodeSet,\poolSize)}-\frac{C_{\arrivalLS}(\nodeSetNull,\poolSize-2)}{C_{\arrivalLS}(\nodeSetNull,\poolSize-1)}\right]\cdot C(\nodeSet,\poolSize)\cdot C_{\arrivalLS}(\nodeSetNull,\poolSize-1)\\
 & -\frac{\etaRatio_{0}}{\arrivalLS+\varepsilon}\cdot\left[\frac{C(\nodeSet,\poolSize-1)}{C(\nodeSet,\poolSize)}-\frac{C_{\arrivalLS+\varepsilon}(\nodeSetNull,\poolSize-2)}{C_{\arrivalLS+\varepsilon}(\nodeSetNull,\poolSize-1)}\right]\cdot C(\nodeSet,\poolSize)\cdot C_{\arrivalLS+\varepsilon}(\nodeSetNull,\poolSize-1)\\
 & =\Bigg\{\underbrace{\left[\frac{C(\nodeSet,\poolSize-1)}{C(\nodeSet,\poolSize)}-\frac{C_{\arrivalLS}(\nodeSetNull,\poolSize-2)}{C_{\arrivalLS}(\nodeSetNull,\poolSize-1)}\right]}_{=:B}\cdot\underbrace{\frac{\etaRatio_{0}}{\arrivalLS}\cdot C_{\arrivalLS}(\nodeSetNull,\poolSize-1)}_{=:D}\\
 & \phantomeq-\underbrace{\left[\frac{C(\nodeSet,\poolSize-1)}{C(\nodeSet,\poolSize)}-\frac{C_{\arrivalLS+\varepsilon}(\nodeSetNull,\poolSize-2)}{C_{\arrivalLS+\varepsilon}(\nodeSetNull,\poolSize-1)}\right]}_{=:C}\cdot\underbrace{\frac{\etaRatio_{0}}{\arrivalLS+\varepsilon}\cdot C_{\arrivalLS+\varepsilon}(\nodeSetNull,\poolSize-1)}_{=:E}\Bigg\}\cdot C(\nodeSet,\poolSize)\\
 & \geq0.
\end{align*}

(i) We show that $B\geq0$, $C\geq0$.

Van der Wal \Citep{vanderWal1989} shows that an increasing population
size increases throughput, hence 
\[
\frac{C(\nodeSet,\poolSize-1)}{C(\nodeSet,\poolSize)}\geq\frac{C(\nodeSet,\poolSize-2)}{C(\nodeSet,\poolSize-1)}.
\]

Furthermore, we show
\[
\frac{C(\nodeSet,\poolSize-2)}{C(\nodeSet,\poolSize-1)}\geq\frac{C_{\arrivalLS}(\nodeSetNull,\poolSize-2)}{C_{\arrivalLS}(\nodeSetNull,\poolSize-1)}.
\]
We consider the right-hand side as throughput of a cyclic Gordon-Newell
network\footnote{$\frac{C(\nodeSet,\poolSize-2)}{C(\nodeSet,\poolSize-1)}$ and $\frac{C_{\arrivalLS}(\nodeSet,\poolSize-2)}{C_{\arrivalLS}(\nodeSet,\poolSize-1)}$
as given in our derivations are the (average) throughput of a cycle
following \citet[Definition 2.6]{Daduna2008}.} with node set $\nodeSetNull:=\left\{ 0,1,\ldots,\nodes\right\} $,
service rates $\mu_{j}(n):=\frac{\nodeServiceRate_{j}(i)}{\etaRatio_{j}}$,
$j=1,\ldots,\nodes$, $n=0,1,\ldots,\poolSize$ and $\mu_{0}(n):=\frac{\arrivalLS}{\etaRatio_{0}}$
and solution $\underbrace{(1,1,\ldots,1)}_{J+1\text{-times}}$ of
the routing matrix for the cycle.

The left-hand side is the throughput of a cyclic Gordon-Newell network
which is obtained from the first cycle by deleting node $0$ and skipping
the gap by the cycling customers.

Lemma 2.8 in \citet{Daduna2008} states that deleting any node of
a cycle with skipping the gap increases the throughput. Consequently,
$B\geq0$ in a two-step conclusion, and similarly $C\geq0$.

(ii) From definition of $D$ and $E$ follows by direct comparison
$D>E$.

(iii) The proof will be finished if we can show $B\geq C$. To do
so, it is sufficient to prove with \citet{SHANTHIKUMAR1986259}
\begin{equation}
\frac{C_{\arrivalLS}(\nodeSetNull,\poolSize-2)}{C_{\arrivalLS}(\nodeSetNull,\poolSize-1)}\leq\frac{C_{\arrivalLS+\varepsilon}(\nodeSetNull,\poolSize-2)}{C_{\arrivalLS+\varepsilon}(\nodeSetNull,\poolSize-1)}.\label{eq:Daduna-2}
\end{equation}
The left-hand side is the throughput of the Gordon-Newell network
according to Definition (2.2) in \citet{SHANTHIKUMAR1986259}.

The right-hand side is the throughput of the Gordon-Newell network
with service rate at node $0$ increased to $\arrivalLS+\varepsilon$.

Corollary 3.1(i) in \citet{SHANTHIKUMAR1986259} states
\[
\etaRatio_{0}\cdot\frac{C_{\arrivalLS}(\nodeSetNull,\poolSize-2)}{C_{\arrivalLS}(\nodeSetNull,\poolSize-1)}\leq\etaRatio_{0}\cdot\frac{C_{\arrivalLS+\varepsilon}(\nodeSetNull,\poolSize-2)}{C_{\arrivalLS+\varepsilon}(\nodeSetNull,\poolSize-1)},
\]
which verifies \prettyref{eq:Daduna-2}.
\end{proof}
\begin{rem}
\label{rem:lambda-eff-N-1-N-2} For $\arrival\in\left(0,\maxArrival\right)=\left(0,\etaRatio_{0}\cdot\tfrac{\abrevC 1}{\abrevC 0}\right)$
it holds
\[
\arrivalLS=\begin{cases}
\frac{\etaRatio_{0}\cdot\arrival}{\etaRatio_{0}-\arrival\cdot\abrevC 0}, & \poolSize=1,\\
-\frac{\etaRatio_{0}}{2\cdot\big(\etaRatio_{0}\cdot\abrevC 1-\arrival\cdot\abrevC 0\big)}\left(\etaRatio_{0}-\arrival\cdot\abrevC 1-\sqrt{\big(\etaRatio_{0}+\arrival\cdot\abrevC 1\big){}^{2}-4\cdot\arrival^{2}\cdot\abrevC 0}\right), & \poolSize=2.
\end{cases}
\]
\end{rem}

\begin{proof}
Due to \prettyref{thm:theorem-ls} we have
\[
\effarrival(\arrivalLS)=\arrivalLS\cdot\left(1-\frac{\abrevC 0}{\sum_{n=0}^{\poolSize}\left(\frac{\arrivalLS}{\etaRatio_{0}}\right)^{n}\cdot\abrevC n}\right)=\arrivalLS\cdot\left(\frac{\sum_{n=1}^{\poolSize}\abrevC n\cdot\arrivalLS\cdot\left(\frac{\arrivalLS}{\etaRatio_{0}}\right)^{\poolSize-n}}{\sum_{n=0}^{\poolSize}\abrevC n\cdot\left(\frac{\arrivalLS}{\etaRatio_{0}}\right)^{\poolSize-n}}\right)
\]
for $\arrivalLS\in(0,\infty)$ and
\[
\abrevC{\poolSize}=1,\quad\abrevC{\poolSize-1}=\sum_{j=1}^{\nodes}\frac{\etaRatio_{j}}{\nodeServiceRate_{j}(1)}
\]
and
\[
\abrevC{\poolSize-2}=\sum_{j=1}^{\nodes}\frac{\etaRatio_{j}}{\nodeServiceRate_{j}(1)\cdot\nodeServiceRate_{j}(2)}+\sum_{j=1}^{\nodes-1}\sum_{k=j+1}^{\nodes}\frac{\etaRatio_{j}\cdot\etaRatio_{k}}{\nodeServiceRate_{j}(1)\cdot\nodeServiceRate_{k}(1)}.
\]
Let $\arrival\in\left(0,\etaRatio_{0}\cdot\tfrac{\abrevC 1}{\abrevC 0}\right)$.
First, we note that
\begin{equation}
\etaRatio_{0}\cdot\abrevC 1-\arrival\cdot\abrevC 0>\etaRatio_{0}\cdot\abrevC 1-\etaRatio_{0}\cdot\frac{\abrevC 1}{\abrevC 0}\cdot\abrevC 0=0.\label{eq:uniqueness_1}
\end{equation}
The equation $\effarrival(\arrivalLS)=\arrival$ is equivalent to
\[
\frac{\sum_{n=1}^{\poolSize}\left(\frac{\arrivalLS}{\etaRatio_{0}}\right)^{\poolSize+1-n}\cdot\etaRatio_{0}\cdot\abrevC n}{\sum_{n=0}^{N}\left(\frac{\arrivalLS}{\etaRatio_{0}}\right)^{\poolSize-n}\cdot\abrevC n}=\arrival
\]
and this to
\begin{equation}
\sum_{n=1}^{\poolSize}\big(\etaRatio_{0}\cdot\abrevC n-\arrival\cdot\abrevC{n-1}\big)\cdot\left(\frac{\arrivalLS}{\etaRatio_{0}}\right)^{\poolSize+1-n}-\arrival\cdot\abrevC{\poolSize}=0.\label{eq:uniqueness_2}
\end{equation}
If $\poolSize=1$, then \eqref{eq:uniqueness_2} is equivalent to
\[
\arrivalLS=\frac{\etaRatio_{0}\cdot\arrival\cdot\abrevC 1}{\etaRatio_{0}\cdot\abrevC 1-\arrival\cdot\abrevC 0}=\frac{\etaRatio_{0}\cdot\arrival}{\etaRatio_{0}-\arrival\cdot\abrevC 0}.
\]
If $\poolSize=2$, then \eqref{eq:uniqueness_2} is equivalent to
\[
\big(\etaRatio_{0}\cdot\abrevC 1-\arrival\cdot\abrevC 0\big)\cdot\left(\frac{\arrivalLS}{\etaRatio_{0}}\right)^{2}+\big(\etaRatio_{0}\cdot\abrevC 2-\arrival\cdot\abrevC 1\big)\cdot\left(\frac{\arrivalLS}{\etaRatio_{0}}\right)-\abrevC 2\cdot\arrival=0.
\]
Since the discriminant fulfils
\[
\big(\etaRatio_{0}\cdot\abrevC 2-\arrival\cdot\abrevC 1\big)^{2}+4\cdot\big(\etaRatio_{0}\cdot\abrevC 1-\arrival\cdot\abrevC 0\big)\cdot\abrevC 2\cdot\arrival>0
\]
by \eqref{eq:uniqueness_1}, it follows that
\begin{align*}
\frac{\arrivalLS}{\etaRatio_{0}} & =-\frac{1}{2\cdot\big(\etaRatio_{0}\cdot\abrevC 1-\arrival\cdot\abrevC 0\big)}\\
 & \phantomeq\cdot\Bigl(\etaRatio_{0}\cdot\abrevC 2-\arrival\cdot\abrevC 1\pm\sqrt{\big(\etaRatio_{0}\cdot\abrevC 2-\arrival\cdot\abrevC 1\big)^{2}+4\cdot\big(\etaRatio_{0}\cdot\abrevC 1-\arrival\cdot\abrevC 0\big)\cdot\abrevC 2\cdot\arrival}\Bigr).
\end{align*}
Due to \eqref{eq:uniqueness_1}, the solution $\arrivalLS$ is positive
only if
\[
\etaRatio_{0}\cdot\abrevC 2-\arrival\cdot\abrevC 1\pm\sqrt{\big(\etaRatio_{0}\cdot\abrevC 2-\arrival\cdot\abrevC 1\big)^{2}+4\cdot\big(\etaRatio_{0}\cdot\abrevC 1-\arrival\cdot\abrevC 0\big)\cdot\abrevC 2\cdot\arrival}<0.
\]
Let $c:=\etaRatio_{0}\cdot\abrevC 2-\arrival\cdot\abrevC 1$ and $d:=4\cdot\big(\etaRatio_{0}\cdot\abrevC 1-\arrival\cdot\abrevC 0\big)\cdot\abrevC 2\cdot\arrival$
which implies $d>0$. If $c\geq0$, then obviously $c+\sqrt{c^{2}+d}>0$.
If $c<0$, then
\[
c+\sqrt{c^{2}+d}>c+\sqrt{c^{2}}=c+|c|=c-c=0.
\]
Hence, the only positive solution (existence guaranteed by \prettyref{thm:theorem-ls})
has to be
\begin{align*}
\arrivalLS & =-\frac{\etaRatio_{0}}{2\cdot\big(\etaRatio_{0}\cdot\abrevC 1-\arrival\cdot\abrevC 0\big)}\\
 & \phantomeq\cdot\Bigl(\etaRatio_{0}\cdot\abrevC 2-\arrival\cdot\abrevC 1-\sqrt{\big(\etaRatio_{0}\cdot\abrevC 2-\arrival\cdot\abrevC 1\big)^{2}+4\cdot\big(\etaRatio_{0}\cdot\abrevC 1-\arrival\cdot\abrevC 0\big)\cdot\abrevC 2\cdot\arrival}\Bigr)\\
 & =-\frac{\etaRatio_{0}}{2\cdot\big(\etaRatio_{0}\cdot\abrevC 1-\arrival\cdot\abrevC 0\big)}\cdot\left(\etaRatio_{0}-\arrival\cdot\abrevC 1-\sqrt{\big(\etaRatio_{0}+\arrival\cdot\abrevC 1\big)^{2}-4\cdot\arrival^{2}\cdot\abrevC 0}\right).
\end{align*}
\end{proof}

\end{appendices}

\newpage{}

\bibliographystyle{plainnat}
\bibliography{main}

\end{document}